\documentclass[a4paper,11pt,reqno]{amsart}
\usepackage[utf8]{inputenc}
\usepackage[T1]{fontenc}
\usepackage{lmodern}
\usepackage[english]{babel}
\usepackage{amsmath,a4wide}
\usepackage{esint}
\usepackage{stmaryrd,mathrsfs,bm,amsthm,mathtools,yfonts,amssymb,color,xfrac,braket,booktabs,graphicx,graphics,amsfonts}
\usepackage{times,latexsym,microtype,indentfirst}
\usepackage{hyperref}
\usepackage{xcolor}
\usepackage{courier}

\begingroup
\newtheorem{theorem}{Theorem}[section]
\newtheorem{lemma}[theorem]{Lemma}
\newtheorem{proposition}[theorem]{Proposition}
\newtheorem{corollary}[theorem]{Corollary}
\endgroup

\theoremstyle{definition}
\newtheorem{definition}[theorem]{Definition}
\newtheorem{remark}[theorem]{Remark}
\newtheorem{ipotesi}[theorem]{Assumption}
\newtheorem{notazioni}[theorem]{Notation}

\numberwithin{equation}{section}
\numberwithin{subsection}{section}

\newcommand{\Z}{\mathbb{Z}}
\newcommand{\R}{\mathbb{R}}
\newcommand{\C}{\mathbb{C}}

\newcommand{\Sf}{\mathbb{S}}

\newcommand{\N}{\mathcal{N}}

\newcommand{\A}{\mathcal{A}}
\newcommand{\G}{\mathcal{G}}
\newcommand{\Dir}{\mathrm{Dir}}
\newcommand{\E}{\mathscr{E}}

\newcommand{\cS}{\mathcal{S}}

\newcommand{\Ha}{\mathcal{H}}

\newcommand{\e}{\varepsilon}
\newcommand{\spt}{\mathrm{spt}}
\newcommand{\dist}{\mathrm{dist}}
\newcommand{\Lip}{\mathrm{Lip}}

\def\XXint#1#2#3{{\setbox0=\hbox{$#1{#2#3}{\int}$ }
\vcenter{\hbox{$#2#3$ }}\kern-.6\wd0}}

\newcommand{\mres}{\mathbin{\vrule height 1.6ex depth 0pt width 
0.13ex\vrule height 0.13ex depth 0pt width 1.3ex}}

\DeclareMathOperator{\reg}{reg}
\DeclareMathOperator{\sing}{sing}

\DeclareMathOperator*{\esssup}{ess\,sup}

\def\I#1{{\mathcal{A}}_{#1}}

\newcommand{\Iq}{{\mathcal{A}}_Q}
\def\a#1{\left\llbracket{#1}\right\rrbracket}

\newcommand{\abs}[1]{\lvert#1\rvert}
\newcommand{\Abs}[1]{\left\lvert#1\right\rvert}
\newcommand{\norm}[1]{\left\lVert#1\right\rVert}

\newcommand{\ps}[2]{\left\langle#1,#2\right\rangle}
\newcommand{\ton}[1]{\left(#1\right)}
\newcommand{\qua}[1]{\left[#1\right]}
\newcommand{\cur}[1]{\left\{#1\right\}}

\newcommand{\Vol}{\text{Vol}}
\newcommand{\B}[2]{B_{#1}\ton{#2}}

\newcommand{\cD}{\mathcal{D}}
\newcommand{\cF}{\mathcal{F}}
\newcommand{\cH}{\mathcal{H}}
\newcommand{\Na}{\mathbb{N}}
\renewcommand{\P}{\mathcal{P}}
\newcommand{\Pinch}{\mathcal{W}}
\newcommand{\cB}{\mathcal{B}}
\newcommand{\cG}{\mathcal{G}}
\newcommand{\cR}{\mathcal{R}}
\newcommand{\cW}{F}

\renewcommand{\epsilon}{\e}

\title[The singular set of multiple valued energy minimizing harmonic maps]{Rectifiability of the singular set of multiple valued energy minimizing harmonic maps}

\author{Jonas Hirsch \and Salvatore Stuvard \and Daniele Valtorta}

\newcommand{\Addresses}{{
  \bigskip
  \footnotesize

  J. H., \textsc{Scuola Internazionale Superiore di Studi Avanzati, via Bonomea, 265, 34136 Trieste, \textsc{Italy}}
  \par\nopagebreak

\textit{E-mail address}, J. H.: \texttt{\href{jonas.hirsch@sissa.it}{jonas.hirsch@sissa.it}}

\bigskip

  S. S. \and D. V., \textsc{Universit\"at Z\"urich, Winterthurerstrasse 190, CH-8057 Z\"urich, \textsc{Switzerland}}
  \par\nopagebreak
  
  \textit{E-mail address}, S. S.: \texttt{\href{salvatore.stuvard@math.uzh.ch}{salvatore.stuvard@math.uzh.ch}}

  \medskip

  \textit{E-mail address}, D. V.: \texttt{\href{daniele.valtorta@math.uzh.ch}{daniele.valtorta@math.uzh.ch}}
   
}}

\begin{document}
 
\begin{abstract}

In this paper we study the singular set of Dirichlet-minimizing $Q$-valued maps from $\R^m$ into a smooth compact manifold $\N$ without boundary. Similarly to what happens in the case of single valued minimizing harmonic maps, we show that this set is always $(m-3)$-rectifiable with uniform Minkowski bounds. Moreover, as opposed to the single valued case, we prove that the target $\N$ being non-positively curved but not simply connected does not imply continuity of the map.

\vspace{4pt}

\vspace{4pt}
\end{abstract}

\maketitle

\section{Introduction}

Multiple-valued harmonic functions ($\Dir$-minimizers) were originally introduced by Almgren in \cite{almgren_big} as first order approximations for the branching singularities of minimal surfaces in codimension higher than one. Roughly speaking, a $Q$-valued Dirichlet minimizer is a function which attains $Q$ different values (counted with multiplicity) for each point in the domain, and minimizes a suitably defined Dirichlet energy with respect to boundary data.

Even though at first sight it might seem that $Q$-valued functions are easy generalizations of classical (single valued) functions, there are some crucial differences. For instance, the space of such functions is not linear, in the sense that the sum of two $Q$-valued function is not a well-defined notion. These differences make the study of such objects both more complicated than their classical counterpart and more interesting. For a recent survey on results of this kind, we direct the reader to \cite{DLS11a}, where the authors revisit Almgren's original regularity theory of $\Dir$-minimizing $Q$-valued functions suggesting a more intrinsic approach which has its roots in the new techniques developed in the last two decades to perform analysis on metric spaces.

Several generalizations of the original $Q$-valued $\Dir$-minimizing functions have been studied in literature, both in the direction of analyzing multiple-valued functions taking values in more general target spaces than $\R^n$ and in the sense of more general functionals to minimize. Here we limit ourselves to mentioning some of these works. In the direction of functionals more general than the Dirichlet energy there are the works \cite{Mattila,J6}, as well as the recent work \cite{SS17b} by one of the authors, where a complete multi-valued theory for the stability operator is studied. The papers \cite{J7,J11,BDPW,J3,SS17a} focus instead on more general target spaces.

The work \cite{Hir16} of one of the authors started analyzing $Q$-valued harmonic maps into compact Riemannian manifolds, introducing the appropriate definitions and developing the basic continuity theory for such objects. In particular, using a suitably modified version of Federer-Almgren's dimension-reduction argument, \cite{Hir16} proves that 
\begin{theorem}[{\cite[Theorem 0.1]{Hir16} }]
Given a smooth compact Riemannian manifold $\N$ and a $Q$-valued map $u:\Omega\subset \R^m \to \A_{Q}(\N)$ locally minimizing the Dirichlet energy, the singular set of $u$ defined as
\begin{gather}
 \sing_{H}(u):=\cur{x\in \Omega \ \ s.t. \ \ u \ \text{is not continuous in a neighborhood of} \ x}
\end{gather}
is a closed set of Hausdorff dimension at most $m-3$. Moreover, outside this set the map $u$ is locally $C^{0,\alpha}$ continuous, with $\alpha=\alpha(m,Q)>0$. 
\end{theorem}

Note that this is the counterpart of the classical Schoen-Uhlenbeck results in \cite{SU} for the singularities of (single valued) harmonic maps between manifolds. The first goal of this work is to improve this result and give $(m-3)$-rectifiability for the singular set $\sing_{H}(u)$ along with uniform $(m-3)$ Minkowski bounds. In particular, we want to show that
\begin{theorem}\label{th_appetizer}
 Given a Dirichlet-minimizing $Q$-valued map $u:\B 2 0 \subseteq \R^m \to \A_{Q}(\N)$ with energy bounded by $\Lambda$, if $B_{r}\left( \sing_{H}(u) \right) := \bigcup_{x \in \sing_{H}(u)} B_{r}(x)$ then we have
 \begin{gather}
  \Vol\ton{\B r {\sing_{H}(u)\cap \B 1 0}}\leq C r^{3}\, ,
 \end{gather}
 where $C = C(m,\N,Q,\Lambda)$. Moreover, $\sing_{H}(u)$ is $(m-3)$-rectifiable.
\end{theorem}
In order to prove this result, we are going to apply the techniques developed in \cite{naber-valtorta:harmonic}, which roughly speaking rely on a quantitative version of the dimension-reduction argument. However, here we will present an alternative definition of the quantitative stratification used in \cite{naber-valtorta:harmonic}, which, for minimizing maps, turns out to be equivalent to the original one introduced in \cite{ChNa1,ChNa2}, but easier to handle. 

The quantitative stratification is based on the analysis of symmetries and approximate symmetries of the map $u$ at different points and scales, and roughly speaking the quantitative stratum $\cS^k_{\epsilon,r}$ is the set of points $x$ for which $u$ on $\B {r}{x}$ is $\epsilon$-far away from being homogeneous and invariant with respect to a $k$-dimensional subspace. While the notion of closeness employed by \cite{ChNa1,ChNa2} relies on the $L^2$ distance of the map $u$ from some homogeneous and $k$-symmetric model map $h$, we propose a notion that focuses on the $L^2$ norm of the \textit{gradient} of $u$ restricted to arbitrary $k$-subspaces. With this notion, we obtain a slightly better control over the different strata. This stratification is introduced in detail in section \ref{sec_strat}.

\vspace{5mm}
In section \ref{sec_nonpositive}, we also consider the special case of non-positive sectional curvatures in the target $\N$. For classical harmonic maps, this assumption implies full-blown continuity of the map $u$ everywhere. On the other hand, in the case of $Q$-valued map this is true only if $\N$ is assumed to be also \textit{simply connected}. We will provide a counterexample to show that this assumption is needed. This example is based on the fact that the graph of a $Q$-valued map can have a different topology from the one of its domain, and it shows once more that the properties of $Q$-valued maps can be very different from their single valued counterparts.

\vspace{5mm}
The plan of the paper is the following: first, we introduce $Q$-valued Dirichlet minimizers and quickly review the standard properties of these maps. In particular, we study the normalized energy $\E(x,B_{r}(x)):= r^{2-n}\int_{\B r x} \abs{D u}^2$ and its mollified version, which, although morally similar to the classical one, will prove itself to be more useful in quantitative estimates.

We then move on to the study of different versions of the $\epsilon$-regularity theorem for $Q$-valued maps. Soon after, we prove the main estimates on the singular set of $u$ and its stratification. This result relies on a sharp version of Reifenberg's theorem, which we quote from literature. Finally, we close the paper with the analysis of the case of non-positively curved target manifolds.

\section*{Acknowledgments}
The research of J. H. has been supported by the MIUR SIR-grant ``Geometric Variational Problems'', ID RBSI14RVEZ. S. S. was supported by the ERC-grant RAM ``Regularity of Area Minimizing currents'', ID 306246. D. V. has been supported by the SNSF grant PZ00P2\_168006. 

\section{Notation and Preliminaries}
Throughout the whole paper, we will denote by $\Omega$ an open subset of Euclidean space $\R^{m}$, $m \geq 2$, and by $\N$ a smooth compact Riemannian manifold of dimension $n$ with empty boundary. Without loss of generality, we regard $\N$ as an isometrically embedded submanifold of a Euclidean space $\R^{N}$. The symbol $A$ will denote the second fundamental form of the embedding $\N \hookrightarrow \R^N$.

The Euclidean scalar product in $\R^N$ is denoted by $\langle \cdot, \cdot \rangle$. Since the metric on $\N$ is induced by the flat metric on $\R^N$, the same symbol will also be adopted for the scalar product between tangent vectors to $\N$. The standard connection in $\R^m$ is denoted by $D$. If $\{ e_{i} \}_{i=1}^{m}$ is an orthonormal basis of $\R^m$, we will denote by $D_{i}$ the directional derivative operator $D_{e_i}$. 

The open ball with center $x$ and radius $r$ in $\R^m$ is denoted $B_{r}(x)$. If $1 \leq k \leq m-1$ and $L \subset \R^m$ is a linear subspace of dimension $\dim(L) = k$, then we will denote the disc $(x + L) \cap B_{r}(x)$ by $B^L_{r}(x)$, or often with the simpler notation $B^{k}_{r}(x)$.

\subsection{Multiple valued functions}
Fix an integer $Q \geq 1$. We will assume that the reader is familiar with the theory of Almgren's $Q$-valued functions, for which we refer to \cite{DLS11a}. In what follows, we briefly recall the main definitions and properties we are going to need in the sequel. The space of $Q$-points in $\R^N$ is denoted $\A_{Q}(\R^N)$, and defined by
\[
\A_{Q}(\R^N) := \left\lbrace T = \sum_{\ell=1}^{Q} \llbracket p_{\ell} \rrbracket \, \colon \, \mbox{each } p_{\ell} \in \R^N \right\rbrace,
\]
where $\llbracket p_{\ell} \rrbracket$ is the Dirac delta measure centered at $p_{\ell}$. Observe that, by definition, a $Q$-point $T$ is a purely atomic measure of mass $Q$ in $\R^N$ which is obtained as the sum of Dirac deltas with integer multiplicities. If $T \in \A_{Q}(\R^N)$, the symbol $\spt(T)$ will denote the \emph{support} of the aforementioned measure. We endow $\A_{Q}(\R^N)$ with the structure of complete metric space determined by the distance $\G(T_1, T_2)$ given by
\[
\G(T_1, T_2)^{2} := \min_{\sigma \in \mathcal{P}_Q} \sum_{\ell=1}^{Q} \abs{p_{\ell} - q_{\sigma(\ell)}}^{2}, 
\]
where $\mathcal{P}_Q$ denotes the group of permutations of $\{1,\dots,Q\}$, and $T_{1} = \sum_{\ell} \llbracket p_{\ell} \rrbracket$, $\,$ $T_{2} = \sum_{\ell} \llbracket q_{\ell} \rrbracket$.

Any map $f \colon \Omega \to \A_{Q}(\R^N)$ will be called a $Q$-valued function. It is a simple observation (cf. \cite[Proposition 0.4]{DLS11a}) that if $f$ is a measurable $Q$-valued function then there are measurable maps $f_{\ell} \colon \Omega \to \R^N$ for $\ell = 1,\dots,Q$ such that $f(x) = \sum_{\ell=1}^{Q} \llbracket f_{\ell}(x) \rrbracket$ at a.e. $x \in \Omega$. Any choice of $\{f_{\ell}\}_{\ell=1}^{Q}$ as above is called a measurable \emph{selection} for $f$.

For $p \in \left[ 1, \infty \right]$, the spaces $L^{p}( \Omega, \A_{Q}(\R^N) )$ consist of those measurable $f \colon \Omega \to \A_{Q}(\R^N)$ for which
\begin{align*}
\| f \|_{L^p}^{p} := \int_{\Omega} \G(f(x), Q\llbracket 0 \rrbracket)^{p} \, \mathrm{d}x &< \infty \quad \mbox{when } 1 \leq p < \infty \, , \\
\| f \|_{L^{\infty}} := \esssup_{x \in \Omega} \G(f(x), Q \llbracket 0 \rrbracket) &< \infty \, .
\end{align*}
For the sake of notational simplicity, we will often set $\abs{T} := \G(T, Q\llbracket 0 \rrbracket)$ if $T \in \A_{Q}(\R^N)$. We remark that if $Q > 1$ then $\A_{Q}(\R^N)$ is not a linear space: hence, in spite of the notation, $T \mapsto |T|$ is not a norm.

A map $f \colon \Omega \to \A_{Q}(\R^N)$ belongs to the Sobolev space $W^{1,p}(\Omega, \A_{Q}(\R^N))$ if there exists $\psi \in L^{p}(\Omega)$ such that for every Lipschitz function $\phi \colon \A_{Q}(\R^N) \to \R$ it holds:
\begin{itemize}
\item[$(i)$] $\phi \circ f \in W^{1,p}(\Omega)$;
\item[$(ii)$] $\abs{D(\phi \circ f)(x)} \leq \Lip(\phi) \psi(x)$ at a.e. $x \in \Omega$. 
\end{itemize}
By \cite[Proposition 4.2]{DLS11a}, if $f \in W^{1,p}(\Omega, \A_{Q}(\R^N))$ then for every $i \in \{ 1, \dots, m \}$ there exists a unique $g_{i} \in L^{p}(\Omega)$ such that
\begin{itemize}
\item[$(i)$] $\abs{D_{i} (\G(f, T))} \leq g_{i}$ a.e. for every $T \in \A_{Q}(\R^N)$;
\item[$(ii)$] if $h_{i} \in L^{p}(\Omega)$ is such that $\abs{D_{i} (\G(f, T))} \leq h_{i}$ a.e. for every $T \in \A_{Q}(\R^N)$ then $g_{i} \leq h_{i}$ a.e. 
\end{itemize}
We will call the function $g_{i}$ the \emph{metric derivative} of $f$ in the direction $e_{i}$.

As usual, $W^{1,p}_{loc}(\Omega, \A_{Q}(\R^N))$ consists of those measurable functions which are in $W^{1,p}(\Omega', \A_{Q}(\R^N))$ for every $\Omega' \Subset \Omega$.

A $Q$-valued function $f \colon \Omega \to \A_{Q}(\R^N)$ is \emph{differentiable} at a point $x \in \Omega$ if there exist $Q$ linear maps $\lambda_{\ell} \colon \R^m \to \R^N$ satisfying
\begin{itemize}
\item[$(i)$] $\G(f(y), \sum_{\ell} \llbracket f_{\ell}(x) + \lambda_{\ell} \cdot (y-x) \rrbracket) = o(\abs{y-x})$ for $|y-x| \to 0$;
\item[$(ii)$] $\lambda_{\ell} = \lambda_{\ell'}$ if $f_{\ell}(x) = f_{\ell'}(x)$.
\end{itemize}
If $f$ is differentiable at $x$, then the $Q$-point $\sum_{\ell=1}^{Q} \llbracket \lambda_{\ell} \rrbracket \in \A_{Q}(\R^{N\times m})$ is the \emph{differential} of $f$ at $x$, and will be denoted $D f(x)$ or $\left.D f \right|_{x}$. We will write $D f_{\ell}(x)$ for the map $\lambda_{\ell}$, so that $D f(x) = \sum_{\ell} \llbracket D f_{\ell}(x) \rrbracket$, and we establish the notation $D_{\tau} f(x) := \sum_{\ell} \llbracket D f_{\ell}(x) \cdot \tau \rrbracket = \sum_{\ell} \llbracket D_{\tau} f_{\ell}(x) \rrbracket \in \A_{Q}(\R^{N})$ for the directional derivative in the direction $\tau \in \R^m$. We will also sometimes write $Df(x) \cdot \tau$ for $D_{\tau}f(x)$, so that $Df(x) \cdot \tau = \sum_{\ell}\llbracket Df_{\ell}(x) \cdot \tau \rrbracket$. 

It is a consequence of the Lipschitz approximation theorem for Sobolev $Q$-valued functions \cite[Proposition 2.5]{DLS11a} and of the $Q$-valued counterpart of Rademacher's theorem \cite[Theorem 1.13]{DLS11a} that every Sobolev $Q$-valued map is approximately differentiable at a.e. $x \in \Omega$. Furthermore, as shown in \cite[Proposition 2.17]{DLS11a}, if $f \in W^{1,2}(\Omega, \A_{Q}(\R^N))$ then for every $i \in \{ 1, \dots, m \}$ it holds 
\[
g_{i}^{2} = \G(D_{i}f, Q \llbracket 0 \rrbracket)^{2} = \abs{D_{i}f}^{2} \quad \mbox{a.e. in } \Omega.
\]
This makes unambiguous the use of the notation $\abs{D_{i} f}$ for the metric derivative $g_{i}$. In particular, for $f \in W^{1,2}_{loc}(\Omega, \A_{Q}(\R^N))$ there is a well-defined notion of (rescaled) \emph{Dirichlet energy} in a ball $B_{r}(x) \Subset \Omega$, given by
\[
\E(f,B_{r}(x)) := r^{2-m} \int_{B_{r}(x)} \abs{D f(y)}^{2} \, \mathrm{d}y = r^{2-m} \int_{B_{r}(x)} \sum_{i=1}^{m} \abs{D_{i} f(y)}^{2} \, \mathrm{d}y.
\]

\subsection{\texorpdfstring{$Q$}{Q}-valued energy minimizing maps}

Now, set
\begin{equation}
W^{1,2}_{loc}(\Omega, \A_{Q}(\N)) := \left\lbrace u \in W^{1,2}_{loc}(\Omega, \A_{Q}(\R^{N})) \, \colon \, \spt(u(x)) \subset \N \, \mbox{ for a.e. } x \in \Omega \right\rbrace.
\end{equation}

\begin{definition}[Energy minimizers, cf. {\cite[Definition 1.1]{Hir16}}] \label{def:minimi_energia}
A map $u \in W^{1,2}_{loc}(\Omega, \A_{Q}(\N))$ is a \emph{local minimizer}, or simply \emph{minimizer}, of the Dirichlet energy if for any $B_{r}(x) \Subset \Omega$ the following holds
\begin{equation} \label{eq:minimality}
\E(u, B_{r}(x)) \leq \E(v, B_{r}(x))
\end{equation} 
for every $v \in W^{1,2}_{loc}(\Omega, \A_{Q}(\N))$ such that $v \equiv u$ in a neighborhood of $\partial B_{r}(x)$.
\end{definition}

The $Q$-valued energy minimizers defined in Definition \ref{def:minimi_energia} are the multi-valued counterpart of classical energy minimizing harmonic maps. We refer the reader to the beautiful monographs of Simon \cite{Simon96}, Moser \cite{Moser} or Lin-Wang \cite{Lin_HM} for more about classical (single-valued) energy minimizing maps. 

As anticipated in the introduction, a partial regularity theory for $Q$-valued energy minimizers was developed by the first author in \cite{Hir16}. For further reference, and for the readers' convenience, let us briefly collect the main results of \cite{Hir16} which we are going to use in the sequel.

A first important observation is that if $u \in W^{1,2}_{loc}(\Omega, \A_{Q}(\N))$ is energy minimizing and if $B_{r}(x) \Subset \Omega$ then one can test the minimality of $u$ along suitably chosen families $u_{\varepsilon}$ of competitors in order to infer that $u$ satisfies some integral equations, known as \emph{variational equations}, which turn out to be of fundamental importance for the regularity theory. There are two important kinds of variations that one may consider in this context: the \emph{inner variations} (obtained by perturbing $u$ by means of right compositions with diffeomorphisms in the \emph{domain}) and the \emph{outer variations} (obtained by perturbing $u$ by means of left compositions with diffeomorphisms in the \emph{target}).

\begin{proposition}[Variational equations, cf. {\cite[Equations (2.2) and (2.5)]{Hir16}}] \label{variational_equations}
Fix $B_{r}(x) \Subset \Omega$ and let $u \in W^{1,2}_{loc}(\Omega, \A_{Q}(\N))$ be energy minimizing. Then, for every vector field $X = \left( X^{1}, \dots, X^{m}\right) \in C^{1}_{c}(B_{r}(x), \R^{m})$ the following \emph{inner variation formula} holds:
\begin{equation}\label{eq:inner_variation_formula}
\int_{B_{r}(x)} \sum_{i,j=1}^{m} \left( |D u|^{2} \delta_{ij} - 2 \sum_{\ell=1}^{Q} \langle D_{i} u_{\ell}, D_{j} u_{\ell} \rangle \right) D_{i} X^{j} \, \mathrm{d}y = 0.
\end{equation} 
Moreover, for any vector field $Y \in C^{1}(B_{r}(x) \times \R^{N}, \R^{N})$ such that $Y(y,p) = 0$ for $y$ in a neighborhood of $\partial B_{r}(x)$ we have the following \emph{outer variation formula}:
\begin{equation} \label{outer_variation_formula}
\int_{B_{r}(x)} \sum_{i=1}^{m} \sum_{\ell=1}^{Q} \left( \langle D_{i} u_{\ell}, D_{i} (Y(y, u_{\ell})) \rangle + \langle A_{u_{\ell}}(D_{i} u_{\ell}, D_{i} u_{\ell}), Y(y, u_{\ell}) \rangle \right) \, \mathrm{d}y = 0.
\end{equation}
\end{proposition}

Recall that in the classical case $Q = 1$ a map $u \in W^{1,2}_{loc}(\Omega, \N)$ satisfying the identity \eqref{outer_variation_formula} for any $Y$ is referred to as a \emph{weakly harmonic map}, whereas a map $u$ for which both \eqref{eq:inner_variation_formula} and \eqref{outer_variation_formula} hold for any choice of $X$ and $Y$ is called a \emph{stationary harmonic map}. Analogously, we will call \emph{stationary $Q$-harmonic} any map $u \in W^{1,2}_{loc}(\Omega, \A_{Q}(\N))$ for which both equations \eqref{eq:inner_variation_formula} and \eqref{outer_variation_formula} hold. Of course, by Proposition \ref{variational_equations} every $Q$-valued energy minimizing map $u$ is stationary $Q$-harmonic. On the other hand, some of the results that we present here hold true under the weaker assumption that $u$ is stationary $Q$-harmonic rather than minimizing, since their proofs are a consequence solely of the variational equations. We will explicitly underline in our statements every time that the result applies also to stationary $Q$-harmonic maps.

A first powerful result stemming from the variational equations is the monotonicity of the map $r \in \left( 0, \dist(x,\partial\Omega) \right) \mapsto \E(u, B_{r}(x))$ for every fixed point $x \in \Omega$, cf. \cite[Equation (2.6)]{Hir16}. As a consequence, if $u$ is stationary $Q$-harmonic then for every $x \in \Omega$ the \emph{density} $\Theta_{u}(x)$ of $u$ at $x$ is well-defined by the formula
\[
\Theta_{u}(x) := \lim_{r \downarrow 0} \E(u, B_{r}(x)) \,.
\]

Multiple-valued energy minimizers also enjoy the following compactness theorem.
\begin{theorem}[Compactness, cf. {\cite[Lemma 4.1]{Hir16}}] \label{HM_Q_compactness}
Let $\{ u_{h} \}_{h=1}^{\infty} \subset W^{1,2}(\Omega, \A_{Q}(\N))$ be a sequence of $Q$-valued minimizing harmonic maps with $\sup_{h \geq 1} \E(u_h, B_{r}(x)) < \infty$ for each ball $B_{r}(x) \Subset \Omega$. Then, there is a subsequence $u_{h_j}$ and a minimizing harmonic map $u \in W^{1,2}(\Omega, \A_{Q}(\N))$ such that 
\begin{itemize}

\item[$(i)$] $\lim_{j \to \infty} \int_{\Omega} \G(u_{h_j}, u)^{2} \, {\rm d}y = 0$;

\item[$(ii)$] $\lim_{j \to \infty} \E(u_{h_j}, B_{r}(x)) = \E(u, B_{r}(x))$ for every ball $B_{r}(x) \Subset \Omega$.

\end{itemize}
\end{theorem}

The monotonicity of the rescaled energy at a fixed point $x_{0} \in \Omega$ together with the Compactness Theorem \ref{HM_Q_compactness} allow to conclude the existence of tangent maps. In particular, for every sequence $r_h$ of radii with $r_h \downarrow 0$ there exists a subsequence $r_{h'}$ such that the maps $T^{u}_{x_0,r_{h'}}(y) := u(x_0 + r_{h'}y)$ converge in $L^2$ and locally in energy to a $Q$-valued energy minimizing map $\phi \in W^{1,2}_{loc}(\R^m, \A_{Q}(\N))$. Any map $\phi$ arising as a limit of a sequence $T^{u}_{x_0,r_h}$ for some sequence $r_h \downarrow 0$ is called a \emph{tangent map} to $u$ at $x_0$ (cf. \cite[Definition 6.1]{Hir16}). Every tangent map $\phi$ is homogeneous of degree zero with respect to $0 \in \R^m$, and thus it satisfies $\E(\phi, B_{\rho}(0)) = \Theta_{\phi}(0) = \Theta_{u}(x_0)$ for every $\rho > 0$. Furthermore, the map $y \in \R^m \mapsto \Theta_{\phi}(y)$ attains its maximum at $y = 0$. The set of points $y \in \R^m$ for which $\Theta_{\phi}(y) = \Theta_{\phi}(0)$ is classically called the \emph{spine} of $\phi$, and it is denoted $S(\phi)$. It turns out that, exactly as in the classical case, $S(\phi)$ is a linear subspace of $\R^m$, and $\phi$ is invariant with respect to compositions with translations by elements in $S(\phi)$, that is $\phi(x + y) = \phi(x)$ for every $x \in \R^m$, for any $y \in S(\phi)$. The dimension of $S(\phi)$ is the number of independent directions along which $\phi$ is invariant. Now, if $u$ is H\"older continuous in a neighborhood of $x_0$ then it is easy to see that $\Theta_{u}(x_0) = 0$, and thus $u$ admits in $x_0$ a tangent map $\phi$ for which $S(\phi) = \R^m$, and $\phi$ is constant. The set $\sing_{H}(u)$ of points $x$ such that $u$ is not H\"older continuous in a neighborhood can be instead classically stratified according to the number of symmetries that the tangent maps at points in it have. In particular, for $0 \leq k \leq m-1$ one defines
\begin{equation} \label{standard_singular_strata}
\cS^{k}(u) := \left\lbrace x \in \sing_{H}(u) \, \colon \, \dim S(\phi) \leq k \mbox{ for every tangent map $\phi$ to $u$ at $x$} \right\rbrace\,.
\end{equation}

The following $\e$-regularity theorem in the spirit of Schoen-Uhlenbeck \cite{SU} is the core of \cite{Hir16} and the key to completing the partial regularity theory.
\begin{theorem}[$Q$-valued $\e$-regularity, cf. {\cite[Lemma 5.2]{Hir16}}] \label{Q_eps_reg}
There exist constants $\e_0 > 0$, $\alpha > 0$ and $C > 1$ depending on $m,\N,Q$ with the property that if $u \in W^{1,2}_{loc}(\Omega, \A_{Q}(\N))$ is energy minimizing in $B_{R_0}(x_0)$ with
\[
\E(u, B_{R_0}(x_0)) \leq \e_0\,,
\]
then the following energy decay estimate holds:
\[
\E(u, B_{r}(x)) \leq C \left( \frac{r}{R} \right)^{2\alpha} \E(u, B_{R}(x)) \quad \forall \, x \in B_{\frac{R_0}{2}}(x_0) \,, \forall \, 0 < r \leq R \leq \frac{R_0}{2}.
\]
In particular, $u \in C^{0,\alpha}(B_{\frac{R_0}{2}}(x_0), \A_{Q}(\N))$.
\end{theorem}

By the $\e$-regularity theorem, the set $\sing_{H}(u)$ coincides with the set $\left\lbrace x \, \colon \, \Theta_{u}(x) > 0 \right\rbrace$. Using this information, it is standard to conclude that if $u$ is minimizing then $\Ha^{m-2}(\sing_{H}(u)) = 0$. On the other hand, if $x_0 \in \sing_{H}(u)$ then for any tangent map $\phi$ one has $\Theta_{\phi}(0) = \Theta_{u}(x_0) > 0$, and thus the spine $S(\phi)$ is a subset of $\sing_{H}(\phi)$. Since $S(\phi)$ is a linear subspace, and since $\Ha^{m-2}(\sing_{H}(\phi)) = 0$, we have that $\dim S(\phi) \leq m-3$. Hence, $\cS^{m-1}(u) = \cS^{m-2}(u) = \cS^{m-3}(u) = \sing_{H}(u)$. By a variation of the standard Federer-Almgren dimension-reduction argument, one then concludes that $\dim_{\Ha}(\sing_{H}(u)) = \dim_{\Ha}(\cS^{m-3}(u)) \leq m-3$.

\vspace{3mm} In the next sections we will turn our attention to the \emph{quantitative stratification} for $Q$-valued minimizing maps, which will allow us to obtain better information on the fine properties of the singular strata $\cS^{k}(u)$. Before doing that, we will slightly modify the definition of the rescaled energy: the new ``mollified'' energy that we are going to introduce in the coming section, will be more useful for quantitative estimates.

\section{The mollified Dirichlet energy and its monotonicity}

\begin{definition}[Mollified energy] \label{def:mollified_energy}
Let $\varphi = \varphi(t)$ be any non-negative function in $C^{1}_{c}(\left[ 0, 1 \right))$ which is constant in a neighborhood of $t = 0$.

Then, for any $u \in W^{1,2}_{loc}(\Omega, \A_{Q}(\N))$ and for any $B_{r}(x) \subset \Omega$ we define the quantity
\begin{equation} \label{def:energy}
\theta_{u}(x,r) := r^{2 - m} \int \varphi\left( \frac{|x-y|}{r}\right) |D u(y)|^{2} \, \mathrm{d}y.
\end{equation}
\end{definition}

When the map $u$ is fixed, we will simply write $\theta(x,r)$ for the sake of notational simplicity. In what follows, we show that, under suitable assumptions on $\varphi$, the function $r \mapsto \theta(x,r)$ is monotone non-decreasing for fixed $x$, and we explicitly compute its derivative.

\begin{notazioni}
For any $x \in \R^{m}$, we shall denote by $r_{x}$ the radial unit vector field with respect to $x$, defined by
\[
r_{x}(y) := \frac{y-x}{|y-x|} \quad \mbox{for every } y \in \R^{m} \setminus \{x\}.
\]
\end{notazioni}

\begin{lemma} Let $u \in W^{1,2}_{loc}(\Omega,  \I{Q}(\N))$ be a stationary $Q$-harmonic map, and let $x \in \Omega$. For any $\varphi$ as in Definition \ref{def:mollified_energy}, the following identity holds true for all $r$ such that $B_r (x) \subset \Omega$:
\begin{equation}\label{eq:monotonicity identity-differnetial form}
\frac{d}{dr} \theta(x,r) = - 2 r^{2-m}  \int  \varphi'\left(\frac{\abs{x-y}}{r}\right)\frac{\abs{x-y}}{r^2}  \Abs{D_{r_x}u(y)}^2 \, \mathrm{d}y.
\end{equation}
In particular, if we let $\psi = \psi(t)$ denote a primitive function of $\varphi'(t) t^{m-2}$, then for $0 < s < r < \dist(x, \partial\Omega)$ we have: 
\begin{equation}\label{eq:monotonicity identity-integrated}
\theta(x,r)-\theta(x,s) = \int \left(\psi\left(\frac{\abs{x-y}}{r}\right)-\psi\left(\frac{\abs{x-y}}{s}\right)\right) \abs{x-y}^{2-m} \Abs{D_{r_x}u(y)}^2 \, \mathrm{d}y.
\end{equation}
In case we choose $\varphi$ to be non-increasing, we have that $r \mapsto \theta(x,r)$ is non-decreasing; furthermore, if $-\varphi'(t) \ge (1-t)^+$ then it holds
\begin{equation}\label{eq:monotonicity identity-integrated lower bound}
\theta(x,r)-\theta(x,r/2) \ge C \int_{B_{\frac{r}{2}}(x)}   \frac{\abs{x-y}}{r^{m-1}} \Abs{D_{r_x}u(y)}^2 \, \mathrm{d}y
\end{equation}
for some positive constant $C = C(m)$. 
\end{lemma}

\begin{proof}
The identity \eqref{eq:monotonicity identity-differnetial form} follows from the inner variation formula, equation \eqref{eq:inner_variation_formula}. Indeed, for any fixed $x \in \Omega$ and $0 < r < \dist(x, \partial\Omega)$ define the vector field $X(y):= \varphi\left(\frac{\abs{x-y}}{r}\right) (y-x)$. If we plug this choice of $X$ in \eqref{eq:inner_variation_formula}, we easily deduce the identity
\[ (m-2) \int \varphi\left(\frac{\abs{x-y}}{r}\right) \abs{D u(y)}^2 \, \mathrm{d}y\, +\, \int \varphi'\left(\frac{\abs{x-y}}{r}\right)\frac{\abs{x-y}}{r} \left( \abs{D u(y)}^2 - 2 \Abs{D_{r_x}u(y)}^2 \right) \, \mathrm{d}y = 0.\]
To conclude, we can differentiate the quantity $\theta(x,r)$ in $r$ and obtain the differential identity \eqref{eq:monotonicity identity-differnetial form}.

Now, let $\psi$ be a primitive function of $\varphi'(t) t^{m-2}$. We have 
\[
\frac{d}{dr} \psi\left(\frac{\abs{x-y}}{r}\right)= -\frac{1}{r} \varphi'\left(\frac{\abs{x-y}}{r}\right) \left(\frac{\abs{x-y}}{r}\right)^{m-1},
\]
and thus we can rewrite \eqref{eq:monotonicity identity-differnetial form} as
\[ \frac{d}{dr} \theta(x,r) = 2 \frac{d}{dr} \int \psi\left(\frac{\abs{x-y}}{r}\right) \abs{x-y}^{2-m} \Abs{D_{r_x}u(y)}^2\, \mathrm{d}y.\]
Integrating immediately leads to \eqref{eq:monotonicity identity-integrated}.

If we choose $\varphi'(t) \le 0$, then \eqref{eq:monotonicity identity-differnetial form} implies that $r\mapsto \theta(x,r)$ is non-decreasing.
In case $-\varphi'(t)\geq(1-t)^+$, we have for $0 < a \le \frac12$
\begin{equation} \label{psi_estimate} 
\psi(a)-\psi(2a) = - \int_a^{2a} \varphi'(t) t^{m-2} \geq a^{m-1} \left( \frac{ 2^{m-1}-1}{m-1} - a \frac{2^m-1}{m} \right) \ge C_m  a^{m-1}. 
\end{equation}
Hence, the estimate \eqref{eq:monotonicity identity-integrated lower bound} can be deduced from \eqref{eq:monotonicity identity-integrated} by using the fact that $\psi$ is non-increasing to estimate
\begin{equation}
\theta(x,r) - \theta(x, r/2) \geq \int_{B_{\frac{r}{2}}(x)} \left(\psi\left(\frac{\abs{x-y}}{r}\right)-\psi\left(\frac{2\abs{x-y}}{r}\right)\right) \abs{x-y}^{2-m} \Abs{D_{r_x}u(y)}^2 \, \mathrm{d}y, 
\end{equation}
and then using the inequality in \eqref{psi_estimate} with $a = \frac{|x - y|}{r}$ for $y \in B_{\frac{r}{2}}(x)$.
\end{proof}

\begin{ipotesi}
For the rest of the paper, we will assume that $\varphi$ has been fixed, and that it satisfies the condition $-\varphi'(t) \geq (1-t)^{+}$, so that the inequality \eqref{eq:monotonicity identity-integrated lower bound} holds.
\end{ipotesi}

\section{Quantitative stratification}\label{sec_strat}

The first step towards the definition of the quantitative singular strata is to introduce the notion of ``model maps'' having a given number of symmetries. This definition is analogous to \cite[Definition 1.1]{naber-valtorta:harmonic}.

\begin{definition}[$k$-symmetric maps] \label{def:hom_sym}
A map $h \in W^{1,2}_{loc}(\R^{m}, \A_{Q}(\N))$ is said to be:
\begin{itemize}

\item[•] \emph{homogeneous} with respect to $x \in \R^{m}$ if
\[
h(x + \lambda v) = h(x + v) \quad \mbox{for all } \lambda > 0, \mbox{ for every } v \in \R^{m},
\]
or equivalently if 
\[
D_{r_{x}}h = Q \llbracket 0 \rrbracket \quad \mbox{a.e. in } \R^{m}.
\]

\item[•] $k$-\emph{symmetric} if it is homogeneous with respect to the origin and there exists a linear subspace $L \subset \R^{m}$ with $\dim(L) = k$ along which $h$ is invariant, that is
\[
h(x + v) = h(x) \quad \mbox{for every } x \in \R^{m}, \mbox{ for all } v \in L,
\]
or, equivalently, such that
\[
D_{v}h(x) = Q \llbracket 0 \rrbracket, \quad \mbox{for a.e. } x \in \R^{m}, \mbox{ for all } v \in L.
\]
\end{itemize}
\end{definition}

Observe that if $h \in W^{1,2}_{loc}(\R^m, \A_{Q}(\N))$ is stationary and homogeneous with respect to $x$ then $\theta_{h}(x,s) = \theta_{h}(x,r)$ for every $0 < s < r$ by \eqref{eq:monotonicity identity-integrated}. Also, if $h$ is $k$-symmetric with invariance subspace $L$ then the energy of $h$ in the direction of any $v \in L$ vanishes. Hence, it is very natural to give the following definition, which is the starting point for introducing the quantitative stratification. 

\begin{definition} \label{def:quant_sym}
Given a stationary $Q$-harmonic map $u \in W^{1,2}_{loc}(\Omega, \A_{Q}(\N))$, we say that a ball $B_{r}(x)$ with $B_{2r}(x) \subset \Omega$ is $(k, \e)$-symmetric for $u$ if and only if the following conditions hold:
\begin{itemize}
\item[$(a)$] $\theta_{u}(x, 2r) - \theta_{u}(x, r) < \e$;

\item[$(b)$] there exists a linear subspace $L \subset \R^{m}$ with $\dim(L) = k$ such that
\[
r^{2-m} \int_{B_{r}(x)} |D_{L}u(y)|^{2} \, \mathrm{d}y \leq \e,
\]
where
\[
\int_{B_{r}(x)} |D_{L}u(y)|^{2} \, \mathrm{d}y := \int_{B_{r}(x)} \sum_{i=1}^{k}  |D_{e_i}u(y)|^{2} \, \mathrm{d}y,
\]
for any orthonormal basis $\{ e_{i} \}_{i=1}^{k}$ of $L$.
\end{itemize}
\end{definition}

\begin{remark}
Observe that the conditions $(a)$ and $(b)$ above are scale-invariant in the following sense. For $x \in \Omega$ and $r > 0$ such that $B_{2r}(x) \subset \Omega$, consider the blow-up map $T_{x,r}^{u}$ given by
\[
T_{x,r}^{u}(y) := u(x + ry).
\]
Then, $B_{r}(x)$ is $(k,\e)$-symmetric with respect to $u$ if and only if $B_{1}(0)$ is $(k, \e)$-symmetric with respect to $T_{x,r}^{u}$. 
\end{remark}

\begin{definition}[Quantitative stratification] \label{def:stratification}

Let $u \in W^{1,2}_{loc}(\Omega, \A_{Q}(\N))$ be stationary $Q$-harmonic, and let $\e, r>0$ and $k \in \{0,\dots,m\}$. We will set
\[
\cS_{\e,r}^{k}(u) := \left\lbrace x \in \Omega \, \colon \, \mbox{for no $r \leq s < 1$ the ball $B_{s}(x)$ is $(k+1,\e)$-symmetric with respect to $u$}\right\rbrace.
\]
It is an immediate consequence of the definition that if $k' \leq k$, $\e' \geq \e$ and $r' \leq r$ then
\[
\cS_{\e', r'}^{k'}(u) \subseteq \cS_{\e,r}^{k}(u).
\]
Hence, we can set:
\[
\cS_{\e}^{k}(u) := \bigcap_{r > 0} \cS_{\e,r}^{k}(u), \quad \quad \cS^{k}(u) := \bigcup_{\e > 0} \cS_{\e}^{k}(u).
\]

\end{definition}

\begin{remark}\label{rem_sing}
Note that from Theorem \ref{Q_eps_reg} one easily deduces that if $u \in W^{1,2}_{loc}(\Omega, \A_{Q}(\N))$ is energy minimizing and a ball $B_{r}(x)$ is $(m, \e_{0})$-symmetric for $u$, with the $\e_{0}$ given in there, then $u$ is H\"older continuous in $B_{\frac{r}{2}}(x)$, and thus in particular $\cS^{k}(u) \cap B_{\frac{r}{2}}(x) = \emptyset$ for every $k \leq m-1$. In fact, we can also conclude that $\cS^{m}(u) \setminus \cS^{m-1}(u)$ coincides with the set $\reg_{H}(u) := \Omega \setminus \sing_{H}(u)$ of points of H\"older continuity for $u$, and $\sing_{H}(u)=\cS^{m-1}(u)$.
\end{remark}

Also observe that we have used the same symbol $\cS^{k}(u)$ to denote both the set $\bigcup_{\e > 0} \cS^{k}_{\e}(u)$ coming from the quantitative stratification and the standard singular stratum defined in \eqref{standard_singular_strata}. The choice is completely justified by the following proposition.

\begin{proposition} \label{prop:charac_strata}
Let $u \in W^{1,2}(\Omega, \A_{Q}(\N))$ be energy minimizing. Then
\[
\cS^{k}(u) = \left\lbrace x \, \colon \, \mbox{no tangent map to $u$ at $x$ is $(k+1)$-symmetric} \right\rbrace.
\]
\end{proposition}

\begin{proof}
First recall that for any $x \in \Omega$ there exists at least one tangent map $\phi \in W^{1,2}_{loc}(\R^{m}, \A_{Q}(\N))$ to $u$ at $x$, and that all tangent maps are energy minimizing and $0$-symmetric.

Now, let $x$ be a point such that there exists a tangent map $\phi$ to $u$ at $x$ which is $(k+1)$-symmetric. Then, there is a sequence $r_{j} \searrow 0$ of radii such that the corresponding sequence of blow-up maps $u_{j} := T_{x,r_{j}}^{u}$ satisfies $\G(u_{j}, \phi) \to 0$ in $L^{2}_{loc}(\R^{m})$ as $j \to \infty$ and furthermore
\[
\theta_{\phi}(0,\rho) = \lim_{j \to \infty} \theta_{u_{j}}(0,\rho) \quad \forall\, \rho > 0.
\]
In particular, since $\phi$ is homogeneous with respect to the origin, and thus $\theta_{\phi}(0,2) - \theta_{\phi}(0, 1) = 0$ by \eqref{eq:monotonicity identity-integrated}, for any $\e > 0$ there exists $j_{0} = j_{0}(\e)$ such that
\begin{equation} \label{eq:pinching}
\theta_{u_{j}}(0,2) - \theta_{u_{j}}(0,1) < \e \quad \quad \forall \, j \geq j_{0}.
\end{equation}
Moreover, since $\phi$ is $(k+1)$-symmetric there exists a linear subspace $L \subset \R^{m}$ with $\dim(L) = k+1$ such that $D_{L}\phi = Q \llbracket 0 \rrbracket$ a.e. in $\R^{m}$. Hence, from the convergence of energy for minimizers we deduce that if $j_{0}$ is chosen suitably large then also
\begin{equation} \label{eq:symmetry}
\int_{B_{1}(0)} |D_{L}u_{j}|^{2} \, \mathrm{d}y \leq \e \quad \quad \forall \, j \geq j_{0}.
\end{equation}
Together, equations \eqref{eq:pinching} and \eqref{eq:symmetry} imply that $B_{r_{j}}(x)$ is $(k+1,\e)$-symmetric for $u$ if $j \geq j_{0}(\e)$, and thus $x \notin \cS^{k}(u)$. This proves the first inclusion, namely
\[
\cS^{k}(u) \subseteq \left\lbrace x \in \Omega \, \colon \, \mbox{no tangent map to $u$ at $x$ is $(k+1)$-symmetric} \right\rbrace.
\]
In order to prove the other inclusion, assume that $x \notin \cS^{k}(u)$. Then, for every $j \in \mathbb{N}$ there exist a radius $r_{j} > 0$ and a $(k+1)$-dimensional linear subspace $L_{j} \subset \R^{m}$ such that if we set $u_{j} := T^{u}_{x,r_j}$ then
\begin{equation} \label{eq:pinching2}
\theta_{u_j}(0,2) - \theta_{u_{j}}(0,1) < \frac{1}{j}
\end{equation} 
and
\begin{equation} \label{eq:symmetry2}
\int_{B_{1}(0)} |D_{L_j}u_{j}|^{2} \, \mathrm{d}y \leq \frac{1}{j}.
\end{equation}
Modulo a simple right composition of each $u_{j}$ with a rotation, we can assume that the invariant subspace is a fixed $(k+1)$-dimensional subspace $L \subset \R^{m}$. By the compactness theorem for $Q$-valued energy minimizing maps, a subsequence (not relabeled) of the $u_{j}$'s converges in $L^{2}_{loc}$ and in energy to an energy minimizing map $\phi$. From \eqref{eq:pinching2} together with \eqref{eq:monotonicity identity-integrated lower bound} we deduce that the limit map $\phi$ is homogeneous with respect to the origin. Furthermore, \eqref{eq:symmetry2} implies that $\phi$ is invariant along the subspace $L$, and thus $\phi$ is $(k+1)$-symmetric. Now, if a subsequence of the $r_{j}$'s converges to $0$ then $\phi$ is by definition a tangent map to $u$ at $x$. If, on the other hand, the $r_{j}$'s are bounded away from $0$ then $u = \phi$ on a ball of positive radius centered at $x$, and thus, in particular, all tangent maps to $u$ at $x$ coincide with $\phi$. In either case, this completes the proof.
\end{proof}

\begin{corollary}\label{prop: empty n-2 strata}
Let $u \in W^{1,2}(\Omega, \A_{Q}(\N))$ be energy minimizing. Then
\[ \cS^{m-1}(u)\setminus \cS^{m-3}(u) = \emptyset.\]
\end{corollary}
\begin{proof}
This is a direct consequence of the Proposition \ref{prop:charac_strata}, since the identity $\cS^{m-1}(u) = \cS^{m-3}(u)$ holds for the standard stratification.
\end{proof}

The definition of quantitative stratification that we have proposed differs from the original one introduced by Cheeger and Naber in \cite{ChNa1,ChNa2} and then used by Naber and Valtorta in \cite{naber-valtorta:harmonic}. Of course, the Cheeger-Naber quantitative stratification can be without any difficulties extended to the $Q$-valued context. We recall the definition here, in order to compare it with Definition \ref{def:stratification}.

\begin{definition} \label{def:old_strata}
Let $u \in W^{1,2}_{loc}(\Omega, \A_{Q}(\N))$, and fix $k \in \{0,\dots,m\}$ and $\e > 0$. A ball $B_{r}(x)$ with $B_{2r}(x) \subset \Omega$ is said to be $(k,\e)$-symmetric for $u$ \emph{in the sense of Cheeger-Naber}, or briefly [CN] $(k,\e)$-symmetric, if there exists some $k$-symmetric map $h \in W^{1,2}_{loc}(\R^{m}, \A_{Q}(\N))$ such that
\begin{equation} \label{eq:old_strata}
\fint_{B_{r}(x)} \G(u(y), h(y-x))^{2} \, \mathrm{d}y \leq \e.
\end{equation} 
The definitions of [CN] $(\e,r)$-singular strata and [CN] $\e$-singular strata can be then straightforwardly obtained according to the definition of [CN] $(k,\e)$-symmetry. In particular, $\cS^{k}_{{\rm [CN]}}(u)$ classically consists of all points $x \in \Omega$ having the property that there exists $\e > 0$ such that no ball $B_{r}(x)$ is [CN] $(k+1,\e)$-symmetric with respect to $u$.
\end{definition}

The following simple proposition shows that if $u$ is a minimizing $Q$-valued map then Definition \ref{def:quant_sym} and Definition \ref{def:old_strata} are equivalent, in the sense that they generate the same stratification. In order to fix the ideas, for the vast majority of the following results we will work under the following assumption.

\begin{ipotesi} \label{ass:harmonic}
Assume that $u \in W^{1,2}(B_{10}(0), \A_{Q}(\N))$ is a $Q$-valued energy minimizing map, and that $\E(u, B_{10}(0)) \leq \Lambda$. 
\end{ipotesi}

\begin{proposition}
For every $\e > 0$ there exists $\delta = \delta(m,\N,N,Q,\Lambda,\e) > 0$ such that for any $u$ satisfying Assumption \ref{ass:harmonic}:
\begin{itemize}
\item[$(i)$] if $B_{r}(x)$ is $(k,\delta)$-symmetric for $u$, then it is ${\rm [CN]}$ $(k,\e)$-symmetric for $u$;

\item[$(ii)$] if $B_{r}(x)$ is ${\rm [CN]}$ $(k,\delta)$-symmetric for $u$, then it is $(k,\e)$-symmetric for $u$.
\end{itemize} 
\end{proposition}

\begin{proof}
Since both the definitions of symmetry are scale-invariant, modulo translations and dilations it suffices to show the validity of the proposition for $x = 0$ and $r = 1$. We start proving the first claim. Assume by contradiction that there exist $\e_{0} > 0$ and a sequence $\{ u_{j} \}_{j\in\mathbb{N}}$ of maps as in Assumption \ref{ass:harmonic} for which the ball $B_{1}$ is $(k,j^{-1})$-symmetric but such that
\begin{equation} \label{contra:1}
\fint_{B_{1}} \G(u_{j}(y), h(y))^{2} \, \mathrm{d}y > \e_{0} \quad \mbox{for every $k$-symmetric function $h$, for every $j \in \mathbb{N}$.}
\end{equation}
Modulo rotations, we can assume that the $k$-dimensional linear subspace $L$ such that condition $(b)$ in Definition \ref{def:quant_sym} is satisfied is fixed along the sequence: namely, we can assume without loss of generality that
\[
\theta_{u_{j}}(0,2) - \theta_{u_{j}}(0,1) < j^{-1}
\]
and that
\[
\int_{B_{1}} |D_{L}u_{j}(y)|^{2} \, \mathrm{d}y \leq j^{-1}
\]
for some fixed $k$-dimensional plane $L \subset \R^{m}$. Now, the compactness theorem for $Q$-valued energy minimizing maps implies that a subsequence of the $u_{j}$'s (not relabeled) converges in $L^{2}(B_{10}(0), \A_{Q}(\R^N))$ and in energy to a $Q$-valued energy minimizing map $h$ for which
\[
\theta_{h}(0,2) - \theta_{h}(0,1) = 0
\]
and
\[
D_{L}h = Q \llbracket 0 \rrbracket \quad \mbox{a.e. in $B_{1}$}.
\]
Hence, by \eqref{eq:monotonicity identity-integrated lower bound} the map $\left.h\right|_{B_1}$ can be extended to a $k$-symmetric map (which we still denote by $h$), and the fact that $\G(u_{j},h) \to 0$ in $L^{2}(B_{1})$ contradicts \eqref{contra:1}.

For the converse, assume again by contradiction that there exist $\e_{0}$, a sequence $\{ u_{j} \}_{j\in\mathbb{N}}$ of maps as in Assumption \ref{ass:harmonic} and a sequence $\{ h_{j} \}_{j\in\mathbb{N}} \subset W^{1,2}_{loc}(\R^{m}, \A_{Q}(\N))$ of $k$-symmetric maps such that
\begin{equation} \label{contra:2}
\fint_{B_{1}} \G(u_{j}(y), h_{j}(y))^{2} \, \mathrm{d}y \leq j^{-1}
\end{equation} 
and such that the ball $B_{1}$ is not $(k,\e_{0})$-symmetric. Again, after applying suitable rotations we can assume that the invariant subspace for the maps $h_{j}$ is a fixed $k$-dimensional plane $L \subset \R^{m}$. By compactness, the maps $u_{j}$ converge, up to subsequences, to an energy minimizing $u \in W^{1,2}(B_{10}(0), \A_{Q}(\N))$. By \eqref{contra:2}, also $h_{j} \to u$ strongly in $L^{2}(B_{1}, \A_{Q}(\N))$. Since the space of $k$-symmetric maps is $L^{2}$-closed, we deduce that $u$ is $k$-symmetric. Since the $u_{j}$'s converge to $u$ also in energy, the ball $B_{1}$ must be $(k,\e_{0})$-symmetric for $u_{j}$ if $j$ is sufficiently large, which is the required contradiction.
\end{proof}

\begin{corollary}
Let $u$ satisfy Assumption \ref{ass:harmonic}. Then, for every $k \in \{0, \dots, m\}$ one has
\begin{equation} \label{equivalence_strata}
\cS^{k}(u) = \cS^{k}_{{\rm [CN]}}(u).
\end{equation}
\end{corollary}

Using more quantitative estimates, the comparison between the two notions of quantitative symmetry can be carried to the case of stationary $Q$-harmonic maps.

\begin{proposition}\label{prop:comparision_to_old_stata}
There exists a constant $C = C(m,\N,N,Q) > 0$ with the following property. Let $u \in W^{1,2}_{loc}(\Omega,\A_{Q}(\N))$ be a stationary $Q$-harmonic map. If a ball $B_r(x) \Subset \Omega$ is $(k,\e)$-symmetric for $u$, then $B_{\frac{r}{4}}(x)$ is $(k,C \abs{\e \ln(\e)})$-symmetric for $u$ in the sense of Cheeger-Naber.
\end{proposition}

\begin{proof}
Without loss of generality, we prove the claim for $x = 0$ and $r = 1$. The idea of the proof is to explicitly construct from $u$ a $k$-symmetric map in $B_{\frac{1}{4}}$. Modulo a rotation, we can assume that the $k$-dimensional plane $L$ of $\e$-almost symmetry is $L=\{x_i = 0 \, \colon \, i > k \} = \R^{k} \times \{0\}$. For convenience, we will denote the variables of $\R^k$ with $y, y'$ and the variables of $\R^{m-k}$ with $z,z'$. The point $x \in \R^{m}$ will be therefore given coordinates $x = (y,z) \in \R^{k} \times \R^{m-k}$. With a slight abuse of notation, we will also sometimes regard $y$ and $z$ as vectors in $\R^{m}$, thus avoiding the cumbersome, although more correct, writings $(y,0)$ and $(0,z)$. Finally, when we integrate a function with respect to the variable $y$ over a ball $B_{r}^{k} \subset \R^{k}$ we will use the notation $B_{r}^{y}$ as domain of integration (and analogously for the variables $y',z,z'$).

In order to construct the $k$-symmetric map, we need to prove two simple inequalities for multiple-valued functions.\\
\begin{itemize}
\item[\textit{Claim 1:}] there exists a constant $C = C(k,N,Q)$ with the following property. For any function $f=f(y,z)$ in $W^{1,2}\left(B^k_1\times B^{m-k}_1, \A_{Q}(\R^N)\right)$, one has
\begin{equation}\label{eq:halfpoincare}
\int_{B^{y'}_1} \int_{B_1^y\times B_1^z} \G( f(y,z), f(y',z))^2 \le C \int_{B_1^y \times B_1^z} \abs{D_L f}^2. 
\end{equation}
\item[\textit{Claim 2:}] Let $0 < s_0< a < 1$. There exists a constant $C = C(d,a,Q)$ such that for any $f \in W^{1,2}(B_1^{d}, \A_{Q}(\R^N))$ and every $a<t\leq 1$ such that $f|_{\partial B_t^{d}} \in W^{1,2}(\partial B_t^{d}, \A_{Q}(\R^N))$ the following holds:
\begin{equation}\label{eq:0-homogeneous comparison} 
\int_{B_1^{d}\setminus B_{s_0}^{d}} \G\left( f(x) , f\left(t \frac{x}{\abs{x}}\right)\right)^2 \le C \abs{\ln(s_0)} \int_{B_1^{d}\setminus B_{s_0}^{d}} \left| D f(x) \cdot x \right|^2. 
\end{equation}
\end{itemize}
\emph{Proof of Claim 1:} The proof is a consequence of the Poincar\'e inequality for multiple valued functions, see \cite[Proposition 4.9]{DLS11a}. Indeed, first observe that for a.e. $z \in B_{1}^{m-k}$ the map $y \mapsto f(y,z) $ is in $W^{1,2}(B_{1}^{k}, \A_{Q}(\R^N))$. Hence, by the aforementioned Poincar\'e inequality, for any such a $z$ there exists a point $\bar{f}(z) \in \A_{Q}(\R^N)$ such that
\[
\int_{B_{1}^{y}} \G(f(y,z), \bar{f}(z))^{2} \leq C \int_{B_{1}^{y}} \abs{D_{L}f(y,z)}^{2}\, ,
\]
where $C = C(k,N,Q)$. Hence, by triangle inequality we infer that
\[
\begin{split}
\int_{B^y_1} \int_{B_1^{y'}} \G(f(y,z) , f(y',z))^2 &\le 2\Ha^{k}(B^{k}_{1}) \left( \int_{B^y_1} \G(f(y,z) , \bar{f}(z))^2  +   \int_{B^{y'}_1} \G(f(y',z) , \bar{f}(z))^2 \right)\\
&\leq C  \int_{B^y_1} \abs{D_Lf(y,z)}^2 \, .
\end{split}
\]
Integrating now this inequality in $z \in B^{m-k}_1$ gives \eqref{eq:halfpoincare}.

\emph{Proof of Claim 2:} First note that for $\Ha^{d-1}$-a.e. $w \in \partial B_{1}^{d}$ the map $r \mapsto g^{w}(r) := f(rw)$ is in $W^{1,2}(\left( 0,1 \right), \A_{Q}(\R^N))$. By the $W^{1,2}$-selection theorem for multiple-valued functions of one variable (cf. \cite[Proposition 1.2]{DLS11a}), there exist $W^{1,2}$ functions $g^{w}_{\ell} \colon \left( 0,1 \right) \to \R^N$ for $\ell = 1,\dots,Q$ such that $\Abs{\frac{d}{dr} g^{w}_{\ell}(r)} \leq \abs{D_{w} f(rw)}$ for a.e. $r \in \left( 0,1 \right)$. Now, fix $t \in \left( a, 1 \right)$. Then, by one-dimensional calculus, we have for $s_0 < s \leq t$ and for every $\ell \in \{ 1,\dots,Q\}$ that
\begin{align*}
\abs{g^{w}_{\ell}(s) - g^{w}_{\ell}(t)}^2 &\le \left(\int_s^t r^{-d-1} \mathrm{d}r \right) \left( \int_{s}^t \left| \frac{d}{dr} g^{w}_{\ell}(r) \right|^2 \, r^{d+1}\mathrm{d}r \right)\\
&\le \frac{s^{-d}}{d}\int_{s_0}^1 \left| D f(rw) \cdot w \right|^2 \, r^{d+1}\mathrm{d}r \\
&= \frac{s^{-d}}{d} \int_{s_0}^{1} \left| D f(rw) \cdot rw \right|^{2} \, r^{d-1} \mathrm{d}r.
\end{align*}
For $t\leq s \leq 1$ the same computation holds true interchanging $t$ and $s$: in this case, we estimate $\frac{t^{-d}}{d} \le \frac{a^{-d}}{d} s^{-d}$. Hence in both cases we have 
\[ \abs{g^{w}_{\ell}(s) - g^{w}_{\ell}(t)}^2 \le C s^{-d} \int_{s_0}^1 \abs{D f(rw) \cdot rw}^2 \, r^{d-1}\mathrm{d}r. \]
Summing over $\ell \in \{1,\dots,Q\}$ and recalling the definition of the metric $\G$ this produces
\[
\G(f(sw), f(tw))^{2} \leq C s^{-d} \int_{s_0}^{1} \left| D f(rw) \cdot rw \right|^{2} \, r^{d-1} \mathrm{d}r \quad \mbox{for every } s \in \left(s_0, 1 \right),
\] 
where $C = C(d,a,Q)$. Multiply by $s^{d-1}$ and integrate in $s$ between $s_0$ and $1$ to obtain
\[ \int_{s_0}^1 \G(f(sw) , f(tw))^2 s^{d-1} \mathrm{d}s \le C \abs{\ln({s_0})} 
 \int_{s_0}^1 \abs{D f(rw) \cdot rw}^2 \, r^{d-1}\mathrm{d}r.\]
Integrating now in $w \in \partial B_1$ gives inequality \eqref{eq:0-homogeneous comparison}.\\ 

We are now ready to prove the proposition. Let $u \in W^{1,2}(\Omega, \A_{Q}(\N))$ be a stationary $Q$-harmonic map, and assume that $B_{1}$ is $(k,\e)$-symmetric for $u$. 
By $(b)$ in Definition \ref{def:quant_sym}, we can fix $\frac{1}{4} \le t \le \frac{1}{\sqrt{2}}$ such that 
$x\in \partial B_t \mapsto u(x)$ is in $W^{1,2}(\partial B_t, \A_{Q}(\N))$ and satisfies $\int_{\partial B_t} \abs{ D_L u}^2 \le C \int_{B_1} \abs{D_Lu}^2$.\\
For a.e. $y' \in B^k_t$ we have that the map $z \mapsto v_{y'}(z):=u(y',z)$ is in $W^{1,2}(B_{t}^{m-k},\A_{Q}(\N))$. Hence, by the scaled version of \eqref{eq:0-homogeneous comparison} with $d = m-k$ we have for any $0<s_0<
\frac14$ that \[ \int_{B^{z}_t\setminus B^z_{s_0}} \G\left( v_{y'}(z) , v_{y'}\left(t \frac{z}{\abs{z}}\right)\right)^2 \le C\abs{\ln(s_0)} \int_{B^{z}_t} \left| D u(y',z) \cdot z \right|^2 \, ,\] where $C = C(m,Q)$. Integrating this now in $y' \in B^k_t$ we obtain  
\[ \int_{B^{y'}_t} \int_{B^{z}_t\setminus B^{z}_{s_0}} \G\left( v_{y'}(z) , v_{y'}\left(t \frac{z}{\abs{z}}\right)\right)^2 \le C \abs{\ln s_0} \int_{B^k_t \times B^{m-k}_t} \left|D u(y,z) \cdot z \right|^2 \, \mathrm{d}y \mathrm{d}z .\]
Adding the scaled version of \eqref{eq:halfpoincare}, since $B^k_t\times B^{m-k}_t \subset B_1$ we obtain 
\begin{align*}
	\int_{B^{y'}_t} &\left( \int_{B_t^y\times B_t^z} \G( u(y,z) , v_{y'}(z))^2 + \int_{B^{z}_t\setminus B^z_{s_0}} \G\left( v_{y'}(z) , v_{y'}\left(t \frac{z}{\abs{z}}\right)\right)^2 \right) \\
	&\qquad \qquad \qquad \qquad \qquad \qquad \le C \left(\int_{B_1} \abs{D_L u}^2  + \abs{\ln (s_0)} \int_{B_1} | D u(x) \cdot x |^2\right).
\end{align*}
Hence there exists $y'_0 \in B^k_t$ such that 
\begin{align*}
\int_{B_t^y\times B_t^z} \G( u(y,z) , v_{y_0'}(z))^2 &+ \int_{B^{z}_t\setminus B^z_{s_0}} \G\left( v_{y_0'}(z) , v_{y_0'}\left(t \frac{z}{\abs{z}}\right)\right)^2 \\
& \le \frac{C}{\Ha^{k}(B^k_t)} \left(\int_{B_1} \abs{D_L u}^2  + \abs{\ln(s_0)} \int_{B_1} \left| D u(x) \cdot x \right|^2\right) \\
&\leq C(1 + \abs{\ln(s_0)}) \e,
\end{align*}
where in the last inequality we have used that the ball $B_{1}$ is, by assumption, $(k,\e)$-symmetric for $u$ together with \eqref{eq:monotonicity identity-integrated lower bound}.

Set $h(x) = h(y,z):=v_{y_0'}\left(t \frac{z}{\abs{z}}\right) \in W^{1,2}(B_{t}, \A_{Q}(\N))$. Note that, by definition, $h$ is homogeneous with respect to $0$. Furthermore, $h$ is $k$-symmetric. An application of the triangle inequality gives
\begin{align*} \int_{B_t^y \times B_t^z} \G( u(y,z), h(x))^2 \le & 2 \int_{B_t^y \times B_t^z} \G( u(y,z), v_{y_0'}(z))^2 + 2 \int_{B_t^y \times (B_t^z\setminus B^z_{s_0}) } \G( v_{y_0'}(z), h(x))^2 \\ &+ 2 \int_{B_t^y \times B_{s_0}^z } \G( v_{y_0'}(z), h(x))^2.\end{align*}
As we have shown above, the first two integrals can be bounded by $C(1+\abs{ \ln(s_0)}) \e$. As for the last integral, we estimate it brutally by fixing a point $p \in \N$ and computing
\[ \int_{B_t^y \times B_{s_0}^z } \G( v_{y_0'}(z), h(x))^2 \le 2 \sup_{x \in B_1} \G(u(x), Q\llbracket p \rrbracket)^{2} \Ha^{k}(B_1^k) \Ha^{m-k}(B_{s_0}^{m-k}) \leq C Q {\rm diam}(\N)^{2} s_{0}^{m-k} \]
Hence choosing $s_0 = \e$ proves the proposition, since we get \[ \int_{B_t} \G( u(x) , h(x) )^{2} \le C \abs{\e \ln(\e)}. \]
\end{proof}

\begin{corollary} \label{stationary_old_new_strata}
Let $u \in W^{1,2}_{loc}(\Omega, \A_{Q}(\N))$ be a stationary $Q$-harmonic map. Then
\[
\cS^{k}_{{\rm [CN]}}(u) \subset \cS^{k}(u).
\]
\end{corollary}

We conclude the section with two propositions about the characterization of the singular set for minimizing and stationary maps. The first one is the following effective version of Corollary \ref{prop: empty n-2 strata}.

\begin{proposition} \label{prop:empty_n-2_eps_strata}
There exists $\e = \e(m,\N,N,Q,\Lambda)$ such that for any map $u$ satisfying Assumption \ref{ass:harmonic} the following holds:
\[
B_{1} \cap \left(\cS^{m-1}(u) \setminus \cS^{m-3}_{\e}(u)\right) = \emptyset.
\] 
\end{proposition}

\begin{proof}
The proof is by contradiction. Assume, therefore, that for every $j \in \mathbb{N}$ there exists $u_{j}$ as in Assumption \ref{ass:harmonic} with a point $x_{j} \in B_{1} \cap \left(\cS^{m-1}(u_j) \setminus \cS^{m-3}_{j^{-1}}(u_j) \right)$. Since $x_{j} \notin \cS^{m-3}_{j^{-1}}(u_j)$, there exists $0 < r_j < 1$ and a linear subspace $L_j \subset \R^m$ with $\dim(L_j) = m-2$ such that
\begin{align} \label{eps_strata_pinch}
\theta_{u_j}(x_j, 2r_j) - \theta_{u_j}(x_j, r_j) &\leq j^{-1}\, , \\ \label{eps_strata_symm}
r_{j}^{2-m} \int_{B_{r_j}(x_j)} \abs{D_{L_j}u_j}^{2} &\leq j^{-1}\, . 
\end{align}
 
As usual, without loss of generality we assume that the $(m-2)$-planes of $j^{-1}$-almost symmetry are a fixed subspace $L$ along the sequence. Set $v_{j}(y) := u_{j}(x_j + r_j y)$, and re-write the equations \eqref{eps_strata_pinch} and \eqref{eps_strata_symm} in terms of $v_j$:
\begin{align} \label{eps_strata_pinch'}
\theta_{v_j}(0,2) - \theta_{v_j}(0,1) &\leq j^{-1} \, , \\ \label{eps_strata_symm'}
\int_{B_{1}} \abs{D_L v_{j}}^{2} &\leq j^{-1} \, .
\end{align}

Now, by an elementary computation it is immediate to see that for every $\rho \in \left( 0, 8 \right)$ one has
\[
\rho^{2-m} \int_{B_{\rho}} \abs{D v_j}^{2} = (\rho r_{j})^{2-m} \int_{B_{\rho r_j}(x_j)} \abs{D u_j}^{2} \leq C_{m} \Lambda.
\]
Hence, by the Compactness Theorem \ref{HM_Q_compactness}, the sequence $\{v_{j}\}_{j \in \mathbb{N}}$ converges up to subsequences in $L^2(B_{8}, \A_{Q}(\R^N))$ and in energy to a $Q$-valued energy minimizing map $v$ for which
\begin{align} \label{eps_strata_pinch_limit}
\theta_{v}(0,2) - \theta_{v}(0,1) &= 0 \, , \\ \label{eps_strata_symm_limit}
\int_{B_1} \abs{D_L v}^{2} &= 0 \, .
\end{align}
In particular, $\left. v \right|_{B_1}$ can be extended to an $(m-2)$-symmetric energy minimizer. This implies that a fortiori $0 \in {\rm reg}_{H}(v)$. Thus, $0 \notin \cS^{m-1}(v_j)$ for $j$ large, which contradicts the fact that $x_{j} \in \cS^{m-1}(u_j)$.
\end{proof}

In the single-valued case $Q = 1$, we have the following result on the quantitative stratification for stationary harmonic maps.

\begin{proposition}\label{prop.empty n-1 strata}
There exists $\e = \e(m, \N)$ such that for any \emph{single-valued} stationary harmonic map $u \in W^{1,2}_{loc}(\Omega, \N)$ the following holds:
\[ \cS^{m-1}(u) \setminus \cS_\e^{m-2}(u) = \emptyset.\]
\end{proposition}

\begin{proof}
Proposition \ref{prop.empty n-1 strata} is a consequence of the inner variation formula.
First we derive a general estimate and show afterwards how it implies the proposition.

Let us consider a single-valued harmonic map $u$ in $B_1$ that satisfies the inner variation formula \eqref{eq:inner_variation_formula} with $Q = 1$. We fix two non-negative, non-increasing functions $\psi, \varphi \in C^1_c\left(\left[ 0, \frac{1}{\sqrt{2}}\right)\right)$ and a $k$-dimensional subspace $L \subset \R^{m}$. After a rotation, we may assume that $L=\{ x_i=0 \colon i = k+1, \dotsc, m\}$. To make the notation a bit simpler we will write $x = (y,z) \in L \times L^\perp$, and by a slight abuse of notation we shall again consider $z = (0,z)$ as a vector in $\R^{m}$.  
Consider the vector field $X(y,z):= \psi(\abs{y})\varphi(\abs{z}) z=\psi \varphi z$. We have $D X = \psi\varphi P^\perp +  \frac{\psi \varphi'}{\abs{z}} z\otimes z + \frac{\psi' \varphi}{\abs{y}} z\otimes y$, where $P^\perp$ denotes the orthogonal projection onto $L^\perp$.  We use this vector field in the inner variation formula \eqref{eq:inner_variation_formula} and obtain
\begin{align*}
0 = \int \Abs{D u}^{2} \left( (m-k) \psi \varphi + \psi \varphi' \abs{z} \right) - 2 \left( \psi \varphi \abs{D_{L^\perp} u}^2 + \psi \varphi' \frac{1}{\abs{z}}\abs{D u \cdot z}^2 + \psi' \varphi \left\langle D u \cdot z, D u \cdot \frac{y}{\abs{y}}\right\rangle \right).
\end{align*}
Observe that
\[(m-k) \psi \varphi \abs{D u}^2 - 2\psi \varphi \abs{D_{L^\perp} u}^2= (m-k-2) \varphi \psi \abs{D u}^2 + 2 \psi \varphi \abs{D_{L} u}^2.\]
Furthermore, we can write $z=x-y$ to estimate
\begin{align*}
\abs{D u \cdot z}^2&\le 2 \abs{D u \cdot x}^2 + 2 \abs{D u \cdot y}^2\\
\left\langle D u \cdot z, D u \cdot \frac{y}{\abs{y}}\right\rangle& \le \frac{1}{2}  \abs{D u \cdot x}^2 + \frac12 \Abs{D u \cdot \frac{y}{\abs{y}}}^2.
\end{align*}
Combining all together, and recalling that $\varphi', \psi' \leq 0$, we obtain the inequality 
\begin{equation}\label{eq:innnervariation_split} \begin{split} &\int - \left((m-k-2) \psi \varphi + \psi \varphi' \abs{z}\right) \abs{D u}^2 \\\le &\int 2\psi \varphi \Abs{D_{L}u}^{2} + \int \frac{-4\psi \varphi'}{\abs{z}} \left(\abs{D u \cdot x}^2 +  \abs{D u \cdot y}^2\right)  - \psi'\varphi \left( \abs{D u \cdot x}^2 + \Abs{D u \cdot \frac{y}{\abs{y}}}^2 \right). \end{split}
\end{equation}

We are ready to prove the proposition. Fix $\e > 0$ to be determined later, and suppose by contradiction that there is a point $x \in \cS^{m-1}(u) \setminus \cS_\e^{m-2}(u)$. Since $x \notin \cS_\e^{m-2}$, there exists $r=r(\e) >0 $ and an $(m-1)$-dimensional subspace $L=L(\e)$ such that $r^{2-m}\int_{B_r(x)} \abs{ D_L u}^2 < \e$ and $\theta(x,2r)-\theta(x,r) < \e$. By translation and scaling, i.e. passing to $T^u_{x,r}$, we may assume that $x=0$ and $r=1$. However, for notational convenience, we will still write $u$ for $T^u_{x,r}$.
After a further rotation we may assume that $L=\cur{x_m=0}$. Now, we have $B_{\frac{1}{2}} \subset B_{\frac{1}{\sqrt{2}}}^{m-1} \times \left(-\frac{1}{\sqrt{2}},\frac{1}{\sqrt{2}} \right) \subset B_1$. Fix a function $\eta \in C^1$ with $\eta' \le0$ and  $\eta(t) =1$ for $t \le \frac12$, $\eta(t) =0$ for $t \ge \frac{1}{\sqrt{2}}$. Set $\varphi = \psi := \eta$ in \eqref{eq:innnervariation_split}. Recall that in our situation $k=m-1$, and thus $-(m-k-2) = 1$. Furthermore, we have $\psi \varphi \ge \mathbf{1}_{B_{\frac12}}$, and $\frac{\abs{4 \psi \varphi'}}{\abs{z}}, \abs{\psi'\varphi}$ are bounded and supported in $B_1$. Hence \eqref{eq:innnervariation_split} reads in our case
\[ \int_{B_{\frac12}}  \abs{D u}^2 \le C \int_{B_1} \left( \abs{x} \abs{D_{r}u}^2 + \abs{D_L u}^2 \right), \]
where $r(x) = r_{0}(x) = \frac{x}{\abs{x}}$. By \eqref{eq:monotonicity identity-integrated lower bound}, we deduce that $\int_{B_{\frac12}}  \abs{D u}^2 \le C \e$. If $\e>0$ is chosen sufficient small, i.e. $C\e  < \e_0$ where $\e_{0} = \e_{0}(m,\N)$ is the threshold in the $\e$-regularity theorem for stationary harmonic maps (cf. \cite{Bethuel,RS}), this allows to infer that $u$ is H\"older continuous in $B_{\frac14}$, and hence $0$ is a regular point. This contradicts the assumption that $0 \in \cS^{m-1}$.
\end{proof}

\begin{remark}
Note that the above proposition could be extended (with exactly the same proof) to the case of stationary $Q$-harmonic maps \emph{if} an $\e$-regularity theorem was available in that case.
\end{remark}

\subsection{Main theorem on the quantitative strata}
Since the relevant terminology has been introduced now, we can finally state the main estimates that we are going to prove on the singular strata. 
\begin{theorem}\label{th_main}
 Given a Dirichlet-minimizing $Q$-valued map $u:\B 2 0 \subseteq \R^m \to \A_{Q}(\N)$ with $\E(u, B_{2}(0)) \leq \Lambda$, let $\cS^k_{\epsilon,r}(u)$ be its quantitative singular strata. Then, if $B_{r}\left( \cS^{k}_{\e,r}(u) \right) := \bigcup_{x \in \cS^{k}_{\e,r}(u)} B_{r}(x)$, we have
 \begin{gather} \label{eq:mink_bounds}
  \Vol\ton{\B r {\cS^k_{\epsilon,r}(u)\cap \B 1 0}}\leq C(m,\N,Q,\Lambda,\epsilon) r^{m-k}\, .
 \end{gather}
 Moreover, $\cS^k_\epsilon(u)$ is $k$-rectifiable for all $\epsilon\geq 0$.
\end{theorem}
\begin{remark}
 This theorem is similar in spirit to \cite[Theorem 1.3]{naber-valtorta:harmonic}.
\end{remark}

Note that this and the $\epsilon$-regularity theorem immediately imply Theorem \ref{th_appetizer} as a corollary. Indeed, we have
\begin{corollary}[{Theorem \ref{th_appetizer}}]
 Given a Dirichlet-minimizing $Q$-valued map $u:\B 2 0 \subseteq \R^m \to \A_{Q}(\N)$ with energy bounded by $\Lambda$, if $B_{r}\left( \sing_{H}(u) \right) := \bigcup_{x \in \sing_{H}(u)} B_{r}(x)$ then we have
 \begin{gather}
  \Vol\ton{\B r {\sing_{H}(u)\cap \B 1 0}}\leq C r^{3}\, ,
 \end{gather}
 where $C = C(m,\N,Q,\Lambda)$. Moreover, $\sing_{H}(u)$ is $(m-3)$-rectifiable.
\end{corollary}
\begin{proof}
 By remark \ref{rem_sing} and proposition \ref{prop:empty_n-2_eps_strata}, there exists an $\epsilon$ such that
 \begin{gather}
  \cS^{m-3}_{\epsilon}(u)\cap \B 1 0=\sing_{H}(u)\cap \B 1 0\, .
 \end{gather}
Thus Theorem \ref{th_main} immediately proves the volume estimates and rectifiability for $\sing_H(u)$.
\end{proof}

We postpone the proof of Theorem \ref{th_main} to Section \ref{sec_cov}, after having discussed a few technical tools needed to complete it.

\section{Quantitative \texorpdfstring{$\e$}{epsilon}-regularity theorems}

In this section we are going to present the proof of a quantitative version of the $\e$-regularity theorem for $Q$-valued minimizers, cf. Theorem \ref{th:quantitative_eps_reg} below, which in turn implies Corollary \ref{corollario_finale_bello}, providing sufficient conditions under which the singular set $\sing_{H}(u)$ is constrained to live in the tubular neighborhood of an affine subspace of $\R^m$ of appropriate dimension. We start with the following definition, analogous to \cite[Definition 4.5]{naber-valtorta:harmonic}.

\begin{definition}
Let $y_{0}, y_{1}, \dots, y_{k}$ be $(k+1)$ points in $B_{1}(0) \subset \R^{m}$, and let $\rho > 0$. We say that these points $\rho$-effectively span a $k$-dimensional affine subspace if 
\begin{equation}
\dist(y_{i}, y_{0} + {\rm span}[y_{1} - y_{0}, \dots, y_{i-1} - y_{0}]) \geq 2\rho \quad \mbox{for every $i = 1,\dots,k$}.
\end{equation}  
A set $F \subset B_{1}(0)$ $\rho$-effectively spans a $k$-dimensional subspace if there exist points $\{ y_{i} \}_{i=0}^{k} \subset F$ which $\rho$-effectively span a $k$-dimensional subspace.
\end{definition}

\begin{remark} \label{rmk:eff_spanning_properties}
It is easy to see that if the points $\{ y_{i} \}_{i=0}^{k}$ $\rho$-effectively span a $k$-dimensional affine subspace then for every point 
\[
x \in y_{0} + {\rm span}[y_{1}-y_{0}, \dots, y_{k}-y_{0}]
\]
there exists a unique set of numbers $\{ \alpha_{i} \}_{i=1}^{k}$ such that
\[
x = y_{0} + \sum_{i=1}^{k} \alpha_{i} (y_{i}-y_{0}), \quad |\alpha_{i}| \leq C(m,\rho) |x-y_{0}|.
\]
Furthermore, the notion of $\rho$-effectively spanning a $k$-dimensional affine subspace passes to the limit: if for every $j \in \mathbb{N}$ the points $\{ y^{j}_{i} \}_{i=0}^{k}$ $\rho$-effectively span a $k$-dimensional subspace and there exist the limits $y_{i} := \lim_{j\to\infty} y_{i}^{j}$, then also the points $\{ y_{i} \}_{i=0}^{k}$ $\rho$-effectively span a $k$-dimensional subspace.
\end{remark}

We can now state the main theorem of this section.

\begin{theorem}\label{th:quantitative_eps_reg}
Let $\e, \rho > 0$ be fixed. There exist $\delta, \overline{r} > 0$, depending on $m,\rho,\Lambda,\e$, with the following property. Let $u \in W^{1,2}(B_{10}(0), \A_{Q}(\N))$ be a stationary $Q$-harmonic map with energy bounded by $\Lambda$, let $r \leq 1$, and let
\[
F := \lbrace y \in B_{r}(0) \, \colon \, \theta(y,4r) - \theta(y,2r) < \delta \rbrace.
\]
If $F$ $(\rho \cdot r)$-effectively spans a $k$-dimensional subspace $L$, then
\begin{equation}
\left( \cS_{\e,r \overline{r}}^{k}(u) \cap B_{\frac{r}{2}}(0) \right) \setminus B_{r \rho}(L) = \emptyset.
\end{equation}
\end{theorem}

\begin{corollary} \label{corollario_finale_bello}
For every $\rho > 0$, there exists $\delta = \delta(m,\N,N,Q,\Lambda,\rho) > 0$ with the following property. Let $u \colon B_{10}(0) \subset \R^{m} \to \A_{Q}(\N)$ be a $W^{1,2}$ map with energy bounded by $\Lambda$, and let $r \leq 1$. 
\begin{itemize}
\item[$(i)$] In case $u$ is energy minimizing, if there exist $m-2$ points $\lbrace y_{i} \rbrace_{i=0}^{m-3} \subset B_{r}(0)$ which $(\rho \cdot r)$-effectively span an $(m-3)$-dimensional affine subspace $L \subset \R^{m}$ and such that
\[
\theta(y_{i}, 4r) - \theta(y_i, 2r) < \delta \quad \mbox{for every } i=0,\dots,m-3,
\]
then
\[
\left( {\rm sing}_{H}(u) \cap B_{\frac{r}{2}}(0) \right) \setminus B_{\rho r}(L) = \emptyset;
\]
\item[$(ii)$] in case $u$ is single-valued and stationary harmonic, if there exist $m-1$ points $\lbrace y_{i} \rbrace_{i=0}^{m-2} \subset B_{r}(0)$ which $(\rho \cdot r)$-effectively span an $(m-2)$-dimensional affine subspace $L \subset \R^{m}$ and such that
\[
\theta(y_{i}, 4r) - \theta(y_i, 2r) < \delta \quad \mbox{for every } i=0,\dots,m-2,
\]
then
\[
\left( {\rm sing}_{H}(u) \cap B_{\frac{r}{2}}(0) \right) \setminus B_{\rho r}(L) = \emptyset.
\]
In particular, if $m=3$ and $u$ is a $Q$-valued energy minimizer, and if $\theta(y_0,4r) - \theta(y_0,2r) < \delta$ then
\[
B_{\frac{r}{2}}(y_0) \setminus B_{\rho r}(y_0) \subset {\rm reg}_{H}(u)\, .
\]
The same holds if $m=2$ and $u$ is single-valued and stationary.
\end{itemize}
\end{corollary}

\begin{proof}[Proof of Corollary \ref{corollario_finale_bello}]
It follows immediately from Theorem \ref{th:quantitative_eps_reg} and Propositions \ref{prop:empty_n-2_eps_strata} for the minimizing case and \ref{prop.empty n-1 strata} for the stationary harmonic case.
\end{proof}

For the proof of Theorem \ref{th:quantitative_eps_reg} we will need the following lemma.

\begin{lemma} \label{lem:pinching_control1}
Let $u \in W^{1,2}(B_{10}(0), \A_{Q}(\N))$ be a stationary $Q$-harmonic map, and let $r \leq 1$. If $\{y_{i}\}_{i=0}^{k} \subset B_{r}(0)$ $(\rho\cdot r)$-effectively span a $k$-dimensional affine subspace $L \subset \R^{m}$, then
\begin{equation}\label{eq:pinching_control1}
r^{-m} \int_{B_{r}(0)} \left( r^{2} |D_{\hat{L}}u(z)|^{2} + |D_{v}u(z)|^{2} \right) \, \mathrm{d}z \leq C(m,\rho) \sum_{i=0}^{k} \left( \theta(y_{i}, 4r) - \theta(y_{i},2r) \right),
\end{equation} 
where $\hat{L}$ is the linear part of $L$ and $v$ is the vector field $v(z) := D \left(\frac{1}{2}  \dist^{2}(z,L) \right)$.
\end{lemma}

\begin{proof}
It is an immediate consequence of \eqref{eq:monotonicity identity-integrated lower bound} that there exists a constant $C = C(m)$ such that
\begin{equation} \label{eq:control_radial_derivative}
r^{-m} \int_{B_{r}(x)} |D u(z) \cdot (z-x)|^{2} \, \mathrm{d}z \leq \int_{B_{r}(x)} \frac{\abs{z-x}}{r^{m-1}} \left| D u(z) \cdot \frac{z-x}{\abs{z-x}}  \right|^{2} \, \mathrm{d}z \leq C(m) \left( \theta(x,2r) - \theta(x,r) \right)
\end{equation}
whenever $B_{2r}(x) \subset B_{10}(0)$. Now, assume that $y_{0}, y_{1}, \dots, y_{k}$ are as in the statement, and observe that for every unit vector $e$ in the linear part $\hat{L}$ of $L$ there exists a unique set of numbers $\{ \alpha_{i} \}_{i=1}^{k}$ such that
\[
e = r^{-1} \sum_{i=1}^{k} \alpha_{i} (y_{i} - y_{0}), \quad |\alpha_{i}| \leq C(m,\rho).
\]
Hence, we get
\[
\begin{split}
r^{2-m} \int_{B_{r}(0)} |D_{e}u(z)|^{2} \, \mathrm{d}z &\leq C(m, \rho) r^{-m} \sum_{i=1}^{k} \int_{B_{r}(0)} \left| D u(z) \cdot (y_{i} - y_{0}) \right|^{2}\, \mathrm{d}z \\
&\leq C(m,\rho) r^{-m} \sum_{i=0}^{k} \int_{B_{r}(0)} \left| D u(z) \cdot (z - y_{i}) \right|^{2} \, \mathrm{d}z \\
&\leq C(m,\rho) r^{-m} \sum_{i=0}^{k} \int_{B_{2r}(y_i)} \left| D u(z) \cdot (z - y_{i}) \right|^{2} \, \mathrm{d}z \\
&\overset{\eqref{eq:control_radial_derivative}}{\leq} C(m,\rho) \sum_{i=0}^{k} \left( \theta(y_{i},4r) - \theta(y_{i},2r) \right).
\end{split}
\]

Summing over an orthonormal basis $e_{1},\dots,e_{k}$ of $\hat{L}$ produces
\begin{equation}\label{first_piece}
r^{2-m} \int_{B_{r}(0)} \left| D_{\hat{L}}u(z) \right|^{2} \, \mathrm{d}z \leq C(m,\rho) \sum_{i=0}^{k} \left( \theta(y_i, 4r) - \theta(y_i, 2r) \right).
\end{equation}

As for the second term, let $z \in B_{r}(0)$, and let $\pi := \pi_{L}(z)$ be the orthogonal projection of $z$ onto $L$. Of course,
\[
v(z) := D \left( \frac{1}{2} \dist^{2}(z,L) \right) = z - \pi.
\]
On the other hand, we have as usual that
\[
\pi = y_{0} + \sum_{i=1}^{k} \alpha_{i} (y_{i} - y_{0}), \quad |\alpha_{i}| \leq C(m,\rho) |\pi - y_0| \leq C(m,\rho)r,
\]
and thus
\[
v(z) = z - \left(  y_{0} + \sum_{i=1}^{k} \alpha_{i} (y_{i} - y_{0}) \right).
\]
Arguing as above, one concludes that also
\begin{equation} \label{second_piece}
r^{-m} \int_{B_{r}(0)} \left| D_{v}u(z) \right|^{2} \, \mathrm{d}z \leq C(m,\rho) \sum_{i=0}^{k} \left( \theta(y_i, 4r) - \theta(y_i, 2r) \right),
\end{equation}
which together with \eqref{first_piece} completes the proof of \eqref{eq:control_radial_derivative}.

\end{proof}

For $k = m-2$, the conclusions of the previous lemma can be improved using again the inner variation formula.

\begin{lemma} \label{lem:pinching_control2}
Let $u \in W^{1,2}(B_{10}(0), \A_{Q}(\N))$ be a stationary $Q$-harmonic map, and let $r \leq 1$. If $\{ y_{i} \}_{i=0}^{m-2} \subset B_{r}(0)$ $(\rho \cdot r)$-effectively span an $(m-2)$-dimensional affine subspace $L \subset \R^{m}$, then
\begin{equation} \label{eq:pinching_control2}
r^{-m} \int_{B_{r}(0)} \left( r^{2} |D_{\hat{L}} u|^{2} + |D_{\hat{L}^{\perp}}u|^{2} |v|^{2} \right) \leq C(m,\rho) \sum_{i=0}^{m-2} \left( \theta(y_{i}, 8r) - \theta(y_{i}, 4r) \right),
\end{equation}
where $\hat{L}$ is the linear part of $L$, $\hat{L}^{\perp}$ is its orthogonal complement in $\R^m$ and $v$ is the vector field $v(x) := D \left( \frac{1}{2} \dist^{2}(x,L) \right)$.
\end{lemma}

\begin{proof}
The proof is very similar to the one of Proposition \ref{prop.empty n-1 strata}: also in this case, we will make use of the stationary equation with a suitable choice of the vector field $X$. Without loss of generality, we can assume that $r = 1$. Furthermore, modulo translations and rotations we can assume that $L = \lbrace x_{i} = 0 \, \colon \, i = m-1,m \rbrace$.
As usual, coordinates on $L$ and $L^{\perp}$ will be denoted by $y$ and $z$ respectively, and in order to simplify our notation the vectors $(y,0)$ and $(0,z)$ in $L \times L^{\perp}$ will be simply denoted by $y$ and $z$. Observe that under these assumptions one has $v(x) = z$ for every $x = (y,z) \in B_{1}$. Now, let $\psi = \psi(y)$ be a cut-off function of the variable $y \in L$, with $\psi \equiv 1$ in $B_{1}^{m-2}$, $\spt(\psi) \subset B_{2}^{m-2}$ and $|D \psi| \leq 1$. Let also $\varphi(t) := \max\lbrace 1-t, 0 \rbrace$, and consider the vector field $X(y,z) := \psi(y) \varphi(|z|^{2}) z = \psi \varphi z$.
We can immediately compute $D X = \psi \varphi P^{\perp} + \varphi z \otimes D \psi + 2 \psi \varphi' z \otimes z$. With this choice of $X$, the inner variation formula \eqref{eq:inner_variation_formula} reads
\[
\begin{split}
0 &= \int |D u|^{2} \left( 2 \psi \varphi + 2 \psi \varphi' |z|^{2} \right) - 2 \left( \psi \varphi |D_{L^{\perp}}u|^{2} + \varphi \langle D u \cdot z , D u \cdot D \psi \rangle + 2 \psi \varphi' |D u \cdot z|^{2} \right)\\
&= 2 \int  \psi \varphi |D_{L}u|^{2} + \psi \varphi' |D u|^{2} |z|^{2} - \left( \varphi \langle D u \cdot z, D u \cdot D \psi \rangle + 2 \psi \varphi' |D u \cdot z|^{2} \right).
\end{split}
\]

In particular, since $\varphi'(|z|^{2}) = - \chi_{\lbrace |z| \leq 1 \rbrace}$, $\left. \psi \right|_{\lbrace |y| \leq 1 \rbrace} \equiv 1$ and $B_{1} \subset B_{1}^{m-2} \times B_{1}^{2} \subset B_{2}$, we immediately deduce
\[
\int_{B_{1}} |D u|^{2} |z|^{2} \leq C \int_{B_2} \left( |D_{L}u|^{2} + |D u \cdot z|^{2} \right).
\]

The estimate \eqref{eq:pinching_control2} then follows from Lemma \ref{lem:pinching_control1}.
\end{proof}

We are now ready to prove Theorem \ref{th:quantitative_eps_reg}.

\begin{proof}[Proof of Theorem \ref{th:quantitative_eps_reg}]
Since the statement is scale-invariant, there is no loss of generality in proving it only in the case $r=1$. Let $\{ y_{i} \}_{i=0}^{k} \subset F$ $\rho$-effectively span the $k$-dimensional subspace $L$, and let $x$ be any point in $B_{\frac{1}{2}}(0) \setminus B_{\rho}(L)$. The goal is to prove that $x \notin \cS^{k}_{\e,\overline{r}}(u)$ for some $\overline{r} > 0$, and thus that there exists $\overline{r} > 0$ and a radius $r_x\in [\overline{r} ,1)$ such that the ball $B_{r_{x}}(x)$ is $(k+1,\e)$-symmetric for $u$. Let $0 < \delta \ll 1$ to be chosen later. Since $x \in B_{\frac{1}{2}}(0)$, $B_{\sigma}(x) \subset B_{1}(0)$ for every $0 < \sigma < \frac12$. Hence, we deduce from Lemma \ref{lem:pinching_control1} that
\[
\int_{B_{\sigma}(x)} |D_{\hat{L}}u|^{2} \leq C(m,\rho) \delta 
\]
for any such $\sigma$. In order to gain another direction along which the energy is small, we let $v(z) := D \left( \frac12 \dist^{2}(z,L)\right)$, and we set $e := \frac{v(x)}{|v(x)|}$. Note that $|v(x)| = \dist(x,L) \geq \rho$. Again by Lemma \ref{lem:pinching_control1} and by the monotonicity of the function $r \mapsto \E(u, B_{r}(x))$, we have
\[
\begin{split}
\int_{B_{\sigma}(x)} |D_{e} u|^{2} &\leq \rho^{-2} \int_{B_{\sigma}(x)} |D u(z) \cdot v(x)|^{2} \\
 &\leq 2\rho^{-2} \left( \int_{B_{\sigma}(x)} |D u(z) \cdot v(z)|^{2} + \int_{B_{\sigma}(x)} |D u(z) \cdot (v(z) - v(x))|^{2} \right) \\
 &\leq C \int_{B_{1}(0)} |D_{v}u|^{2} + C \sigma^{2} \int_{B_{\sigma}(x)} |D u|^{2} \\
 &\leq C \delta + C \Lambda \sigma^{m}, 
\end{split}
\]
where $C = C(m,\rho)$. Hence, if $V := \hat{L} \oplus {\rm span}(e)$ then
\begin{equation} \label{higher_symmetry}
\int_{B_{\sigma}(x)} |D_{V}u|^{2} \leq C \delta + C \Lambda \sigma^{m}
\end{equation}
for every $0 < \sigma < \frac12$. Note that $\dim(V) = k+1$. 

Fix now $\e > 0$, and let $\overline{\sigma} = \overline{\sigma}(m,\rho,\Lambda,\e) < \frac12$ be such that $C \Lambda \overline{\sigma}^{2} \leq \frac{\e}{2}$. We claim that for any $0 < \tau \ll 1$ there exists $\tau \overline{\sigma} \leq r_{x} < \overline{\sigma}$ such that
\begin{equation} \label{good_radius}
\theta(x, 2r_{x}) - \theta(x, r_{x}) \leq \frac{2c_{1}(m) \Lambda}{- \log_{2}(2\tau)} \, .
\end{equation}
Indeed, otherwise for any integer $M \in \left( \frac{3}{4} \log_{2}\left( \frac{1}{2\tau}\right), \log_{2}\left( \frac{1}{2\tau} \right) \right)$ we would get
\[
c_{1}(m)\Lambda \geq \theta(x,\overline{\sigma}) \geq \sum_{i=0}^{M} \theta(x, 2^{-i} \overline{\sigma}) - \theta(x,2^{-(i+1)} \overline{\sigma}) \geq M \frac{2c_{1}(m) \Lambda}{- \log_{2}(2\tau)} \geq \frac{3}{2} c_{1}(m) \Lambda , 
\]
which is impossible. Hence, if we fix $\tau = \tau(m,\Lambda,\e)$ so small that $\frac{2 c_{1}(m) \Lambda}{- \log_{2}(2\tau)} \leq \e$, the above argument allows to conclude that if we set $\overline{r} := \tau \overline{\sigma}$ then there is a radius $r_{x} \in \left( \overline{r}, \overline{\sigma} \right)$ such that
\begin{equation} \label{conclusion:homo}
\theta(x,2r_{x}) - \theta(x,r_x) \leq \e.
\end{equation}
Furthermore, formula \eqref{higher_symmetry} with $r_{x}$ in place of $\sigma$ implies that
\begin{equation} \label{conclusion:sym}
r_{x}^{2-m} \int_{B_{r_x}(x)} |D_{V}u|^{2} \leq C \delta \overline{r}^{2-m} + C \Lambda \overline{\sigma}^{2}.
\end{equation}

We can finally chose $\delta = \delta(m,\rho,\Lambda,\e)$ such that $C \delta \overline{r}^{2-m} \leq \frac{\e}{2}$. From equations \eqref{conclusion:homo} and \eqref{conclusion:sym} we infer that $B_{r_x}(x)$ is $(k+1,\e)$-symmetric for $u$.  

\end{proof}

We conclude the section with the following proposition, according to which if the mollified energy is pinched enough at $k$ points spanning a $k$-plane $L$, then it is almost constant along this $L$.

\begin{proposition}\label{prop_unipinch}
Let $u$ satisfy Assumption \ref{ass:harmonic}. Let $0 < \rho < 1$ and $\eta > 0$ be fixed, and assume that $\theta(y,8) \leq E$ for every $y \in B_{1}(0)$. There exists $\delta_{0} = \delta_{0}(m,\N,N,Q,\Lambda, \rho, \eta) > 0$ such that if the set $F := \lbrace y \in B_{1}(0) \, \colon \, \theta(y,\rho) > E - \delta_0 \rbrace$ $(2\rho)$-effectively spans a $k$-dimensional affine subspace $L \subset \R^m$ then
\[
\abs{\theta(x,\rho)-E} < \eta \quad \mbox{for every } x \in L \cap B_{1}(0).
\]
\end{proposition}

\begin{proof}
The proof is by contradiction. Assume that there are $0 < \rho_0 < 1$, $\eta_0 > 0$ and a sequence $u_{i}$ of maps satisfying Assumptions \ref{ass:harmonic} and the condition $\theta_{u_i}(y,8) \leq E$ everywhere in $B_1$, and with the property that for every $i \in \mathbb{N}$ there are points $\lbrace y^{i}_{j} \rbrace_{j=0}^{k} \subset B_{1}(0)$ with $\theta_{u_i}(y^{i}_{j}, \rho_0) > E - i^{-1}$ $(2\rho_0)$-effectively spanning a $k$-dimensional affine subspace $L_i \subset \R^m$ but with $\theta(x_i,\rho_0) \leq E - \eta_0$ for some $x_i \in L \cap B_{1}(0)$. As usual, without loss of generality we can assume that the subspace $L = L_i$ is fixed along the sequence. By the usual compactness for energy minimizers, modulo passing to a subsequence (not relabeled) the $u_i$'s converge in $L^2$ and in energy to a minimizer $u$. Up to further extracting another subsequence, we can also assume that $y^{i}_{j} \to y_{j}$ and $x_i \to x$. By Remark \ref{rmk:eff_spanning_properties}, also the $y_j$'s $(2\rho_0)$-effectively span $L$. Moreover, $\theta_{u}(y_j, \rho_0) \geq E$, and thus $\theta_{u}(y_j,8) - \theta_{u}(y_j, \rho_0) \leq 0$. By monotonicity, then it has to be 
\[
\theta_{u}(y_j,8) - \theta_{u}(y_j, \rho_0) = 0,
\] 
and hence, by Lemma \ref{lem:pinching_control1}, $u$ is invariant along $L$ in $B_{2}(0)$. Since $\theta_{u}(y_j, \rho_0) = E$, it has to be $\theta_u(y,\rho_0) =E$ everywhere on $L \cap B_{1}(0)$, which contradicts the existence of $x$.
\end{proof}

\section{Reifenberg theorem}
This section is dedicated to Reifenberg-type results needed for the proof of the main theorem. The results will only be quoted without proof, and they are in some sense a quantitative generalization of Reifenberg's topological disk theorem (see \cite{reif_orig}). Many generalizations of this landmark theorem are available in literature, we limit ourselves to citing \cite{toro,davidtoro} among the various present. Here we will need two versions of this theorem originally proved in \cite{naber-valtorta:harmonic}. 

Before quoting the theorems, we need the following definition of the the so-called Jones' $\beta_2$ numbers. 
\begin{definition}\label{definition_beta}
Given a positive Borel measure $\mu$ defined in $\R^m$, for all positive radii $r>0$ and dimensions $k\in \Na$, we define
\begin{gather}\label{eq_definition_beta}
 D_{\mu}^k (x,r) := \min\cur{\int_{\B r x}\frac{\dist^2(y,V)}{r^2}\,\frac{\mathrm{d}\mu(y)}{r^k} \,  \colon \, \mbox{$V \subset \R^m$ is an affine subspace with } \operatorname{dim}(V)=k}\, .
\end{gather}
Usually in literature this quantity is referred to as Jones' $\beta$-2 number $\beta_{2,\mu}^k(x,r)^2$. 
\end{definition}

$D$ captures in a scale invariant way the distance between the support of $\mu$ and some $k$-dimensional subspace $V$. Indeed, the factor $r^{-2}$ in the distance term makes the integrand scale-invariant, while $r^{-k}\mu$ is scale invariant if we assume that $\mu$ is Ahlfors upper $k$-regular, in the sense $\mu(\B r x)\leq Cr^k$ for some constant $C$. For example, this is the case if $\mu$ is the $k$-dimensional Hausdorff measure on a $k$-dimensional subspace $V \subset \R^m$. 

Here we mention two easy and crucial properties of $D$. 
\begin{lemma}[Bounds on $D$] Given two measures $\mu,\mu'$ such that $\mu' \leq\mu $, for all $x,r$ and $k\in \Na$ we can bound
 \begin{gather}
  D^k_{\mu'} (x,r)\leq D^k_{\mu} (x,r)\, .
 \end{gather}
Also, for all $x,y,r$ such that $\abs{x-y}\leq r$:
\begin{gather}\label{eq_beta_rough}
 D^k_{\mu}(x,r) \leq 2^{k+2} D^k_{\mu}(y,2r)\, .
\end{gather}
\end{lemma}
\begin{proof} 
 The proof follows immediately from the definition.
\end{proof}

\subsection{Quantitative Reifenberg Theorems}
Assuming a sort of integral Carleson-type condition on the $D$ numbers, we can obtain uniform scale invariant properties on the measure $\mu$. For the reader's convenience, here we quote two key theorems that we are going to use in order to get the final estimates on the singular set of $Q$-valued minimizers. The first one is about upper Ahlfors bounds for discrete measures, and is quoted from \cite[Theorem 3.4]{naber-valtorta:harmonic}. This theorem is enough for our purposes, but we mention that some generalizations have been obtained in \cite{ENV_Reif}. The second important theorem is about rectifiability properties for general $\mu$, and is quoted from \cite[Theorem 1.1]{AzzTol}.

\begin{theorem}\cite[Theorem 3.4]{naber-valtorta:harmonic}\label{th_disc_reif}
For some constants $\delta_R(m)$ and $C_R(m)$ depending only on the dimension $m$, the following holds. Let $\{\B {r_x/10}{x}\}_{x\in \cD }\subseteq \B 3 0 \subset \R^m$ be a collection of pairwise disjoint balls with their centers $x\in \B 1 0$, and let $\mu\equiv \sum_{x\in \cD} r_x^k \delta_{x}$ be the associated measure. Assume that for each $B_r(x)\subseteq B_2$
\begin{gather}\label{eq_reif_ass}
 \int_{\B r x }\ton{\int_0^r D^k_\mu(y,s) \,{\frac{\mathrm{d}s}{s}}}\, \mathrm{d}\mu(y)<\delta_R^2 r^{k}\, .
\end{gather}
Then, we have the uniform estimate
\begin{gather}
\sum_{x\in \cD} r_x^k<C_R(m)\, .
\end{gather}
\end{theorem}

Condition \eqref{eq_reif_ass} prescribes some integral Carleson-type control over the quantity $D(x,r)$. If the measure $\mu$ is the $k$-dimensional Hausdorff measure restricted to some $S$, this bound is enough to guarantee also the rectifiability of $S$, as seen in the following theorem. Note that in \cite{AzzTol} the theorem is presented in a more general form.

\begin{theorem}[{\cite[Corollary 1.3]{AzzTol}}] \label{th_azzamtolsa}
 Given a Borel measurable subset $S$ of $\R^m$, let $\mu:=\cH^k\mres S$ be the $k$-dimensional Hausdorff measure restricted to $S$. The set $S$ is countably $k$-rectifiable if and only if 
 \begin{gather}
  \int_0^1 D^k_{\mu}(x,s)\frac{\mathrm{d}s}{s}<\infty\, \quad \mbox{for $\mu$-a.e. x }.
 \end{gather}
\end{theorem}

\section{Best approximating plane}
In this section, we record the main technical lemma needed for the final proof of Theorem \ref{th_main}.  Although several technical points need to be addressed, this lemma contains most of the important estimates in the paper and provides an estimate on the $D$ numbers using the normalized energy $\theta(x,r)$. 

The basic ideas behind the estimates in this section are similar to the ones in \cite[Theorem 7.1]{naber-valtorta:harmonic}, however the new definition of $(k,\epsilon)$-symmetries allows for more quantitative and easier proofs. 

For any $f\in W^{1,2}(\Omega, \A_{Q}(\N))$, and for all $\B r x \subseteq \Omega$, we introduce the following quantity 
\begin{gather}
 \P_f(x,r):= r^{-m}\int_{\B r x} \abs{Df(y) \cdot (y-x)}^2 \, \mathrm{d}y \, .
\end{gather}

Note that in the case $u$ is a Dirichlet minimizing $Q$-valued harmonic map we have by \eqref{eq:control_radial_derivative}:
\begin{gather}\label{eq_Pvstheta}
 \P_u(x,r)\leq C(m)\qua{\theta(x,2r)-\theta(x,r)}\, .
\end{gather}
However, here we carry out the estimates in a very general setting, and we will exploit this bound only at the very last step in our main proof.

\begin{theorem}\label{th_best_pinch}
Let $u\in W^{1,2}(\B 2 0 , \A_{Q}(\N))$, and fix $\epsilon>0$, $0<r\leq 1$ and some $x\in \B 1 0$. Let also $\mu$ be any positive Radon measure supported on $\B 1 0$. Assuming that 
\begin{gather}\label{eq_nablaV_lower_bound}
 \inf\cur{r^{2-m} \int_{\B r x} \abs{D_V u}^2 \, \colon \, \mbox{$V \subset \R^m$ linear with } \operatorname{dim}(V)=k+1} \geq \epsilon\, ,
\end{gather}
we conclude 
\begin{gather}\label{eq_best_pinch_1}
 D_\mu^k(x,r)\leq \frac{(m-k)(k+1)2^m}{\epsilon r^k} \int_{\B {r}{x}} \P_u(y,2r)\,\mathrm{d}\mu(y)\, .
\end{gather}
\end{theorem}
\begin{remark}
We remark that \eqref{eq_best_pinch_1} does not change if $\mu$ is multiplied by a positive constant, thus for convenience for the rest of this section we are going to assume without loss of generality that $\mu$ is a probability measure. Moreover, we can also assume without loss of generality that $x=0$ and $r=1$.
\end{remark}

Note that for this theorem we will not exploit any property specific to Dirichlet-minimizers. For future convenience, we record a simple corollary that rephrases the previous theorem with the language of Dirichlet-minimizers and quantitative stratification.
\begin{corollary}\label{cor_best_pinch}
Under Assumption \ref{ass:harmonic}, fix $\epsilon>0$, $0<r\leq 1$ and some $x\in \B 1 0$. Let also $\mu$ be any positive Radon measure supported on $\B 1 0$. Assuming that $\B r x$ is $(k,\epsilon)$-symmetric but NOT $(k+1,\epsilon)$-symmetric, we conclude 
\begin{gather}\label{eq_best_pinch}
 D_\mu^k(x,r)\leq \frac{C(m)}{\epsilon r^k} \int_{\B {r}{x}} \qua{\theta(y,4r)-\theta(y,2r)}\,\mathrm{d}\mu(y)\, .
\end{gather}
\end{corollary}
\begin{proof}
 The proof follows immediately from the definition of $(k+1,\epsilon)$-symmetric and the bound in \eqref{eq_Pvstheta}.
\end{proof}

\subsection{Properties of the best approximating plane}
For fixed $k$, and given any probability measure $\mu$, for all $(x,r)$ we set $V(x,r)$ to be the $k$-dimensional affine subspace minimizing
\begin{gather}
 \int_{\B r x} \dist^{2}(y,V)\, \mathrm{d}\mu(y)\,,
\end{gather}
so that, in particular, 
\begin{gather}
D^{k}_{\mu}(x,r) = r^{-(k+2)} \int_{B_{r}(x)} \dist^{2}(y, V(x,r))\, \mathrm{d}\mu(y)\,.
\end{gather}
Since in this section we focus on $x=0$ and $r=1$, we will in fact mostly consider only the $k$-dimensional subspace $V(0,1)$.

First of all, note that necessarily $V(x,r)$ will pass through the center of mass of $\mu$ in $B_{r}(x)$, defined as
\begin{gather}
 x_m(\mu,x,r)=x_m := \int_{\B r x} x\, \mathrm{d}\mu(x)\, .
\end{gather}

It will be convenient to phrase some of the estimates needed for theorem \ref{th_best_pinch} in terms of a suitable quadratic form on $\R^m$, defined as
\begin{gather}
 R(w) :=  \int_{\B 1 0} \abs{\ps{x-x_{m}}{w}}^2\, \mathrm{d}\mu(x)\, .
\end{gather}
By standard linear algebra, there exists an orthonormal basis $\cur{e_1,\cdots, e_m}$ of eigenvectors for $R$ with non-negative eigenvalues $\lambda_1,\cdots,\lambda_m$, which we will take for convenience in decreasing order. Note that by the variational characterization of $\lambda_k$ we have that
\begin{gather}
  e_k\in \operatorname{argmax}\cur{ \int_{B_{1}(0)} \abs{\ps{x-x_{m}}{e}}^2\, \mathrm{d}\mu(x) \ \ s.t. \ \ \abs{e}^2=1 \ \ \text{and} \ \ \ps{e}{e_i}=0 \ \ \forall i\leq k   }\, ,\\
 \lambda_{k}=\int_{B_{1}(0)} \abs{\ps{x-x_{m}}{e_k}}^2\, \mathrm{d}\mu(x)\, , \label{eq_lambda_k}
\end{gather}
and so
\begin{gather}\label{eq_beta_V}
 D^k_\mu (0,1) = \int_{\B 1 0} \dist^2(x,V(0,1))\,\mathrm{d}\mu(x) = \sum_{i=k+1}^m \lambda_i\, .
\end{gather}
Indeed, by minimality of $V$, $V(0,1)=x_m + \operatorname{span}\left[e_1,\cdots,e_k\right]$, and thus
\begin{gather}
 \int_{\B 1 0} \dist^2(x,V(0,1))\,\mathrm{d}\mu(x) = \sum_{i=k+1}^m \int_{\B 1 0} \abs{\ps{x-x_{m}}{e_i}}^2\, \mathrm{d}\mu(x) =\sum_{i=k+1}^m \lambda_i\, .
\end{gather}

Using simple geometry, it is possible to prove that for \textit{any} map $f\in W^{1,2}$ we have the following estimate involving $\lambda_k$ and $\P_f$. 
\begin{lemma}\label{lemma_best_V}
Let $f = \sum_{\ell=1}^{Q} \llbracket f_{\ell} \rrbracket  \in W^{1,2}(\B {3r} x, \A_{Q}(\N))$, and let $\mu$ be a probability measure on $\B r x$. Then
\begin{align}
\lambda_k \int_{\B r x} \abs{Df(z) \cdot e_{k}}^2\,\mathrm{d}z \leq 2^m\int_{\B r x} \P_f(y,2r)\, \mathrm{d}\mu(y)\, \quad \mbox{for every } k=1,\dots,m\,.
\end{align}
\end{lemma}
\begin{proof}
For simplicity, we assume $x=0$ and $r=1$. Moreover, note that evidently we can assume $\lambda_k>0$, otherwise there is nothing to prove. Fix some $z\in \B 1 0$. By definition of eigenvectors $e_k$, we have for every $\ell \in \{ 1,\dots,Q\}$ that
\begin{gather}
 \int_{\B 1 0} \ps{x-x_m}{e_k} \left( Df_{\ell}(z) \cdot (x - x_{m}) \right)  \, \mathrm{d}\mu(x) = \lambda_k Df_{\ell}(z) \cdot e_{k} \, .
\end{gather}
By definition of center of mass, we can write
\begin{gather}
 \int_{\B 1 0}  \ps{x-x_m}{e_k}(z-x_m)\ \mathrm{d}\mu(x) = 0\, ,  
\end{gather}
and so 
\begin{gather}
\lambda_k Df_{\ell}(z) \cdot e_{k} = \int_{\B 1 0} \ps{x-x_m}{e_k} \left( Df_{\ell}(z) \cdot (x-z) \right)  \,\mathrm{d}\mu(x) \, .
\end{gather}
By Cauchy-Schwartz and by \eqref{eq_lambda_k}, we have
\begin{gather}
 \lambda_k^2\abs{Df_{\ell}(z) \cdot e_{k}}^2 \leq \lambda_k\int \abs{Df_{\ell}(z) \cdot (x-z)}^2\,\mathrm{d}\mu(x) \, ,
\end{gather}
and thus, summing over $\ell$, 
\begin{gather}
\lambda_{k} \abs{Df(z) \cdot e_{k}}^{2} \leq \int \abs{Df(z) \cdot (x-z)}^{2} \, \mathrm{d}\mu(x)\,.
\end{gather}

Taking the integral of this inequality in $\B 1 0$ with respect to the volume measure in $z$, we obtain the estimate 
\begin{equation}\label{eq_k_vs_P}
\begin{split}
\lambda_k &\int_{\B 1 0}\abs{Df(z) \cdot e_{k}}^2\,\mathrm{d}z\leq  \iint_{\B 1 0 \times \B 1 0} \abs{Df(z) \cdot (x-z)}^2\,\mathrm{d}z\,\mathrm{d}\mu(x) \\
&\leq \int_{\B 1 0} \int_{\B 2 x} \abs{Df(z) \cdot (x-z)}^2\,\mathrm{d}z\,\mathrm{d}\mu(x)\leq 2^m\int_{\B 1 0} \P_f(x,2)\ \mathrm{d}\mu(x)\, .
\end{split}
\end{equation}
\end{proof}

From this proposition, the proof of Theorem \ref{th_best_pinch} follows as a simple corollary.
\begin{proof}[Proof of theorem \ref{th_best_pinch}]
 As before, we assume without loss of generality that $x=0$ and $r=1$. Moreover, by \eqref{eq_beta_V} it is sufficient to prove that
 \begin{gather}
  \lambda_{k+1}\leq \frac{C(m)}{\epsilon} \int_{\B {1}{0}} \P_u(y,2)\,\mathrm{d}\mu(y)\, .  
 \end{gather}
By the previous lemma, we have
\begin{gather}
 \lambda_{k+1} \sum_{j=1}^{k+1} \int_{\B 1 0} \abs{Du \cdot e_{j}}^2\leq \sum_{j=1}^{k+1} \lambda_j \int_{\B 1 0} \abs{Du \cdot e_{j}}^2 \leq C(m) \int_{\B 1 0} \P_u(x,2)\, \mathrm{d}\mu(x)\, .
\end{gather}
By the lower bound in \eqref{eq_nablaV_lower_bound}, we must have
\begin{gather}
 \sum_{j=1}^{k+1} \int_{\B 1 0} \abs{Du \cdot e_{j}}^2\geq \epsilon\, ,
\end{gather}
and this concludes the proof.
\end{proof}

\section{Proof of the main theorem via covering arguments}\label{sec_cov}
This section is dedicated to the proof of the Theorem \ref{th_main}. We split it into two pieces, one containing the uniform Minkowski bounds and one with the rectifiability part. Once the Minkowski bounds are obtained, the rectifiability is almost an immediate corollary.

The Minkowski bounds will be obtained with a covering argument similar to the one in \cite{NV16}. 

\begin{proposition}\label{prop_cover}
There exist a small constant $\delta=\delta(m,\N,Q,\Lambda,\epsilon)>0$ and $C_{III}(m)$ such that the following holds. Let $u$ satisfy assumption \ref{ass:harmonic}, let $\epsilon>0$, $p \in B_{1}(0)$, and $0<r\leq R\, , \ 0<R\leq 1$ be chosen in an arbitrary fashion. For any subset $\cS \subseteq \cS^{k}_{\e,\delta r}(u)$, setting $E = \sup_{x\in \B {2R} p \cap \cS} \theta(x,3R)$, there exists a covering
\begin{gather}
 \cS\cap \B R p  \subseteq \bigcup_{x\in \cD} \B {r_x}{x}\, , \quad \text{ with } \ \ r_x\geq r \ \ \text{ and}\ \ \sum_{x\in \cD} r_x^k\leq 2 C_{III}(m) R^k\, .
\end{gather}
Moreover, for all $x\in \cD$, either $r_x=r$, or for all $y\in \B {2r_x}{x}$:
  \begin{gather}\label{eq_Edrop}
   \theta(y,3r_x) \leq E- \delta\, .
  \end{gather}
\end{proposition}

\subsection{Proof of the main theorem \ref{th_main}}
Before we move to the proof of the proposition, we use it to prove the main theorem. This proof is basically a corollary of the covering proposition \ref{prop_cover}. We will use this proposition inductively to produce a family of coverings of $\cS=\cS^k_{\epsilon,\delta r}(u)\cap \B 1 0$ indexed by a parameter $i \in \mathbb{N}$ of the form
\begin{gather}\label{eq_final_cov_1}
 \cS\subseteq \bigcup_{x\in \cD^i} \B {r_x}{x}\, , \quad \sum_{x\in \cD^i} r_x^k \leq (c(m)C_F(m))^i\, .
\end{gather}
Moreover, if $E=\sup_{x\in \cS^{k}_{\e,\delta r}(u)\cap \B 2 0}\theta(x,3)$, we have for all $i$ 
\begin{gather}\label{eq_final_cov_2}
 r_x \leq r \quad \text{ or } \quad \forall y\in \cS \cap \B {2r_x}{x}\, , \ \ \theta(y,3r_x)\leq E- i \delta\, .
\end{gather}
Evidently, for $i\geq \lfloor E/\delta\rfloor +1$, the second condition cannot be verified, and so all the radii in the covering are going to be equal to $r$. As a consequence, we have the Minkowski bound 
\begin{gather}
 \Vol\ton{\B r {\cS^k_{\epsilon,\delta r}(u)}\cap \B 1 0}\leq (c(m)C_F(m))^{\lfloor \delta^{-1} E \rfloor +1}r^{m-k}\, .
\end{gather}
Since $\delta=\delta(m,\Lambda)$, it is clear that, up to enlarging the constant in the estimate, the same bound holds also for $\cS^k_{\epsilon,r}(u)$ in the place of $\cS^k_{\epsilon,\delta r}(u)$, and this concludes the proof of the Minkowski bounds in \eqref{eq:mink_bounds}.

\vspace{5mm}
In order to produce the covering in \eqref{eq_final_cov_1}, we will apply inductively the covering proposition \ref{prop_cover}. For $i=1$, 
we can apply this proposition to $\B 1 0$ and obtain the desired covering. Inductively, consider all the balls $\cur{\B {r_x}{x}}_{x\in \cD^i}$ and apply proposition \ref{prop_cover} to these balls. For each $x\in \cD^i$, we obtain a covering of the form
\begin{gather}
 \cS\cap \B {r_x}{x} \subseteq \bigcup_{y\in \cD_x} \B {r_y}{y}\, , \quad \sum_{y\in \cD_x} r_y^k \leq 2 C_{III}(m) r_x^k\, ,\\
 r_y \leq r \quad \text{ or } \quad \forall z\in \cS \cap \B {2r_y}{y}\, , \ \ \theta(z,3r_y)\leq E- (i+1) \delta\, .
\end{gather}
Set 
\begin{gather}
 \cD^{i+1}= \bigcup_{x\in \cD^i} \cD_{x}\, ,
\end{gather}
and the induction step is completed.

\subsubsection{Proof of the rectifiability of \texorpdfstring{$\cS^k_\epsilon$}{the quantitative strata}}
As for the rectifiability, this is going to be a corollary of Theorem \ref{th_azzamtolsa}, the uniform Minkowski bound \eqref{eq:mink_bounds} and the approximation theorem \ref{th_best_pinch}. 

In particular, let $\mu=\cH^k\mres \cur{\cS^k_\epsilon(u)\cap \B 1 0}$. From \eqref{eq:mink_bounds} we deduce that this measure is finite, as
\[
\mu(B_{1}(0)) \leq C(m,\Lambda,\e).
\]
In turn, by scaling this implies that for all $x \in B_{1}(0)$ and $r > 0$
\begin{gather}
 \mu \ton{\B r x}\leq C(m,\Lambda,\epsilon) r^k\, ,
\end{gather}
and thus $\mu$ is Ahlfors upper $k$-regular.

Now by the best approximation theorem \ref{th_best_pinch} and a simple change of variables we can write
\begin{equation}
\begin{split}
 \int _{\B 1 0} \int_0^1 D^k_{\mu} (x,r) \frac{\mathrm{d}r}{r} \mathrm{d}\mu(x)&\leq C(m,\epsilon) \int_0^1 \int _{\B 1 0} r^{-k} \int_{\B {r}{x}} \P_u(y,2r)\,\mathrm{d}\mu(y) \mathrm{d}\mu(x) \frac{\mathrm{d}r}{r} \\
& \leq C(m,\epsilon) \int_0^1 \int_{\B 1 0} \P_u(y,2r) \ton{r^{-k} \int_{\B {r}{y}} \,\mathrm{d}\mu(x)} \mathrm{d}\mu(y) \frac{\mathrm{d}r}{r} \\
& \leq C(m,\epsilon,\Lambda) \int_{\B 1 0} \int_0^1 \P_u(y,2r) \frac{\mathrm{d}r}{r} \mathrm{d}\mu(x) \\
&\leq C(m,\epsilon,\Lambda)\Lambda\, ,
\end{split}
\end{equation}
where the last inequality follows from 
\begin{gather}
 \int_0^1 \P_u(y,2r) \frac{\mathrm{d}r}{r}\stackrel{\eqref{eq_Pvstheta}}{\leq }  \int_0^1 [\theta(y,4r)-\theta(y,2r)] \frac{\mathrm{d}r}{r}=\lim_{t\to 0} \int_t^1 [\theta(y,4r)-\theta(y,2r)] \frac{\mathrm{d}r}{r} \\
 =\int_{1/2}^1 \theta(y,4r) \frac{dr}{r} + \underbrace{\lim_{t\to 0}\int_{t}^{1/2} \theta(y,4r) \frac{dr}{r} -\int_{2t}^{1} \theta(y,2r) \frac{dr}{r}}_{=0} -\lim_{t\to 0}\int_{t}^{2t} \theta(y,2r) \frac{dr}{r} \leq C(m)\Lambda\, .\notag
\end{gather}
The rectifiability of $\cS^k_\epsilon(u)$ is now a consequence of theorem \ref{th_azzamtolsa}.

By countable additivity, the rectifiability of $\cS^k(u)$ is a corollary of the rectifiability of $\cS^k_\epsilon(u)$ for all $\epsilon>0$. 

It is worth remarking that the uniform Ahlfors upper estimates obtained a priori for the measure $\mu=\cH^k\mres \cur{\cS^k_\epsilon (u)\cap \B 1 0}$ are \textit{essential} to carry out this computation, and actually they are the most difficult part of the estimate. This is why the proof of the rectifiability property is so easy.

\begin{flushright}
 $\square$
\end{flushright}

\subsection{Proof of Proposition \ref{prop_cover}}
 Now we turn to the proof of the covering proposition. We split this proof in two pieces by introducing a secondary covering proposition. 
 
 \begin{proposition}\label{prop_cover_sec}
 Under the assumptions of proposition \ref{prop_cover}, for all $0 < \rho < 1/100$, there exist $\delta=\delta(m,\N,Q,\Lambda,\epsilon,\rho)>0$ and $C_{II}(m)$ such that the following is true.
 
 There exists a finite covering of $\cS=\cS^k_{\epsilon,\delta r}(u) \cap \B R {p} $ of the form
\begin{gather}\label{eq_packing_II}
 \cS \subseteq \bigcup_{x\in \cD} \B {r_x}{x}\, , \quad \text{ with } \ \ r_x\geq r \ \ \text{ and}\ \ \sum_{x\in \cD} r_x^k\leq C_{II}(m) R^k\, .
\end{gather}
Moreover, for each $x\in \cD$, either there exists a $(k-1)$-dimensional space $W_x$ such that 
  \begin{gather}\label{eq_k-1}
   \cW_{x,r_x}\equiv \cur{y\in \cS \cap \B {2r_x}{x} \ \ \text{with}\ \ \theta(y,\rho r_x/20)\geq E-\delta}\subseteq \B {\rho r_x/10}{W_x}\, ,
  \end{gather}
  or $r_x=r$.  
\end{proposition}

Assuming this proposition, we prove proposition \ref{prop_cover}. The idea is simple: we consider this second covering, and refine it inductively on each ball with $r_x\geq r$ and no uniform energy drop.

\begin{proof}[Proof of proposition \ref{prop_cover}]
Let $0 < \rho < 1$ to be fixed later, and let $A \in \mathbb{N}$ be the first integer such that $\rho^{A} < r$. Also assume without loss of generality $p = 0$ and $R = 1$.
 
For all $i=1,\cdots, A$, we construct a covering of $\cS$ of the form
 \begin{gather}\label{eq_covcov}
  \cS\cap \B 1 0 \subseteq \bigcup_{x\in \cR_i} \B {r}{x}  \cup  \bigcup_{x\in \cF_i} \B {r_x}{x}\cup  \bigcup_{x\in \cB_i} \B {r_x}{x}\, ,
 \end{gather}
where $\cR_i$ are the balls of radius $r$ in the covering, $\cF_i$ are the balls where the uniform energy drop condition \eqref{eq_Edrop} is satisfied, and $\cB_i$ are the bad balls, where none of the two conditions is verified. We want to obtain uniform packing bounds on $\cR_i$ and $\cF_i$, and exponentially small packing bounds on $\cB_i$. We will refine our covering only on bad balls by re-applying the second covering lemma on those, and this is why we need smallness on their packing bounds. In detail, we want 
\begin{gather}\label{eq_cov5}
 \sum_{x\in \cR_i\cup \cF_i } r_x^k \leq C_F(m)\ton{\sum_{j=0}^i 7^{-j} }\, ,\quad \sum_{x\in \cB_i }r_x^k \leq 7^{-i}\, .
\end{gather}

\subsubsection{Induction step}
Pick a generic ball $\B {R}{p}$, and apply the second covering in Proposition \ref{prop_cover_sec} to it. We obtain a covering of the form
\begin{gather}
 \cS \cap \B R p \subseteq \bigcup_{x\in \cD} \B {r_x}{x}\, , \quad \text{ with } \ \ r_x\geq r \ \ \text{ and}\ \ \sum_{x\in \cD} r_x^k\leq C_{II}(m) R^k\, .
\end{gather}
We split $\cD$ into two disjoint sets: $\cD=\cD_r\cup \cD_+$, where the first set is the one with $r_x\leq 60\rho^{-1} r$. Observe that if $x$ is in the second set then \eqref{eq_k-1} is valid. For all $x\in \cD_r$, consider a simple covering of $\B {r_x}{x}$ by balls of radius $r$ with number bounded by $c(m)\rho^{-m}$, and let $\cR_p$ be the union of all centers in these coverings. Note that if $r_x=r$, we can keep this ball unchanged.

For all $x\in \cD_+$, consider a covering of $\B {r_x}{x}$ made of balls of radius $\rho r_x/{60}> r$ centered inside this ball and such that the family of balls with half the radius are pairwise disjoint. In particular, let
\begin{gather}
 \B {r_x}{x}\subseteq \bigcup_{y\in \cB_{x}}\B {\rho r_x/60}{y} \cup\bigcup_{y\in \cF_x} \B {\rho r_x/60}{y}\, ,
\end{gather}
where 
\begin{gather}
 \cW_{x,r_x}\cap \bigcup_{y\in \cF_x} \B {2\cdot(\rho r_x/60)}{y}=\emptyset \, , \quad \cB_{x}\subseteq \B {\rho r_x}{W_x}\, .
\end{gather}
Thus, the balls in $\cF_x$ will have a uniform energy drop, in particular we have that for all $y\in \cF_x$ and $z\in \B {2\cdot(\rho r_x/60)}{y}=\B {2 r_y}{y}$, 
\begin{gather}
 \theta(z,3r_y)< E-\delta\, .
\end{gather}

Moreover, the number of balls in $\cB_{x}$ is well controlled. Indeed, since $\cB_{x}\subseteq \B {\rho r_x}{W_x}$, $\cur{\B {\rho r_x/120}{y}}_{y\in \cB_x}$ are pairwise disjoint and $W_x$ is a $k$-dimensional subspace, then 
\begin{gather}
 \#\cur{\cF_x}\leq c(m) \rho^{-m}\, , \quad \#\cur{\cB_x}\leq c(m) \rho^{1-k}\, .
\end{gather}

Set $\cB_{p} = \cup_{x\in \cD} \cB_x$ and $\cF_p = \cup_{x\in \cD} \cF_x$. We have
\begin{gather}
 \sum_{z\in \cR_p \cup \cF_p} r_z^k \leq c(m) \rho^{-m+k}\sum_{x\in \cD} r_x^k \leq c(m) \rho^{-m+k} C_{II}(m)R^k \, ,  \\
 \sum_{z\in \cB_p} r_z^k \leq c(m) \rho^1\sum_{x\in \cD} r_x^k \leq c(m) \rho C_{II}(m)R^k\, .
\end{gather}
We choose $\rho=\rho(m)\leq 1/100$ sufficiently small so that
\begin{gather}
 c(m) \rho C_{II}(m)\leq 1/7\, .
\end{gather}
In this way, we have the estimates 
\begin{gather}
 \sum_{z\in \cR_p \cup \cF_p} r_z^k \leq C_{III}(m)R^k \, \quad \sum_{z\in \cB_p} r_z^k \leq 7^{-1}R^k\, .
\end{gather}

\subsubsection{Finishing the proof}
With the induction step, the proof follows easily. For $i=1$, apply the induction step to $\B 1 0$ and we obtain \eqref{eq_covcov} for $i=1$ with \eqref{eq_cov5}.

For generic $i$, we have by induction
\begin{gather}
 \cS\cap \B 1 0 \subseteq \bigcup_{x\in \cR_i} \B {r}{x}  \cup  \bigcup_{x\in \cF_i} \B {r_x}{x}\cup  \bigcup_{x\in \cB_i} \B {r_x}{x}\, .
\end{gather}
Apply the induction step on all the balls $\cur{\B {r_x}{x}}_{x\in \cB_i}$ separately, and define
\begin{gather}
 \cR_{i+1}= \cR_i \cup \bigcup_{x\in \cB_i} \cR_{x} \, , \quad \cF_{i+1}= \cF_i \cup \bigcup_{x\in \cB_i} \cF_{x}\, , \quad \cB_{i+1} = \bigcup_{x\in \cB_i} \cB_{x}\, .
\end{gather}
By construction, we have the estimates
\begin{gather}\label{eq_ind_est}
  \sum_{z\in \cR_{i+1} \cup \cF_{i+1}} r_z^k \leq C_{III}(m) \sum_{s=0}^{i} 7^{-s} \, \quad \sum_{z\in \cB_{i+1}} r_z^k \leq 7^{-i-1}\, .
\end{gather}
Note that at the step $i=A$ all the balls in our covering will either have energy drop (if they are in $\cF_A$) or have radius $=r$ (if they are in $\cR_A$). Equation \eqref{eq_ind_est} for $i=A$ gives the desired bound on the final covering.

\end{proof}

\vspace{5mm}

Now we turn our attention to the proof of proposition \ref{prop_cover_sec}, which is the last one needed to complete the main theorem.

\subsection{Proof of proposition \ref{prop_cover_sec}}
For convenience, we assume $p=0$ and $R=1$. Fix $\e, \rho > 0$, and let $A$ be such that $\rho^{A}\leq r<\rho^{A-1}$. The proof is based on an inductive covering by balls, where the discrete Reifenberg is applied in order to control the number of these balls.

\subsubsection{Construction of the covering} We split the inductive covering in two parts: at first we simply construct the covering inductively, and then we prove the packing bounds using Reifenberg's theorem. Specifically, we start by looking for an inductive (for $i=0,1,\cdots,A$) covering of the form 
\begin{gather}
\cS  \subseteq \bigcup_{x\in \cB_i} \B {r_x} x \cup \bigcup_{x\in\cG_i} \B {r_x} x  \, ,
\end{gather}
where the elements of $\cB_i$ are the centers of the \textit{bad balls} in our covering, and $\cG_i$ are the centers of the \textit{good balls}. In particular, we want :
\begin{enumerate}
 \item for all $i$ and $x\in \cB_i$, $r_x\geq \rho^i$ and there exists a $(k-1)$-dimensional subspace $W_x$ such that 
 \begin{gather}\label{eq_cWxr}
  \cW_{x,r_x}\equiv \cur{y\in \cS \cap \B {2r_x}{x} \ \ \text{such that}\ \ \theta(y,\rho r_x/20)\geq E-\delta}\subseteq \B {\rho r_x/10}{W_x}\,,
 \end{gather}
 where $\delta > 0$ is fixed, to be determined later;
 \item for all $i=1,\cdots,A$ and $x\in \cG_i$, $r_x=\rho^i$ and the set $\cW_{x,r_x}$ defined above $(\rho r_x/20)$-effectively spans some $k$-dimensional affine subspace $V_x$;
 \item for $i=A$, we have the bound
 \begin{gather}
  \sum_{x\in \cB_A\cup \cG_A} r_x^{k}\leq C_{II}(m)\, .
 \end{gather}

\end{enumerate}
Moreover, we request some extra properties of the centers of the covering in order to apply the discrete-Reifenberg theorem:
\begin{enumerate}\setcounter{enumi}{3}
 \item for all $i$, the balls in the collection $\cur{\B {r_x/10}{x}}_{x\in \cG_i\cup \cB_i}$ are pairwise disjoint;
 \item for all $i\geq 1$ and $x\in \cG_i$, we have the energy bound 
 \begin{gather}\label{eq_E-eta}
  \theta(x,r_x)\geq E-\eta\, \quad \mbox{for some $\eta > 0$} \,;
 \end{gather}
 \item there exists a constant $c(m)$ such that for all $i$, the balls in the collection $\cur{\B s {x}}_{x\in \cG_i, \   s\in [r_x,1]}$ are not $(k+1,\epsilon/c(m))$-symmetric.
\end{enumerate}
At each induction step, we will refine our covering on the good balls, while leaving the bad balls untouched.

For $i=0$, consider the set $\cW_{0,1}$. If this set does NOT $\rho/20$-effectively span something $k$-dimensional, then we call $B_{1}(0)$ a bad ball, set $\cG_i=\emptyset$ for all $i$ and $\cur{0}=\cB_0=\cB_A$ with $r_0=1$. This covering immediately satisfies all the properties of proposition \ref{prop_cover_sec}. 

In the other case, set $\cG_0=\cur{0}$ with $r_0=1$. 

\paragraph{Induction step}
Assuming by induction that all the properties listed above are valid up to the index $i$, we want to produce the covering for $i+1$. In order to do so, we want to refine our covering on good balls, and leave the previous bad balls intact.

Fix an arbitrary $x\in \cG_i$, and consider the set $\cW_{x,r_x}$. Since $\B {\rho^i}{x}$ is a good ball, by definition this set $[\rho^{i+1}/20]$-effectively spans a $k$-dimensional affine subspace $V_x$. By applying theorem \ref{th:quantitative_eps_reg} to the ball $\B {4\rho^i}{x}$, we find that there exists a $\delta(m,\Lambda,\epsilon,\rho)$ sufficiently small so that 
\begin{gather}\label{eq_cov2}
 \cS^k_{\e,\delta r} (u)\cap \B{2\rho^i}{x} \subset \B {\rho^{i+1}/10}{V_x}\, 
\end{gather}

Consider the set
\begin{gather}
 K =\bigcup_{x \in \cG_i} \ton{\B {\rho^i}{x} \cap V_x} \setminus \bigcup_{x\in \cB_i}\B {r_x}{x}\, .
\end{gather}
Given the inclusion \eqref{eq_cov2}, and since we have chosen $\rho\leq 1/100$, we have
\begin{gather}
 \cS \setminus \bigcup_{x\in \cB_i}\B {r_x}{x} \subseteq \B{\rho^{i+1}/5 }{K}\, .
\end{gather}
Let $\cD_K\subseteq K$ be a maximal subset of points at least $\rho^{i+1}/5$ apart, so that the balls $\cur{\B {\rho^{i+1}/10}{x}}_{x\in \cD_K}$ are pairwise disjoint. Note that these balls are also disjoint from $\cur{\B {r_x/3}{x}}_{x\in \cB_i}$ by construction. Moreover, by maximality of the subset
\begin{gather}
 \cS \setminus \bigcup_{x\in \cB_i}\B {r_x}{x} \subseteq \bigcup_{x\in \cD_K} \B{2 \rho^{i+1}/5 }{x}
\end{gather}
We can discard from this collection all the balls $\B{2 \rho^{i+1}/5 }{x}$ that have empty intersection with $\cS$. Now consider the collection
\begin{gather}
 \cur{\B{ \rho^{i+1} }{x}}_{x\in \cD_K}\, ,
\end{gather}
and classify these points into good and bad balls according to whether or not \eqref{eq_cWxr} is satisfied. In particular, if $\cW_{x,\rho^{i+1}}$ $\rho^{i+2}/20$-effectively spans a $k$-dimensional subspace $V_x$, then we say that $x\in \tilde \cG_{i+1}$, and $x\in \tilde \cB_{i+1}$ otherwise. We set
\begin{gather}
 \cB_{i+1}=\cB_i \cup \tilde \cB_{i+1}\, , \quad \cG_{i+1}=\tilde \cG_{i+1}\, .
\end{gather}

This takes care of properties $1$ and $2$ in the induction. 

Now fix any $x\in \cD_K$. By construction, there exists an $x'\in \cG_i$ such that $x\in V_{x'}\cap \B {r_{x'}}{x'}$. Hence, we can apply proposition \ref{prop_unipinch}, and prove that for all $\eta>0$ there exists a $\delta(m,\N,Q,\Lambda,\rho,\eta)$ sufficiently small so that 
\begin{gather}\label{eq_db}
 \theta\ton{x,\rho^{i+1}/40}\geq E-\eta\, .
\end{gather}
Moreover, there also exists some $x'\in \cS\cap \B {2\rho^{i+1}/5}{x}$. By definition of $\cS$, this implies that for every $(k+1)$-dimensional subspace $T=T_{x'}$:
\begin{gather}\label{eq_est_k+1}
 c(m) \rho^{(2-m)(i+1)} \int_{\B {\rho^{i+1}}{x}}\abs{D_T u}^2 \geq \ton{2\rho^{i+1}/5}^{2-m} \int_{\B {2\rho^{i+1}/5}{x'}} \abs{D_T u}^2 \geq \epsilon\, .
\end{gather}
In other words, $\B {\rho^{i+1}}{x}$ is not $(k+1,\epsilon/c(m))$-symmetric. Thus all the properties of our inductive covering are satisfied.

\paragraph{i=A} For $i=A$, one can use the same construction as above, but with radius $r$ instead of radius $\rho^A$. At this stage, one also does not need to make any distinction between good and bad balls.

At this stage, we also set
\begin{gather}
 \cD = \cB_A \cup \cG_A\, .
\end{gather}
We are left to prove the packing estimates \eqref{eq_packing_II}. 

\subsubsection{Volume estimates}\label{sec_packing_II}
We will apply the discrete Reifenberg theorem to the measure 
\begin{gather}
 \mu_{\cD} = \sum_{x\in \cD} r_x^k \delta_{x}\, .
\end{gather}
In order to do so, we need to check that \eqref{eq_reif_ass} is satisfied for this $\mu$, and we exploit the best approximation theorem \ref{th_best_pinch}.

However, as it will be evident later on, we cannot apply this theorem directly. Instead, we will prove the volume estimate with an upwards induction.

\subsubsection{Inductive statement}
For convenience, we define the one-parameter family of measures $\mu_t$ by setting
\begin{gather}
 \cD_t = \cD\cap \cur{r_x\leq t}\, ,\quad  \mu_t =\mu\mres \cD_t \leq \mu\, .
\end{gather}

Let $T$ be such that $2^{T-1} r < 1/70 \leq 2^{T} r$. We will prove by induction on $j=0, 1,\cdots, T$ that there exists a constant $C_{I}(m)$ such that for all $x\in \B 1 0$ and $s=2^j r$:
\begin{gather}
 \mu_s \ton{\B s x}=\sum_{y\in \cD\cap \B {s}{x} \ s.t. \ r_y\leq s} r_y^k  \leq C_{I}(m) s^k\, .
\end{gather}
Once this has been proved, with a simple covering argument we can turn the estimates for $j=T$ into the estimates \eqref{eq_packing_II}, replacing $C_{I}(m)$ with $C_{II}(m)=c(m)C_{I}(m)$ if necessary.

\paragraph{Base step in the induction, $j=0$ .} The first step of the induction is easy. Since by construction $r_x\geq r$ for all $x\in \cD$, and since the balls $\cur{\B{r_x/10}{x}}_{x\in \cD}$ are pairwise disjoint, a standard covering argument shows that for all $x\in \B 1 0$, 
\begin{gather}\label{eq_Ciii}
 \mu_r \ton{\B r x}\leq C_{0}(m) r^k\, .
\end{gather}

\subsubsection{Induction step}
The induction step is divided into two parts: first we are going to prove a weak packing bound for balls of radius $2^{j+1}r$. With this estimate, we will be able to apply the discrete Reifenberg theorem, which gives us a uniform scale invariant upper bound for the measure that lets us complete the induction.

\paragraph{\emph{Coarse bounds}}  Assuming that the induction step $j$ is proved, we can easily obtain a rough bound for $j+1$. Indeed, let $x\in \B 1 0$ be arbitrary, and consider the ball $\B {2^{j+1} r}{x}$. By covering this ball with $c(m)$ balls of half the radius, and using the induction hypothesis, we can estimate
\begin{gather}
 \mu_{2^j r} \ton{\B {2^{j+1} r}{x}}\leq c(m) C_{I}(m) (2^{j+1}r)^k\, .
\end{gather}
With a similar covering argument, we can estimate the ``new contributions'' in $\mu_{2^{j+1} r}$. To be precise, since $\cur{\B{r_x/10}{x}}_{x\in \cD}$ are all pairwise disjoint, we have
\begin{gather}
 \bar \cD= \cur{x\in \cD\cap \B {2^{j+1}r}{x}  \ \ \text{with} \  r_x\in (2^{j}r,2^{j+1}r] } \quad \quad \sum_{x\in \bar \cD} r_x^k \leq C_{0}(m) (2^{j+1}r)^k\, .
\end{gather}
Thus, choosing $C_{I}(m)\geq C_{0}(m)$, we have
\begin{gather}
 \mu_{2^{j+1} r} \ton{\B {2^{j+1} r}{x}}\leq c(m) C_{I}(m) (2^{j+1}r)^k\, .
\end{gather}

\paragraph{\emph{Refined estimate}}
In order to refine this last estimate, we need to apply the discrete Reifenberg \ref{th_disc_reif}. An essential tool is given by the estimates in corollary \ref{cor_best_pinch}. Fix any $\B {2^{j+1}r}{x}$ for $x\in \cD$. For convenience, hereafter we will denote
\begin{gather}\label{eq_mu_restricted}
 \mu_{2^{j+1}r}\mres \B {2^{j+1}r}{x} \equiv \mu\, .
\end{gather}
Set also for $y\in \cD$:
\begin{equation}
 \Pinch_{\cD}(y,s) = \begin{cases}
              \theta(y,4s)-\theta(y,2s) & \text{ for } s\geq r_y/10\, ,  \\
              0&\text{ for } s< r_y/10\, .
               \end{cases}
\end{equation}

By construction, and in particular by the estimate in \eqref{eq_est_k+1} and \eqref{eq_E-eta}, for $\eta$ sufficiently small we can apply Corollary \ref{cor_best_pinch} to $\mu$ and any ball $\B s x$ with $x\in \cD$ and $s\in [r_x,1]$, and obtain 
\begin{gather}
 D_{\mu}^k(x,s)  \leq C_1 s^{-k}\int_{\B {s}{x}} \Pinch_{\cD}(y,s)\,\mathrm{d}\mu(y)\, .
\end{gather}
As a corollary of this and \eqref{eq_beta_rough}, we can extend this relation for all $s\in [r_x/10,1]$, up to enlarging $C_1$ by $c(m)$:
\begin{gather}\label{eq_best_applied}
 D_{\mu}^k(x,s)  \leq c(m) C_1 s^{-k}\int_{\B {10 s}{x}} \Pinch_{\cD}(y,10s)\,\mathrm{d}\mu(y)\, .
\end{gather}
Note that this relation is trivially true also for $s\leq r_x/10$, because in this case the support of the measure $\mu$ inside the ball $\B {r_x/10}{x}$ is an isolated point.

We can use this estimate to prove \eqref{eq_reif_ass} for the measure $\mu$. Indeed, fix any $y\in \B {2^{j+2}r}{x}$, $t\in (0,2^{j+1}r]$, and in turn choose any $s\in [0,t]$. For these parameters, we can bound:
\begin{gather}
 \int_{\B t y} D^k_{\mu}(z,s) \, \mathrm{d}\mu(z)\stackrel{\eqref{eq_best_applied}}{\leq} C_1 s^{-k}\int_{\B t y} \qua{\int_{\B {10 s}{z}} \Pinch_{\cD}(p,10s)\, \mathrm{d}\mu(p)} \mathrm{d}\mu(z)\, .
\end{gather}
Considering that
\begin{gather}
 \cur{(p,z) \ s.t. \ \ \abs{z-y}\leq t \ \ \text{and} \ \ \abs{p-z}\leq 10 s }\subset \cur{(p,z) \ s.t. \ \ \abs{p-y}\leq t+10s \ \ \text{and} \ \ \abs{p-z}\leq 10 s }\, ,
\end{gather}
we can exchange the variables of integration and estimate
\begin{equation}
\begin{split}
 \int_{\B t y} D^k_{\mu}(z,s) \, \mathrm{d}\mu(z)&\leq C_1 \int_{B_{11t}(y)} \frac{\mu(\B {10s} p)}{s^k} \Pinch_{\cD}(p,10s) \ \mathrm{d}\mu(p) \\ &\leq c(m) C_1 C_{II}\int_{B_{11t}(y)} \Pinch_{\cD}(p,10s) \ \mathrm{d}\mu(p)\, .
 \end{split}
\end{equation}
Recall that by \eqref{eq_mu_restricted}, $\mu(A)= \mu(A\cap \B {2r^{j+1}r}{x}$. Note that the induction hypothesis and the coarse estimates have been used to obtain the last inequality. 

By integrating this inequality on $\int_0^t \frac{\mathrm{d}s}{s}$, we get
\begin{gather}
  \int_{\B t y}\ton{\int_0^t D^k_{\mu}(z,s)\,{\frac{\mathrm{d}s}{s}}}\, \mathrm{d}\mu(z) \leq c(m) C_1 C_{II} \int_{\B {11t} y} \qua{\int_0^{t} \Pinch_{\cD}(z,10s) \frac{\mathrm{d}s}{s}}\mathrm{d}\mu(z)\, .
\end{gather}
Note that for all $x\in \cD$, $\theta(0,1)-\theta(0,r_x)\leq \eta$. Thus for $t\leq 1/70$ we have
\begin{gather}
 \int_0^{t} \Pinch_{\cD}(x,10s) \frac{\mathrm{d}s}{s}=\int_{r_x}^{t} [\theta(x,40s)-\theta(x,20s)] \frac{\mathrm{d}s}{s}\\
 =\int_{t/2}^{t} \theta(x,40s)\frac{\mathrm{d}s}{s} + \underbrace{\int_{r_x}^{t/2} \theta(x,40s)\frac{\mathrm{d}s}{s}-\int_{2r_x}^{t/2} \theta(x,20s)\frac{\mathrm{d}s}{s}}_{=0} -\int_{r_x}^{2r_x} \theta(x,20s)\frac{\mathrm{d}s}{s}\\
 =\int_{t/2}^{t} \qua{\theta(x,40s)-\theta\ton{x,40 \frac{r_x}{t} s}}\frac{\mathrm{d}s}{s}\leq c\eta\, .
\end{gather}
This in turn implies
\begin{gather}
 \int_{\B t y}\ton{\int_0^t D^k_{\mu}(z,s)\,{\frac{\mathrm{d}s}{s}}}\, \mathrm{d}\mu(z) \leq c(m) C_1(m,\epsilon) C_{II} \eta t^k\, .
\end{gather}

By picking $\eta$ sufficiently small (in turn: by picking $\delta(m,\N,Q,\Lambda,\epsilon,\eta)$ sufficiently small), we can apply the discrete Reifenberg theorem to $\mu$ and prove that
\begin{gather}
  \mu_{2^{j+1} r} \ton{\B {2^{j+1} r}{x}}\leq C_R(m) (2^{j+1}r)^k\, .
\end{gather}
By picking $C_{II}(m) = \max\cur{C_{0}(m),C_R(m)}$, we complete the induction step, and in turn the proof of this proposition.

\begin{flushright}
 $\square$
\end{flushright}

\section{Continuity in non-positively curved spaces}\label{sec_nonpositive}

This section is devoted to the proof of the following result.
\begin{theorem}\label{thm.nonpositive=continuous}
Let $\N$ be a complete, simply connected manifold all of whose sectional curvatures are non-positive. Then, every minimizing harmonic map $u \in W^{1,2}(\Omega, \Iq(\N))$ satisfies \[\sing_{H}(u)= \emptyset.\]
\end{theorem}

The proof will be split into two parts. In the first part of the argument we will show a general lemma, Lemma \ref{lem.convex=subharmonic}. Then, in subsection \ref{subsec.simply connected situation} we will show how the lemma implies the theorem.

Observe that in the single-valued case $Q=1$ the hypothesis that $\pi_{1}(\N) = \{0\}$ is not necessary: indeed, in subsection \ref{subsec.general situation} we will show how the same result holds when $Q=1$ under the weaker assumption that $\N$ is connected. The proof will follow from the simply connected situation by means of lifting of Lipschitz-continuous functions into covering spaces. The hypothesis that $\N$ is simply connected, instead, is indispensable when $Q > 1$: in subsection \ref{subsec:Q-valued counterexample} we will provide an example of a \emph{singular} $Q$-valued minimizing harmonic map in a flat target manifold $\N$.

\begin{lemma}\label{lem.convex=subharmonic}
Let $f: \N \to \R$ be a $C^2$-regular function such that $\nabla^2f \ge 0$ on $T_p\N$ for all $p \in N$. Then 
\[ f\circ u = \sum_{\ell=1}^Q f(u_\ell) = {\rm const}. \]
for any $0$-homogeneous Dirichlet minimizer $u: \R^m \to \Iq(\N)$.
\end{lemma}
\begin{proof}
We will split the proof of the lemma into two steps:
\begin{itemize}
\item[claim 1:] for any Dirichlet minimizer $u: \Omega \to \Iq(\N)$, $\Omega \subset \R^m$ open, we have that $f\circ u \colon \Omega \to \R$ is subharmonic in the sense of distributions i.e.
\begin{equation}\label{eq.subharmonic}
\Delta (f\circ u) \ge 0\,;
\end{equation}
\item[claim 2:] any $0$-homogeneous subharmonic function is constant. 
\end{itemize}
The lemma is an immediate consequence of claim 1 and claim 2.

\emph{Proof of claim 1:} Let $\hat{f}$ be any extension of $f$ to $\R^N$ such that $\hat{f}$ is $C^2$ (for instance, take $\hat{f}(p) := \phi(p) f(\Pi(p))$, where $\Pi(p): \mathbf{U}_{\delta}(\N) \to \N$ is the nearest point projection from a $\delta$-tubular neighborhood $\mathbf{U}_{\delta}(\N)$ and $\phi$ is a non-negative smooth bump function supported in $\mathbf{U}_{\delta}(\N)$ and constantly equal to $1$ in a small neighborhood of $\N$). Observe that for every $p \in \N$ we have $\nabla^2 f = (D (D\hat{f})^{T_p} )^{T_p} $, where $v^{T_p}$ denotes the orthogonal projection of $v$ onto $T_p\N$. In order to deduce the claim, let $\varphi = \varphi(x) \in C^1_{c}(\Omega)$ non-negative be given and define the vector field 
\[ Y(x,p):= \varphi(x) \nabla \hat{f} (p)= \varphi(x) (D \hat{f}(p))^{T_{\Pi(p)}}\,. \]
The outer variation formula \eqref{outer_variation_formula} provides now
\begin{align*} 0& = \int_\Omega \sum_{i=1}^{m} \sum_{\ell=1}^Q \left( \langle D_iu_\ell , \nabla \hat{f}(u_\ell) \rangle D_i \varphi + \langle D_i u_\ell, D \nabla \hat{f} \cdot D_i u_\ell \rangle \varphi \right) \\
&= \int_\Omega \sum_{i=1}^{m} \left( D_i (f \circ u) D_i \varphi + \sum_{\ell=1}^Q \nabla^2f(u_\ell)(D_iu_\ell, D_iu_\ell) \varphi \right).
\end{align*}
In the last line we have used that $D_iu_\ell \in T_{u_\ell}\N$ and so $\langle D_i u_\ell, D \nabla \hat{f} \cdot D_i u_\ell \rangle = \nabla^2f(u_\ell)(D_iu_\ell, D_iu_\ell)$. Since the last term is non-negative we deduce the claim: 
\[ \int_{\Omega} \langle D (f\circ u),  D \varphi \rangle  \le 0 \quad \text{ for all } \varphi \in C^1_c(\Omega), \varphi \ge 0.\]

\emph{Proof of claim 2:}
Let $h \in W^{1,2}(\R^m)$ be $0$-homogeneous and subharmonic in the sense of distributions i.e.
\begin{equation}\label{eq:claimsubharmonic}
\int \langle Dh, D \varphi \rangle \le 0 \quad \text{ for all } \varphi \in C^1_c(\R^m), \varphi \ge 0.
\end{equation}
Suppose $h$ is not constant. Then there exists $a >0$ such that $h$ is not constant on the super-level set $\{ x \colon h(x) \ge - a\}$, which in turn implies $(h+a)^+$ is not constant. Take any non-negative $\eta(t), \eta(t)=0$ for $t>R$ (possibly a smooth approximation of $(R-t)^+$), and consider the test function $\varphi(x) = \eta(\abs{x}^2) (h+a)^+$ in \eqref{eq:claimsubharmonic}. Observe that $D_i \varphi = \eta(\abs{x}^2) D_i (h+a)^+  + \eta'(\abs{x}^2) 2 x^i (h+a)^+$. But $\sum_{i} D_ih(x) x^i= 0$ for a.e. $x$ in $\R^m$ because $h$ is homogeneous. Hence we deduce 
\[ 0 \ge \int \abs{D(h+a)^+}^2 \eta(\abs{x}^2).\]
But this contradicts the assumption that $(h+a)^+$ is not constant.
\end{proof}

\subsection{Proof of Theorem \ref{thm.nonpositive=continuous}}\label{subsec.simply connected situation}
In this subsection we conclude the proof of Theorem \ref{thm.nonpositive=continuous}. Recall that the hypotheses on $\N$ imply by the Cartan-Hadamard Theorem that $\exp_p: T_p\N \to \N$ is a covering map for every $p \in \N$. Furthermore, since $\N$ is assumed to be smooth we have $\dist_{\N}(q,p) = \abs{\exp_p^{-1}(q)}$. As a further consequence we deduce that for each $p$ the map $q \mapsto d^2_p(q):=\dist_{\N}(q,p)^2$ is smooth. By the second variation formula for length we deduce that $\nabla^2 d^2_p \ge 0$.\\ 
\begin{proof}[Proof of theorem \ref{thm.nonpositive=continuous}]
Again we split the proof in two parts:
\begin{enumerate}
\item[claim 1:] every $0$-homogeneous and locally minimizing $u: \R^m \to \Iq(\N)$ is constant;
\item[claim 2:] claim 1 implies that every locally minimizing map $u \in W^{1,2}(\Omega, \A_{Q}(\N))$ is continuous. 
\end{enumerate}
Obviously claim 2 is equivalent to the theorem. Let us first show how claim 2 follows from claim 1:\\
\emph{Proof of claim 2:} Let $u \in W^{1,2}(\Omega, \A_{Q}(\N))$ be locally energy minimizing, and suppose by contradiction that $\sing_{H}(u)\neq \emptyset $. Due to the characterization of the H\"older regular set by means of the tangent maps \cite[Lemma 6.1]{Hir16}, there is $y \in \sing_{H}(u)$ with a non-constant tangent map $T^u_y$ at $y$. But every tangent map is $0$-homogeneous and locally minimizing, and thus constant by claim 1. This is the required contradiction.\\ 
\emph{Proof of claim 1:} Let $u \in W^{1,2}(\R^m, \A_{Q}(\N))$ be any $0$-homogeneous locally minimizing map. As a consequence of the previous discussion, for every $k > 1$ and $p\in \N$ the function $q \in \N \mapsto f(q):=(d_p(q)^2)^k$ is $C^2$ regular and satisfying $\nabla^2 f \ge 0$ on $T_q\N$ since $ t \mapsto t^k$ is convex. Hence, we can apply lemma \ref{lem.convex=subharmonic} and deduce that for all $p \in \N$, $k >1$
\begin{equation}\label{eq:constant distance} d_p^{2k}\circ u = \sum_{\ell=1}^Q d_p^{2k}(u_\ell) \end{equation}
is constant. To conclude we need the following small algebraic fact, whose proof we postpone and first show the end of the argument.
\begin{lemma}\label{lem.algebraic identity}
Let $\{a_\ell\}_{\ell=1}^Q , \{b_\ell\}_{\ell=1}^Q$ be two families of non-negative real numbers. Suppose that for some sequence $k_i \to \infty$ we have 
\[ \sum_{\ell=1}^Q a_\ell^{k_i} = \sum_{\ell=1}^Q b_\ell^{k_i}\,. \]
Then, $\{a_\ell\}_{\ell=1}^Q = \{b_\ell\}_{\ell=1}^Q$.
\end{lemma}
In order to conclude the proof, fix any $x,y \in \R^m$ and let $u(x)=\sum_{\ell=1}^Q \a{p_\ell}, u(y)=\sum_{\ell=1}^Q \a{q_\ell}$. For a fixed $p_j$ we have by \eqref{eq:constant distance} that for all $k>1$
\[ \sum_{\ell=1}^Q \dist_\N(p_\ell,p_j)^{2k} =  \sum_{\ell=1}^Q \dist_\N(q_\ell,p_j)^{2k}. \]
But so the lemma \ref{lem.algebraic identity} implies that the number of zeros of the left- and right-hand side are the same. So we conclude that $\#\{ \ell : p_\ell = p_j\} = \#\{\ell \colon q_\ell = p_j\}$. Since $p_j$ was arbitrary we have $u(x)=u(y)$, that is $u$ is constant. \\
It remains to give the proof of the lemma.
\begin{proof}[Proof of lemma \ref{lem.algebraic identity}]
This lemma follows by induction on $Q$. For $Q=1$ the claim is obvious.\\
Suppose the claim is proven for $Q'<Q$. We may assume that the families are ordered, i.e. $a_1\ge a_2 \ge \dotsb \ge a_Q$ and $b_1 \ge b_2 \ge \dotsb b_Q$. If $a_1 =0 $ the claim follows. Hence we may assume $a_1>0$. The hypothesis implies that for all $k_i$
\[ \sum_{\ell=1}^Q \left(\frac{a_\ell}{a_1}\right)^{k_i} = \sum_{\ell=1}^Q \left(\frac{b_\ell}{a_1}\right)^{k_i}.\]
If we consider the limits for $k_i \to \infty$ we deduce that the LHS converges to $\#\{ \ell \colon a_\ell = a_1 \}$. If $b_1 >a_1$, the RHS converges to $+ \infty$. If $b_1< a_1$, on the other hand, the RHS converges to $0$. Hence, $b_1 = a_1$. Furthermore the RHS converges therefore to $\# \{ \ell \colon b_\ell = b_1 = a_1 \}$ which must be the same number as for the family $\{a_\ell\}_{\ell=1}^Q$. Hence we conclude that the assumption can now be written as 
\[ \sum_{\ell \colon a_\ell =a_1} a_1^{k_i} + \sum_{ \ell \colon a_\ell \neq a_1} a_\ell^{k_i} = \sum_{\ell \colon b_\ell =a_1} b_1^{k_i} + \sum_{ \ell \colon b_\ell \neq a_1} b_\ell^{k_i}. \]
As we have just shown the first sum on the left agrees with the first sum on the right, hence we deduce equality for the second sums for all $k_i$. The lemma follows now by induction hypothesis.  
\end{proof}
\end{proof}

\subsection{The improved result when $Q=1$}\label{subsec.general situation}
Although it is a known result we want to give a short proof of how the previous implies the following theorem. The important fact to remark is that for the single valued case the topology of the target does not play a role. 
 
\begin{theorem}\label{thm.nonpositive=continuous.Q=1}
Let $\N$ be a complete, connected manifold all of whose sectional curvatures are non-positive. Then, every locally energy minimizing map $u \in W^{1,2}(\Omega, \N)$ is smooth.
\end{theorem}

\begin{proof}
It is classical that every continuous harmonic map is smooth, hence it is sufficient to prove the continuity of the harmonic map.
We will show it by induction on the dimension $m$ of the base space $\Omega \subset \R^m$.
In fact, we will proceed similarly to the simply connected situation:
\begin{itemize}
\item[claim 1:] every $0$-homogeneous locally energy minimizer $u: \R^m \to \N$ is constant; 
\item[claim 2:] every locally energy minimizing map $u \in W^{1,2}(\Omega, \N)$ is continuous. 
\end{itemize}
\emph{Proof of claim 1:} Assume claim 1 is proven for $m'<m$. In a first step we want to show that the map $\left. u \right|_{\Sf^{m-1}}$ is continuous. For $m\le 3$ this holds true since $\mathcal{H}^{m-2}(\sing(u)) = 0$, \cite[Lemma 1 section 2.10]{Simon96}.
Now let $u: \R^m \to \N$ be $0$-homogeneous and energy minimizing, but suppose by contradiction that when restricted to the sphere $\Sf^{m-1}$ $u$ is not continuous, i.e. $\sing(u)\cap \Sf^{m-1} \neq \emptyset$. Hence we can find $y \in \sing(u)\cap \Sf^{m-1}$ at which there is a tangent map $T$ with at least one line of symmetry, i.e. such that for some $z \in \R^m$ one has $T(x+\lambda z)= T(x)$ for all $\lambda \in \R$, for all $x$. But this implies that $T$ is a locally energy minimizing $0$-homogeneous map from $\R^{m-1}$ to $\N$. By induction hypothesis $T$ must be constant. Hence $\sing(u)\cap \Sf^{m-1} = \emptyset$. \\
We have thus concluded that $v:=\left.u\right|_{\Sf^{m-1}} :\Sf^{m-1} \to \N$ is continuous and so smooth. Let $P: \tilde{\N} \to \N$ be an isometric covering map e.g. we can take $P=\exp_p: T_p\N \to \N$ by the Cartan-Hadamard Theorem. Since $\Sf^{m-1}$ is simply connected we have that $u_*(\pi_1(\Sf^{m-1})) \subset P_*(\pi_1(\R^n))$ and hence there exists a lift $\tilde{v}: \Sf^{m-1} \to \tilde{\N}$ of $v$, that is with $P \circ \tilde{v} = v$, compare \cite[Proposition 1.33]{Hatcher02}. The $0$-homogeneous extension $\tilde{u}(x):=\tilde{v}(\frac{x}{\abs{x}})$ must be locally energy minimizing since $P$ is isometric (indeed, if $\tilde{w}$ is a local competitor for $\tilde{u}$ then $w:=P\circ \tilde{w}$ is a local competitor for $u$, and $\int_{\Omega} \abs{Dw}^2 = \int_{\Omega} \abs{D\tilde{w}}^2$; hence, $\tilde{u}$ must be locally minimizing if $u$ is). But as proven in the simply connected situation every $0$-homogeneous locally energy minimizing map $\tilde{u}: \R^m \to \tilde{\N}$ is constant, compare claim 1 in subsection \ref{subsec.simply connected situation} with $Q=1$. This shows the claim.\\
\emph{Proof of claim 2:} Assume $\sing(u) \neq \emptyset$. Hence we can find $y \in \sing(u)$ at which there is a non-trivial tangent map $T$. But the existence of $T$ is ruled out by claim 1.  
\end{proof}

\subsection{\texorpdfstring{$Q$}{Q}-valued counterexample}\label{subsec:Q-valued counterexample}
In this subsection we want to present an example that the continuity fails for $Q$-valued functions if the target is not simply connected. Due to the results in subsection \ref{subsec.simply connected situation} we already know that the reason must be of topological nature.

\begin{proposition}\label{prop:counterexample to continuity into nonpositive curved spaces}
There is a $2$-valued Dirichlet minimizing map $u$ from $B^3 \subset \R^3$ into the flat torus $\mathbb{T}^2=\C/\Z^2$ with the property that $u|_{\Sf^2}$ is Lipschitz continuous, $\sing_{H}(u) \Subset B_3$ and $\sing_{H}(u) \neq \emptyset$.
\end{proposition}
\begin{proof}
The construction of the example proceeds as follows:
\begin{enumerate}
\item we present an explicit example of a branched covering $\pi: \mathcal{V} \to \Sf^2$, where $\mathcal{V}$ is a torus. $\mathcal{V}$ is constructed as a complex variety in $\hat{\C}\times \hat{\C}$;
\item using $\pi$ we construct a $2$-valued, Lipschitz continuous map $v$ from $\Sf^2$ into the flat torus $\mathbb{T}^2=\C/ \Z^2$ with finite energy;
\item let $u$ be a minimizer of the Dirichlet energy with respect to $g(x):= v(\frac{x}{\abs{x}})$. We will show that $u$ cannot be continuous.
\end{enumerate}
Let us now present the details to the outlined steps:\\

\emph{step 1:} Let $\hat{\C}$ be the Riemann sphere. We fix two non zero, unequal complex numbers $a,b$ and define the meromorphic function $m(z):= z \frac{z-a}{z-b}$. Consider the complex variety \[ \mathcal{V}:=\left\lbrace (w,z) \in \hat{\C} \times \hat{\C} \colon w^2 = z \frac{z-a}{z-b} \right\rbrace .\]
Consider the projection $\pi: \hat{\C}\times \hat{\C} \to \hat{\C}$ onto the second component. Restricted to $\mathcal{V}$, we obtain a ramified covering map
\[ \pi: \mathcal{V} \to \hat{\C}. \]
The map $\pi$ by definition is a two valued covering with ramification points in $P_1=(0,0)$, $P_2=(0,a)$, $P_3=(\infty, b)$ and $P_4=(\infty, \infty)$.
We claim that $\pi$ takes the form $\pi(\zeta)=\zeta^2$ at each of the ramifications points $P_i$. Furthermore this implies that $\mathcal{V}$ is smoothly embedded, i.e. does not have any singular points.\\
Set $p_1=0=p'_4\,, p_2=a\,, p'_3=\frac{1}{b}$ ($p_3=b=\frac{1}{p'_3}\,, p_4=+\infty= \frac{1}{p_4'}$).\\
At $P_1, P_2$ we have $m(z) = (z-p_i) h_i(z-p_i)$ with $h_i$ holomorphic in a neighborhood $U_i$ of $0$ and $h_i(0) \neq 0$. We deduce that $\varphi_i(z):=(z-p_i) h_i(z-p_i)=m(z)$ is locally a holomorphic diffeomorphism between $p_i+U_i$ and a neighborhood $V_i\subset \C$ of $0$. Now it is straightforward to check that
\[\Phi_i: \zeta \in V_i \mapsto (\zeta, \varphi_i^{-1}(\zeta^2)), \]
is a local  parametrization of $\mathcal{V}$ around $P_i$, i.e. $\varphi_i\circ \pi \circ \Phi_i(\zeta) = \zeta^2$. Changing $U_i$ we may assume that $V_i= \mathbb{D}_{r_i}$ for each $i=1,2$, where $\mathbb{D}_{r}$ is the disc centered at $0 \in \C$ with radius $r$. Furthermore since $\Phi_i$ is a smooth regular map $P_i$ is not a singular point of $\mathcal{V}$.
 
To analyze the ramification points $P_3,P_4$ we use the inversion $I: \hat{\C} \to \hat{\C}$ with $I(z)=\frac{1}{z}$. Observe that $(w,z) \in \mathcal{V}$ if and only if $(w'=I(w), z'=I(z))$ is a solution of $(w')^2 = m'(z')$ with $m'(z)= I\circ m \circ I = \frac{b}{a} \,z'\frac{z'-\frac{1}{b}}{z'-\frac{1}{a}}$ or
 \[ I(\mathcal{V})=\left\lbrace (w',z')  \in \hat{\C} \times \hat{\C} \colon w'^2 =\frac{b}{a} \, z' \frac{z'-\frac{1}{b}}{z'-\frac{1}{a}}\right\rbrace. \]
Now we can argue for $P_3, P_4$ as for $P_1, P_2$ interchanging $p_1, p_2$ with $p'_4$ and $p'_3$ (and denote with $U_i'$, $i=3,4$ the related neighborhoods of $0$).  As a conclusion we can apply the Riemann-Hurwitz formula, 
and obtain 
\[ \chi(\mathcal{V}) = - 4 + \sum_{i=1}^4 (2-1) = 0.\]
Hence $\mathcal{V}$ is a torus.  \\

\emph{step 2:}
In the following we equip $\mathcal{V}$ with the pullback metric ${\bf g}:=\iota^*\delta$ of its immersion $\iota: \mathcal{V} \hookrightarrow \hat{\C} \times \hat{\C}$. Observe that the metric ${\bf g}$ is compatible with the conformal structure considered in step 1. \\
The construction of $v$ will be done in two steps. First, since $\pi: \mathcal{V} \to \hat{\C}$ is a branched conformal covering of degree two there is a natural way to define $2$-valued maps with finite energy. These maps are not Lipschitz continuous, in fact only $C^{0,\frac12}$, but we are able to find a Lipschitz continuous map with similar properties nearby.

Let $f: \mathcal{V} \to \N$ be any smooth function from the Riemann surface $\mathcal{V}$ into a manifold $\N$. We define a two valued map $u=u_f \colon \hat{\C} \to \A_{2}(\N)$ using the branched covering map $\pi: \mathcal{V} \to \hat{\C}$ as follows 
\[ u(z) := \sum_{ P \in \pi^{-1}(z)} \a{f(P)}, \]
counting multiplicities i.e. $u(p_i)=2\a{f(P_i)}$ for $i=1,\cdots, 4$.\\ We claim that $u \in W^{1,2}(\Sf^2,\I{2}(\N))$ with 
\begin{equation}\label{eq:Dirichlet energy of a branched covering} \int_{\Sf^2} \abs{\nabla u}^2 = \int_{\mathcal{V}} \abs{\nabla f}^2. \end{equation}
Let $\gamma$ be a smooth path connecting $p_1,p_2,p_3,p_4$. We obtain a simply connected domain $\Omega\subset \C$ setting
\[\Omega := \hat{\C} \setminus \left( \bigcup_{i=1,2} (p_i+U_i) \cup \bigcup_{i=3,4} I( p_i' + U'_i) \cup \gamma\right)\,. \]
Hence there exist two holomorphic maps $\psi_i: \Omega \to \pi^{-1}(\Omega)$ with $\psi_1(\Omega)\cup \psi_2(\Omega)= \pi^{-1}(\Omega)$ such that 
\[ u(z) = \a{f\circ \psi_1} + \a{f\circ \psi_2} \quad \mbox{for every } z \in \Omega\,. \]
Since the Dirichlet energy is conformally invariant (cf. \cite[Lemma 3.12]{DLS11a}), we have
\[ \int_{\Omega} \abs{\nabla u}^2 = \int_{\pi^{-1}(\Omega)} \abs{\nabla f}^2. \]
Now we consider a ramification point, for instance $P_1$ and the related neighborhood $p_1 + U_1$. Using the previously introduced parametrization $\Phi_1$ we have 
\[ u\circ \varphi_1^{-1}(\zeta) = \a{ f\circ\Phi_1(\zeta^{\frac12})} + \a{ f\circ\Phi_1(-\zeta^{\frac12})}. \]
The maps $\zeta\in \mathbb{D}_{r_1^2} \mapsto \pm \zeta^{\frac12}$ both together parametrize $\mathbb{D}_{r_1}$. Hence, as before, due to the conformal invariance of Dirichlet energy we obtain 
\[ \int_{\varphi_1^{-1}(\mathbb{D}_{r_1^2})} \abs{\nabla u}^2 = \int_{\Phi_1(\mathbb{D}_{r_1})} \abs{\nabla f}^2. \]
Summing up all the pieces and using that $\Ha^2(\gamma)=0$ we obtain \eqref{eq:Dirichlet energy of a branched covering}.\\
By step 1 $\mathcal{V}$ is a smoothly embedded torus in $\hat{\C}\times \hat{\C}$; hence, there exists a smooth diffeomorphism $\Phi: \mathcal{V} \to \mathbb{T}^2$. Apply the above construction with the specific choice $f=\Phi$ to obtain
\[ \tilde{v}(z):= \sum_{P \in \pi^{-1}(z) } \a{ \Phi(P) } \in W^{1,2}(\hat{\C}, \I{2}(\mathbb{T}^2)). \]
It remains to show that there is $v \in \Lip(\hat{\C}, \I{2}(\mathbb{T}^2))$ nearby. This will be a consequence of the following approximation lemma:

\begin{lemma}\label{lem.approximationbylipschitzfunctions}
Given $w \in W^{1,2}(\Omega, \Iq(\N))\cap C^0(\Omega, \Iq(\N))$, for every $\Omega' \Subset \Omega$ there exists $w_j \in W^{1,2}(\Omega, \Iq(\N))\cap C^0(\Omega, \Iq(\N))$ with
\begin{align*} w_j \in \Lip(\Omega', \Iq(\N)); &\quad w_j = w \text{ in a neighborhood of } \partial \Omega\\
\norm{\G(w_j,w)}_{L^\infty(\Omega')} \to 0; &\quad \int_{\Omega'} \abs{D w_j}^2 \to \int_{\Omega'} \abs{D w}^2 \text{ as } j \to \infty.\end{align*}
\end{lemma}

Before coming to the proof of this lemma let us present how to conclude. 
Apply the lemma to the $0$-homogeneous extension of $\tilde{v}$ in the annulus $\Omega := B^{3}_{2}(0) \setminus B^{3}_{\frac{1}{4}}(0)$ to obtain an approximating sequence $v_j \in W^{1,2}(B^{3}_2(0) \setminus B^{3}_{\frac{1}{4}}(0), \A_{2}(\mathbb{T}^2))\cap \Lip(B^{3}_{\frac32}(0) \setminus B^{3}_{\frac{1}{2}}(0), \A_{2}(\mathbb{T}^2))$. 
Choosing $j$ sufficiently large we can guarantee that for every $p \in \mathbb{T}^2\setminus \bigcup_{i=1}^4 B_{2^{-2017}}(\Phi(P_i))$ there is precisely one $z \in \hat{\C} \simeq \partial B^{3}_{1}(0)$ with $p \in \spt(v_j(z))$.  Now fix such $j$ sufficiently large and set $v := \left. v_{j} \right|_{\hat{C}}$. The $0$-homogeneous extension of $v$ i.e. $g(x):=v(\frac{x}{\abs{x}})$ for $x \in B_1\subset \R^3$ is an element of $W^{1,2}(B_1, \A_{2}(\mathbb{T}^2))$ and Lipschitz continuous outside of $0$. Now we may apply the direct method to obtain a Dirichlet minimizing map $u: B_1 \to \A_{2}(\mathbb{T}^2)$ with $u|_{\Sf^2}=g|_{\Sf^2}$, compare \cite[Theorem 0.8]{DLS11a}.\\

\begin{proof}[Proof of Lemma \ref{lem.approximationbylipschitzfunctions}] Since $\N \hookrightarrow \R^N$ smooth isometrically there exists a smooth nearest point projection $\Pi: \mathbf{U}_{\delta}(\N) \to \N$ for some $\delta>0$. Let $\pmb{\xi}_{BW} : \Iq(\R^N) \to \R^M$ be the locally isometric "improved" Almgren/B. White embedding of $\A_{Q}(\R^N)$, cf. \cite[Section 2]{DLS11a}. We will denote with $\pmb{\rho}_{BW}: \R^M \to \Iq(\R^N)$ the related Lipschitz retraction, satisfying $\pmb{\rho}_{BW}\circ \pmb{\xi}_{BW} = \mathrm{id}$ on $\Iq(\R^N)$,  \cite[Corollary 2.2]{DLS11a}.\\
Since $w$ is assumed to be continuous, there exists $\tilde{w}_j$ with $\tilde{w}_j \to \pmb{\xi}_{BW} \circ w$ in $L^\infty(\Omega, \R^M)\cap W^{1,2}(\Omega, \R^M)$, $\tilde{w}_j \in \Lip(\Omega', \R^M)$ for every $\Omega' \Subset \Omega$, and $\tilde{w}_j= \pmb{\xi}_{BW} \circ w$ in a neighborhood of $\partial \Omega$. For instance, one may take $\tilde{w}_j = (1-\theta)\; \pmb{\xi}_{BW} \circ w + \theta\; \eta_{\e_j} \star (\pmb{\xi}_{BW} \circ w)$, for an appropriate cut-of-function $\theta$ and a sequence of mollifiers $\eta_{\e_j}$.\\
Since $\pmb{\rho}_{BW}$ is a Lipschitz-retraction and $\pmb{\xi}_{BW}$ is a local isometry we conclude that the sequence
\[ \hat{w}_j:= \pmb{\rho}_{BW} \circ \tilde{w}_j : \Omega \to \Iq(\R^N) \]
has the claimed properties up to the fact that $\hat{w}_j$ does not necessarily take values in $\N$. But for sufficient large $j$ we have $\G(\hat{w}_j(x), w(x)) < \frac12 \delta $ for all $x \in \Omega$ hence 
\[ w_j(x):=\Pi\circ \hat{w}_j(x)= \sum_{\ell=1}^Q \left\llbracket \Pi((\hat{w}_j(x))_{\ell}) \right\rrbracket \]
is well-defined and has all the claimed properties. It is clearly Lipschitz continuous on $\Omega'$ since $\Pi$ is smooth and Lipschitz. The sequence $w_j$ converges uniformly to $w$ since $\Pi$ is the identity on $\N$ and finally 
\[ \int_{\Omega} \abs{\nabla w}^2 \le \liminf_{j \to \infty } \int_{\Omega} \abs{\nabla w_j}^2 \le \liminf_{j  \to \infty} \int_{\Omega} \frac{\abs{\nabla \hat{w}_j}^2}{(1 - \dist(\hat{w}_j(x), \N)C)^2}  = \int_{\Omega} \abs{\nabla w}^2.\]
In the first inequality we used the lower-semicontinuity of the Dirichlet energy, in the second an estimate on the derivative of the nearest point projection $\Pi$, compare \cite[Remark 2.1 (iv)]{Hir16}.
\end{proof}

\emph{step 3:}
That $\sing_{H}(u) \Subset B^3$ follows from the fact that  $u|_{\Sf^2}$ is Lipschitz continuous and a boundary regularity result for $Q$-valued locally energy minimizing maps, which can be obtained from the analogous result of \cite{Hir16b} for ``classical'' $\R^N$-valued $\Dir$-minimizers modulo slight modifications of the arguments: precisely, this is how to proceed in order to obtain the boundary regularity result \cite[Theorem 0.1]{Hir16b} in the manifold valued setting for $s=1$. Only in the proof of Proposition 3.3, one replaces the application of Lemma B.2. to obtain the interpolation $\varphi(k')$ by the application of the $Q$-valued Luckhaus lemma, \cite[Lemma 3.1]{Hir16} to obtain $\varphi(k')$. Due to the $L^\infty$-bound in the Luckhaus lemma one can apply the nearest point projection $\Pi: \mathbf{U}_{\delta}(\N) \to \N$ and obtain an interpolation function $\Pi\circ \varphi(k')$ that satisfies the same bounds.\\

To show that $\sing_{H}(u) \neq \emptyset$ the idea is to use the "degree" of $u|_{\Sf^2}$ to show that $u$ cannot be continuous. We will use the notion of "degree" suggested by the theory of Cartesian currents.
We will need the following fact about push-forwards of integral currents by $Q$-valued proper Lipschitz continuous functions (see, for instance, \cite[Section 1]{DLS13a} or \cite[Section 2]{SS17a}): let $\Omega \subset \R^m$ be open (non necessarily connected) with smooth boundary $\partial \Omega$, $\Sigma \subset \Omega$ any smooth $k$-dimensional surface, and $f \colon \Omega \to \A_{Q}(\N)$ Lipschitz and proper. Then, the following holds:
\begin{itemize}
\item $T:=f_\sharp \a{\Omega}$ is an $m$-dimensional integer rectifiable current in $\N$, $S:=f_\sharp \a{\Sigma}$ is a $k$-dimensional integer rectifiable current in $\N$;
\item it holds $\partial T = f_\sharp \a{\partial \Omega}$.
\end{itemize}
In case $\Omega$ is $3$-dimensional, $\Sigma$ and $\N$ are $2$-dimensional without boundary, the constancy theorem for integral currents implies that
\begin{itemize}
\item[(i)] $T=f_\sharp \a{\Omega} = 0$ since $T$ is a $3$-dimensional current supported in a $2$-dimensional manifold;
\item[(i)] $S=f_\sharp \a{\Sigma}=\theta_\Sigma \a{\N}$ for some $\theta_\Sigma \in \Z$ since $S$ is a $2$-dimensional integer rectifiable current without boundary supported in a $2$-dimensional manifold;
\item[(iii)] the following identity holds true 
\begin{equation}\label{eq:boundarycurrentidentiy}
0 = \partial T = f_\sharp \a{\partial\Omega} = \sum_{j=1}^J \theta_{\Sigma_j} \a{\N}  
\end{equation}
where $\Sigma_j$ are the different components of $\partial \Omega$ i.e. $\partial \Omega = \bigcup_{j=1}^J \Sigma_j$.
\end{itemize}

Now we can conclude \emph{step 3}. Assume by contradiction that $u$ is continuous. First extend $u$ to $B_2$ setting $u(x)=u(\frac{x}{\abs{x}})$ for $\abs{x}>1$. Apply the approximation lemma \ref{lem.approximationbylipschitzfunctions} to $u$ with $\Omega = B_{\frac{3}{2}}$ and $\Omega' = B_1$ to obtain a sequence $u_j \in W^{1,2}(B_{\frac32}, \I{2}(\mathbb{T}^2))$ with $u_j|_{\partial B_{\frac32}}= u|_{\partial B_{\frac32}}$ for all $j$. Since $u$ is Lipschitz continuous on $B_2\setminus \overline{B_1}$ we have that $u_j \in \Lip(B_{\frac32}, \I{2}(\mathbb{T}^2))$. Modifying $u_j$ slightly we can assume that $u_j$ is constant in a small ball $B_r(0)$. This can be achieved for instance by composing $u_j$ with a Lipschitz function of the form
\[ \psi(x) := \begin{cases} x &\text{ for } \abs{x} \ge 2r\\
\frac{ \abs{x} - r}{r} x & \text{ for } r \le \abs{x} < 2r\\
0 & \text{ for } \abs{x}<r. \end{cases} \]
Now consider the set $\Omega=B_{\frac32} \setminus B_{\frac{r}{2}}$ with smooth boundary components $\Sigma_1, \Sigma_2$ given by $\a{\Sigma_1}= \a{\partial B_{\frac32}}$ and $\a{\Sigma_2} = - \a{\partial B_{\frac{r}{2}}}$ in the sense of currents. Since $u_j$ is constant on $B_r$ we have $({u_j})_\sharp\a{\Sigma_2}=0$ by the very definition of push-forward. The identity \eqref{eq:boundarycurrentidentiy} implies that
\[ 0 = ({u_j})_\sharp \a{\Sigma_1} = u_\sharp \a{\partial B_{\frac32}} = u_\sharp \a{\partial B_1}. \]
We used that $u_j=u$ on $\partial B_{\frac32}$ for all $j$ and $u$ is $0$-homogeneous on $B_2\setminus B_1$. But this is a contradiction since $u_\sharp \a{\partial B_1}\neq 0$ by the way $u$ was constructed (compare the choice of the boundary datum in the approximation above). 
\end{proof}

\subsection{Example of a "non-classical" tangent map}\label{subsec:tangentmap}
In this section we want to observe that tangent maps of $Q$-valued locally Dirichlet minimizing maps may have different structures than "classical" one-valued tangent maps.\\
Following the classical scheme we make the following definition:

\begin{definition}\label{def.regular-singular-point}
Let $u \in W^{1,2}(\Omega, \A_{Q}(\N))$ be energy minimizing. A point $x \in \sing_{H}(u)$ is called a regular-singular point if for every tangent map $T$ at $x$ there are classical one-valued tangent maps $T_\ell: \R^m \to \N$, i.e. $0$-homogeneous locally energy minimizing maps, such that 
\[ T = \sum_{\ell=1}^Q \a{T_\ell}. \]
\end{definition}
It is worth noting that every continuity point of a locally energy minimizing map has the property above, by the identification of regular points by the existence of a constant tangent map, \cite[Lemma 6.1 (iii)]{Hir16}.

We will show the following
\begin{proposition}\label{prop.example of tangent map}
Let $u: B_{1}(0) \subset \R^3 \to \I{2}(\mathbb{T}^2)$ be the Dirichlet minimizing map constructed in the previous section. Then, $\sing_{H}(u)$ does not contain any regular-singular point.
\end{proposition}
\begin{proof}
It was shown in \emph{step 3} of the previous section that $\sing_{H}(u) \neq \emptyset$ and $\sing_H(u) \Subset B^3$, hence at every point $x \in \sing_H(u)$ a tangent map exists.  Let $T: \R^3 \to \I{2}(\mathbb{T}^2)$ be an arbitrary tangent map at some some $y \in \sing_{H}(u)$. Assume by contradiction that there are "classical" tangent maps $T_1, T_2: \R^3 \to \mathbb{T}^2$ such that 
\[ T = \a{T_1} + \a{T_2}. \]
Each $T_i$ is $0$-homogeneous and locally energy minimizing. Since $\mathbb{T}^2$ is flat each $T_i$ satisfies the assumptions of claim 1 in the proof of Theorem \ref{thm.nonpositive=continuous.Q=1}, hence $T_i$ must be constant. But this contradicts that $T$ is a non-constant tangent map and concludes the proof of the proposition.
\end{proof}

\nocite{*}
\bibliographystyle{aomalpha}
\bibliography{Biblio}

\providecommand{\bysame}{\leavevmode\hbox to3em{\hrulefill}\thinspace}
\providecommand{\noopsort}[1]{}
\providecommand{\mr}[1]{\href{http://www.ams.org/mathscinet-getitem?mr=#1}{MR~#1}}
\providecommand{\zbl}[1]{\href{http://www.zentralblatt-math.org/zmath/en/search/?q=an:#1}{Zbl~#1}}
\providecommand{\jfm}[1]{\href{http://www.emis.de/cgi-bin/JFM-item?#1}{JFM~#1}}
\providecommand{\arxiv}[1]{\href{http://www.arxiv.org/abs/#1}{arXiv~#1}}
\providecommand{\doi}[1]{\url{http://dx.doi.org/#1}}
\providecommand{\MR}{\relax\ifhmode\unskip\space\fi MR }
\providecommand{\MRhref}[2]{%
  \href{http://www.ams.org/mathscinet-getitem?mr=#1}{#2}
}
\providecommand{\href}[2]{#2}
\begin{thebibliography}{BDPW15}

\bibitem[Alm00]{almgren_big}
\bgroup\scshape{}F.~J. Almgren, Jr.\egroup{}, \emph{Almgren's big regularity
  paper}, \emph{World Scientific Monograph Series in Mathematics} \textbf{1},
  World Scientific Publishing Co., Inc., River Edge, NJ, 2000, $Q$-valued
  functions minimizing Dirichlet's integral and the regularity of
  area-minimizing rectifiable currents up to codimension 2, With a preface by
  Jean E. Taylor and Vladimir Scheffer. \mr{1777737}.  \zbl{0985.49001}.

\bibitem[AT15]{AzzTol}
\bgroup\scshape{}J.~Azzam\egroup{} and \bgroup\scshape{}X.~Tolsa\egroup{},
  Characterization of {$n$}-rectifiability in terms of {J}ones' square
  function: {P}art {II},  \emph{Geom. Funct. Anal.} \textbf{25} (2015),
  1371--1412. \mr{3426057}.  \zbl{06521333}.  \doi{10.1007/s00039-015-0334-7}.

\bibitem[Bet93]{Bethuel}
\bgroup\scshape{}F.~Bethuel\egroup{}, On the singular set of stationary
  harmonic maps,  \emph{Manuscripta Math.} \textbf{78} (1993), 417--443.
  \mr{1208652}.  \doi{10.1007/BF02599324}.  Available at
  \url{http://dx.doi.org/10.1007/BF02599324}.

\bibitem[BDPG15]{J3}
\bgroup\scshape{}P.~Bouafia\egroup{}, \bgroup\scshape{}T.~De~Pauw\egroup{}, and
  \bgroup\scshape{}J.~Goblet\egroup{}, Existence of {$p$} harmonic multiple
  valued maps into a separable {H}ilbert space,  \emph{Ann. Inst. Fourier
  (Grenoble)} \textbf{65} (2015), 763--833. \mr{3449167}.  Available at
  \url{http://aif.cedram.org/item?id=AIF_2015__65_2_763_0}.

\bibitem[BDPW15]{BDPW}
\bgroup\scshape{}P.~Bouafia\egroup{}, \bgroup\scshape{}T.~De~Pauw\egroup{}, and
  \bgroup\scshape{}C.~Wang\egroup{}, Multiple valued maps into a separable
  {H}ilbert space that almost minimize their {$p$} {D}irichlet energy or are
  squeeze and squash stationary,  \emph{Calc. Var. Partial Differential
  Equations} \textbf{54} (2015), 2167--2196. \mr{3396449}.
  \doi{10.1007/s00526-015-0861-y}.  Available at
  \url{http://dx.doi.org/10.1007/s00526-015-0861-y}.

\bibitem[CN13a]{ChNa1}
\bgroup\scshape{}J.~Cheeger\egroup{} and \bgroup\scshape{}A.~Naber\egroup{},
  Lower bounds on {R}icci curvature and quantitative behavior of singular sets,
   \emph{Invent. Math.} \textbf{191} (2013), 321--339. \mr{3010378}.
  \zbl{06143195}.  \doi{10.1007/s00222-012-0394-3}.  Available at
  \url{http://arxiv.org/abs/1103.1819}.

\bibitem[CN13b]{ChNa2}
\bgroup\scshape{}J.~Cheeger\egroup{} and \bgroup\scshape{}A.~Naber\egroup{},
  Quantitative stratification and the regularity of harmonic maps and minimal
  currents,  \emph{Comm. Pure Appl. Math.} \textbf{66} (2013), 965--990.
  \mr{3043387}.  \zbl{1269.53063}.  \doi{10.1002/cpa.21446}.  Available at
  \url{http://arxiv.org/abs/1107.3097}.

\bibitem[DT12]{davidtoro}
\bgroup\scshape{}G.~David\egroup{} and \bgroup\scshape{}T.~Toro\egroup{},
  Reifenberg parameterizations for sets with holes,  \emph{Mem. Amer. Math.
  Soc.} \textbf{215} (2012), vi+102. \mr{2907827}.  \zbl{1236.28002}.
  \doi{10.1090/S0065-9266-2011-00629-5}.

\bibitem[DLFS11]{J6}
\bgroup\scshape{}C.~De~Lellis\egroup{}, \bgroup\scshape{}M.~Focardi\egroup{},
  and \bgroup\scshape{}E.~N. Spadaro\egroup{}, Lower semicontinuous functionals
  for {A}lmgren's multiple valued functions,  \emph{Ann. Acad. Sci. Fenn.
  Math.} \textbf{36} (2011), 393--410. \mr{2757522}.
  \doi{10.5186/aasfm.2011.3626}.  Available at
  \url{http://dx.doi.org/10.5186/aasfm.2011.3626}.

\bibitem[DLGT04]{J7}
\bgroup\scshape{}C.~De~Lellis\egroup{}, \bgroup\scshape{}C.~R.
  Grisanti\egroup{}, and \bgroup\scshape{}P.~Tilli\egroup{}, Regular selections
  for multiple-valued functions,  \emph{Ann. Mat. Pura Appl. (4)} \textbf{183}
  (2004), 79--95. \mr{2044333}.  \doi{10.1007/s10231-003-0081-5}.  Available at
  \url{http://dx.doi.org/10.1007/s10231-003-0081-5}.

\bibitem[DS15]{DLS13a}
\bgroup\scshape{}C.~{De Lellis}\egroup{} and
  \bgroup\scshape{}E.~{Spadaro}\egroup{}, {Multiple valued functions and
  integral currents},  \emph{Ann. Sc. Norm. Super. Pisa Cl. Sci. (5)}
  \textbf{XIV} (2015), 1239--1269.

\bibitem[DLS11]{DLS11a}
\bgroup\scshape{}C.~De~Lellis\egroup{} and \bgroup\scshape{}E.~N.
  Spadaro\egroup{}, {$Q$}-valued functions revisited,  \emph{Mem. Amer. Math.
  Soc.} \textbf{211} (2011), vi+79. \mr{2663735}.  \zbl{1246.49001}.
  \doi{10.1090/S0065-9266-10-00607-1}.

\bibitem[ENV16]{ENV_Reif}
\bgroup\scshape{}N.~Edelen\egroup{}, \bgroup\scshape{}A.~Naber\egroup{}, and
  \bgroup\scshape{}D.~Valtorta\egroup{}, Quantitative {R}eifenberg theorem for
  measures,  \emph{ArXiv:1612.08052} (2016). Available at
  \url{https://arxiv.org/abs/1612.08052}.

\bibitem[Gob09]{J11}
\bgroup\scshape{}J.~Goblet\egroup{}, Lipschitz extension of multiple
  {B}anach-valued functions in the sense of {A}lmgren,  \emph{Houston J. Math.}
  \textbf{35} (2009), 223--231. \mr{2491878}.

\bibitem[Hat02]{Hatcher02}
\bgroup\scshape{}A.~Hatcher\egroup{}, \emph{Algebraic topology}, Cambridge
  University Press, Cambridge, 2002. \mr{1867354}.

\bibitem[Hir16a]{Hir16b}
\bgroup\scshape{}J.~Hirsch\egroup{}, Boundary regularity of {D}irichlet
  minimizing {$Q$}-valued functions,  \emph{Ann. Sc. Norm. Super. Pisa Cl. Sci.
  (5)} \textbf{16} (2016), 1353--1407. \mr{3616337}.

\bibitem[Hir16b]{Hir16}
\bgroup\scshape{}J.~Hirsch\egroup{}, Partial {H}\"older continuity for
  {$Q$}-valued energy minimizing maps,  \emph{Comm. Partial Differential
  Equations} \textbf{41} (2016), 1347--1378. \mr{3551461}.
  \doi{10.1080/03605302.2016.1204313}.  Available at
  \url{http://dx.doi.org/10.1080/03605302.2016.1204313}.

\bibitem[LW08]{Lin_HM}
\bgroup\scshape{}F.~Lin\egroup{} and \bgroup\scshape{}C.~Wang\egroup{},
  \emph{The analysis of harmonic maps and their heat flows}, World Scientific
  Publishing Co. Pte. Ltd., Hackensack, NJ, 2008. \mr{2431658}.
  \doi{10.1142/9789812779533}.  Available at
  \url{http://dx.doi.org/10.1142/9789812779533}.

\bibitem[Mat83]{Mattila}
\bgroup\scshape{}P.~Mattila\egroup{}, Lower semicontinuity, existence and
  regularity theorems for elliptic variational integrals of multiple valued
  functions,  \emph{Trans. Amer. Math. Soc.} \textbf{280} (1983), 589--610.
  \mr{716839}.  \doi{10.2307/1999635}.  Available at
  \url{http://dx.doi.org/10.2307/1999635}.

\bibitem[Mos05]{Moser}
\bgroup\scshape{}R.~Moser\egroup{}, \emph{Partial regularity for harmonic maps
  and related problems}, World Scientific Publishing Co. Pte. Ltd., Hackensack,
  NJ, 2005. \mr{2155901}.  \doi{10.1142/9789812701312}.  Available at
  \url{http://dx.doi.org/10.1142/9789812701312}.

\bibitem[NV15]{naber-valtorta:varifold}
\bgroup\scshape{}A.~Naber\egroup{} and \bgroup\scshape{}D.~Valtorta\egroup{},
  The singular structure and regularity of stationary and minimizing varifolds,
   \emph{arXiv:1505.03428} (2015). Available at
  \url{https://arxiv.org/abs/1505.03428}.

\bibitem[NV16]{NV16}
\bgroup\scshape{}A.~Naber\egroup{} and \bgroup\scshape{}D.~Valtorta\egroup{},
  Stratification for the singular set of approximate harmonic maps,
  \emph{ArXiv:1611.03008} (2016). Available at
  \url{https://arxiv.org/abs/1611.03008}.

\bibitem[NV17]{naber-valtorta:harmonic}
\bgroup\scshape{}A.~Naber\egroup{} and \bgroup\scshape{}D.~Valtorta\egroup{},
  Rectifiable-{R}eifenberg and the regularity of stationary and minimizing
  harmonic maps,  \emph{Ann. of Math. (2)} \textbf{185} (2017), 131--227.
  \mr{3583353}.  \doi{10.4007/annals.2017.185.1.3}.  Available at
  \url{http://dx.doi.org/10.4007/annals.2017.185.1.3}.

\bibitem[Rei60]{reif_orig}
\bgroup\scshape{}E.~R. Reifenberg\egroup{}, Solution of the {P}lateau {P}roblem
  for {$m$}-dimensional surfaces of varying topological type,  \emph{Acta
  Math.} \textbf{104} (1960), 1--92. \mr{0114145}.  \zbl{0099.08503}.
  Available at \url{http://link.springer.com/article/10.1007%2FBF02547186}.

\bibitem[RS08]{RS}
\bgroup\scshape{}T.~Rivi{\`e}re\egroup{} and
  \bgroup\scshape{}M.~Struwe\egroup{}, Partial regularity for harmonic maps and
  related problems,  \emph{Comm. Pure Appl. Math.} \textbf{61} (2008),
  451--463. \mr{2383929}.  \doi{10.1002/cpa.20205}.  Available at
  \url{http://dx.doi.org/10.1002/cpa.20205}.

\bibitem[SU82]{SU}
\bgroup\scshape{}R.~Schoen\egroup{} and \bgroup\scshape{}K.~Uhlenbeck\egroup{},
  A regularity theory for harmonic maps,  \emph{J. Differential Geom.}
  \textbf{17} (1982), 307--335. \mr{664498}.  \zbl{0521.58021}.  Available at
  \url{http://projecteuclid.org/getRecord?id=euclid.jdg/1214436923}.

\bibitem[Sim96]{Simon96}
\bgroup\scshape{}L.~Simon\egroup{}, \emph{Theorems on regularity and
  singularity of energy minimizing maps}, \emph{Lectures in Mathematics ETH
  Z\"urich}, Birkh\"auser Verlag, Basel, 1996, Based on lecture notes by
  Norbert Hungerb\"uhler. \mr{1399562}.  \doi{10.1007/978-3-0348-9193-6}.
  Available at \url{http://dx.doi.org/10.1007/978-3-0348-9193-6}.

\bibitem[{Stu}17a]{SS17b}
\bgroup\scshape{}S.~{Stuvard}\egroup{}, {Multiple valued Jacobi fields},
  \emph{ArXiv:1701.08753} (2017). Available at
  \url{https://arxiv.org/abs/1701.08753}.

\bibitem[{Stu}17b]{SS17a}
\bgroup\scshape{}S.~{Stuvard}\egroup{}, {Multiple valued sections of vector
  bundles: the reparametrization theorem for $Q$-valued functions revisited},
  \emph{ArXiv:1705.00054} (2017). Available at
  \url{https://arxiv.org/abs/1705.00054}.

\bibitem[Tor95]{toro}
\bgroup\scshape{}T.~Toro\egroup{}, Geometric conditions and existence of
  bi-{L}ipschitz parameterizations,  \emph{Duke Math. J.} \textbf{77} (1995),
  193--227. \mr{1317632}.  \zbl{0847.42011}.
  \doi{10.1215/S0012-7094-95-07708-4}.

\bibitem[Zhu06]{J22}
\bgroup\scshape{}W.~Zhu\egroup{}, A theorem on frequency function for
  multiple-valued dirichlet minimizing functions,  \emph{ArXiv:math/0607576}
  (2006). Available at \url{https://arxiv.org/abs/math/0607576}.

\end{thebibliography}

\Addresses

\end{document}